\documentclass[12pt,a4paper]{amsart}
\usepackage{amssymb}

\calclayout

\newcommand{\+}{\protect\nobreakdash-}
\renewcommand{\.}{\mskip.5\thinmuskip}
\renewcommand{\:}{\colon}

\newcommand{\ot}{\otimes}
\newcommand{\rarrow}{\longrightarrow}
\newcommand{\larrow}{\longleftarrow}

\newcommand{\bu}{{\text{\smaller\smaller$\scriptstyle\bullet$}}}
\newcommand{\lrarrow}{\.\relbar\joinrel\relbar\joinrel\rightarrow\.}

\DeclareMathOperator{\Hom}{Hom}
\DeclareMathOperator{\Ext}{Ext}

\DeclareMathOperator{\Spec}{Spec}

\DeclareMathOperator{\pd}{pd}

\newcommand{\modl}{{\operatorname{\mathsf{--mod}}}}
\newcommand{\contra}{{\operatorname{\mathsf{--contra}}}}

\newcommand{\ctra}{{\operatorname{\mathsf{-ctra}}}}

\renewcommand{\b}{\mathsf b}
\newcommand{\qs}{{\mathsf{qs}}}

\newcommand{\sC}{\mathsf C}
\newcommand{\sD}{\mathsf D}
\newcommand{\sE}{\mathsf E}
\newcommand{\sF}{\mathsf F}

\let\SS\S
\renewcommand{\S}{\mathbf S}

\newcommand{\s}{\mathbf s}

\newcommand{\fR}{\mathfrak R}
\newcommand{\m}{\mathfrak m}
\newcommand{\n}{\mathfrak n}
\newcommand{\p}{\mathfrak p}
\newcommand{\q}{\mathfrak q}

\newcommand{\oR}{{\overline R}}
\newcommand{\oS}{{\overline S}}
\newcommand{\oT}{{\overline T}}

\newcommand{\boZ}{\mathbb Z}
\newcommand{\boQ}{\mathbb Q}
\newcommand{\boR}{\mathbb R}

\theoremstyle{plain}
\newtheorem{thm}{Theorem}[section]

\newtheorem{lem}[thm]{Lemma}
\newtheorem{prop}[thm]{Proposition}
\newtheorem{cor}[thm]{Corollary}
\newtheorem{oc}[thm]{Optimistic Conjecture}
\newtheorem{ml}[thm]{Main Lemma}

\theoremstyle{definition}

\newtheorem{rem}[thm]{Remark}

\newcommand{\Section}[1]{\bigskip\section{#1}\medskip}
\begin{document}

\title{On strongly flat and weakly cotorsion modules}

\author{Leonid Positselski and Alexander Sl\'avik}

\address{Department of Mathematics, Faculty of Natural Sciences,
University of Haifa, Mount Carmel, Haifa 31905, Israel; and
\newline\indent Laboratory of Algebraic Geometry, National Research
University Higher School of Economics, Moscow 119048; and
\newline\indent Sector of Algebra and Number Theory, Institute for
Information Transmission Problems, Moscow 127051, Russia; and
\newline\indent Charles University, Faculty of Mathematics and Physics,
Department of Algebra, Sokolovsk\'a~83, 186~75 Prague~8,
Czech Republic}

\email{positselski@yandex.ru}

\address{A.S.: Charles University, Faculty of Mathematics and Physics,
Department of Algebra, Sokolovsk\'a~83, 186~75 Prague~8,
Czech Republic}

\email{Slavik.Alexander@seznam.cz}

\begin{abstract}
 The aim of this paper is to describe the classes of strongly flat and
weakly cotorsion modules with respect to a multiplicative subset or
a finite collection of multiplicative subsets in a commutative ring.
 The strongly flat modules are characterized by a set of conditions,
while the weakly cotorsion modules are produced by a generation
procedure.
 Passing to the collection of all countable multiplicative subsets,
we define quite flat and almost cotorsion modules, and show that, over
a Noetherian ring with countable spectrum, all flat modules are quite
flat and all almost cotorsion modules are cotorsion.
\end{abstract}

\maketitle

\tableofcontents

\section{Introduction}
\medskip

 The aim of the theory developed in this paper is to describe flat 
modules over associative rings.
 Another related goal is to describe Enochs cotorsion modules (see
the discussion below in Section~\ref{historical-introd}).
 These goals, obviously important, are difficult to achieve in
a substantial enough yet fully general way.
 In this paper we provide such descriptions for Noetherian  commutative
rings with countable spectrum (see
Sections~\ref{finite-dimensional-introd}\+-\ref{quite-flat-introd}).
 We also describe certain subclasses of flat modules and overclasses
of Enochs cotorsion modules over commutative rings.

 It needs to be explained what is meant by ``describing flat modules''.
 One such description is provided by the Govorov--Lazard
theorem~\cite{Gov,Laz} claiming that, for any associative ring $R$,
the flat left $R$\+modules are precisely the filtered inductive limits
of finitely generated projective (or free) $R$\+modules.
 The following example illustrates the difference between
the Govorov--Lazard theorem and our approach.

 Let $R=\boZ$ be the ring of integers.
 Then the flat $R$\+modules are simply the torsion-free abelian groups.
 The Govorov--Lazard theorem tells that these are precisely the filtered
inductive limits of finitely generated free abelian groups.
 This assertion is obvious, though: any torsion-free abelian group is
the directed union of its finitely generated subgroups, which are free
abelian groups.
 \emph{Classifying} torsion-free abelian groups is a hopeless task;
still, one can \emph{describe} them in a way much more nontrivial than
the one provided by the Govorov--Lazard theorem.

 The description that we have in mind claims, in the case of the ring
$R=\boZ$, that \emph{any torsion-free abelian group $F$ is a direct
summand of an abelian group $G$ for which there exists a short exact
sequence $0\rarrow U\rarrow G\rarrow V\rarrow 0$, where $U$ is a free
abelian group and $V$ is a $\boQ$\+vector space}.
 It is this result, going back to Harrison~\cite[Section~2]{Har},
Matlis~\cite[Theorem~4.1]{Mat}, and Trlifaj~\cite[Proposition~2.8]{Trl0},
that we seek to extend to more complicated commutative rings $R$
in this paper.

 In a way of further motivation, let us say that flat modules over
commutative rings are obviously of interest in algebraic geometry;
and Enochs cotorsion modules over commutative rings have a geometric
significance, too, as the modules of cosections of locally cotorsion
contraherent cosheaves over affine schemes~\cite{Pcosh}
(cf.\ the discussion in the introduction to~\cite{Pcta}).

\subsection{{}} \label{historical-introd}
 Before proceeding to formulate the main results of this paper, let us
give a more detailed historical overview.

 The word ``cotorsion'' was introduced to homological algebra by
Harrison~\cite{Har}, who called ``co-torsion abelian groups''
the abelian groups $C$ such that $\Hom_\boZ(\boQ,C)=0=
\Ext_\boZ^1(\boQ,C)$.
 Subsequently Matlis~\cite{Mat} studied, under the name of cotorsion
modules over a commutative integral domain $R$, the $R$\+modules $C$
such that $\Hom_R(Q,C)=0=\Ext_R^1(Q,C)$, where $Q$ is the field
of fractions of~$R$. {\uchyph=0\par}

 The category of Harrison's co-torsion abelian groups is an abelian
subcategory in $\boZ\modl$.
 So is the category of Matlis' cotorsion $R$\+modules, provided that
the $R$\+module $Q$ has projective dimension~$1$ (see~\cite[p.~13]{Mat}
and~\cite[Theorem~3.4(a)]{PMat}).
 Commutative domains $R$ with the latter property came to be known as
\emph{Matlis domains}~\cite[Section~IV.4]{FSbook}.

 The modern terminology originates from Enochs' paper~\cite{En},
where an $R$\+module $C$ was called \emph{cotorsion} if
$\Ext^1_R(F,C)=0$ for all flat $R$\+modules~$F$.
 This definition, applicable to modules over an arbitrary
associative ring, is now central to a vast theory, including
two proofs~\cite{BBE} of Enochs' \emph{flat cover conjecture},
the notion of a \emph{cotorsion theory} (or \emph{cotorsion pair})
(going back to Salce's paper~\cite{Sal}), and numerous examples and
applications~\cite{GT}.
 In particular, the pair of full subcategories (flat modules, cotorsion
modules) in the category of left modules over an associative ring
is a thematic example of complete cotorsion theory.

 The following two definitions are due to
Trlifaj~\cite[Section~2]{Trl0}.
 Let $Q$ denote the maximal ring of quotients of an associative
ring~$R$.
 An $R$\+module $C$ is called \emph{weakly cotorsion} (or
\emph{Matlis cotorsion}) if $\Ext_R^1(Q,C)=0$.
 An $R$\+module $F$ is called \emph{strongly flat} if
$\Ext_R^1(F,C)=0$ for all weakly cotorsion $R$\+modules~$C$.

 When $R$ is a commutative domain (so $Q$ is its field of fractions),
every $R$\+module has a weakly cotorsion
envelope~\cite[Theorem~2.10(3)]{Trl0}, \cite[Example~3.2]{Trl}.
 The problem of characterizing domains $R$ for which every
$R$\+module has a strongly flat cover was posed in the lecture
notes~\cite[Section~5]{Trl}.

 This problem was solved in the series of papers by Bazzoni and
Salce~\cite{BS,BS2}, where it was shown that, for a commutative
domain $R$, all $R$\+modules have strongly flat covers if and only if 
all flat modules are strongly flat, and if and only if $R$ is
an \emph{almost perfect domain}.
 The latter condition means that, for every nonzero element $r\in R$,
the quotient ring $R/rR$ is \emph{perfect} (in the sense of
the classical paper of Bass~\cite{Bas}), that is, all flat modules
over $R/rR$ are projective.

 The following generalization of the results of Bazzoni--Salce was
obtained in a recent work of Fuchs and Salce~\cite{FS}.
 Let $R$ be a commutative ring and $Q$ be its classical ring of
fractions.
 Then all flat $R$\+modules are strongly flat if and only if $R$ is
an \emph{almost perfect ring}, which means that the quotient ring
$R/rR$ is a perfect ring for every nonzero-divisor $r\in R$ and
the ring of fractions $Q$ is a perfect ring.
 This result is the starting point for the present research.

\subsection{{}} \label{one-multsubset-theorems-introd}
 Let $R$ be a commutative ring and $S\subset R$ be a multiplicative
subset (which may contain some zero-divisors in~$R$).
 Let us say that an $R$\+module $C$ is \emph{$S$\+weakly cotorsion}
if $\Ext_R^1(S^{-1}R,C)=0$, where $S^{-1}R$ denotes the localization
of the ring $R$ at the multiplicative subset~$S$.
 An $R$\+module $F$ is said to be \emph{$S$\+strongly flat} if
$\Ext_R^1(F,C)=0$ for all $S$\+weakly cotorsion $R$\+modules~$C$.

 Equivalently, an $R$\+module $F$ is $S$\+strongly flat if and only if
it is a direct summand of an $R$\+module $G$ for which there exists
an exact sequence of $R$\+modules
\begin{equation} \label{strongly-flat-sequence}
 0\lrarrow U\lrarrow G\lrarrow V\lrarrow0,
\end{equation}
where $U$ is a free $R$\+module and $V$ is a free $S^{-1}R$\+module.
 This is a corollary of the general results about cotorsion theories
generated by a set of objects in the category of modules over
a ring~\cite[Corollary~6.13]{GT}.

 The first aim of this paper is to discuss the following

\begin{oc} \label{one-multsubset-conj}
 Let $R$ be a commutative ring and $S\subset R$ be a multiplicative
subset such that the projective dimension of the $R$\+module
$S^{-1}R$ does not exceed\/~$1$.
 Then a flat $R$\+module $F$ is $S$\+strongly flat if and only if
the following two conditions hold: the $R/sR$\+module $F/sF$ is
projective for every element $s\in S$, and the $S^{-1}R$\+module
$S^{-1}F$ is projective.
\end{oc}

 Here the notation is $S^{-1}F=S^{-1}R\ot_RF$.
 Notice that the ``only if'' assertion in
Optimistic Conjecture~\ref{one-multsubset-conj} follows
immediately from the description of strongly flat modules in terms
of the exact sequence~\eqref{strongly-flat-sequence}, as such
a sequence remains exact after applying $R/sR\ot_R{-}$, and
the $R/sR$\+module $V/sV$ vanishes for all $s\in S$, while
$S^{-1}V=V$ is a free $S^{-1}R$\+module.

 It should be pointed out that the condition on the projective
dimension of the $R$\+module $S^{-1}R$ in Optimistic
Conjecture~\ref{one-multsubset-conj} does indeed seem to be necessary.
 Indeed, let $H$ be the first syzygy module of the $R$\+module
$S^{-1}R$, i.~e., the leftmost term of a short exact sequence of
$R$\+modules $0\rarrow H\rarrow P\rarrow S^{-1}R\rarrow0$ with
a projective $R$\+module~$P$.
 Let us introduce the notation $\pd_RM$ for the projective dimension
of an $R$\+module~$M$.

\begin{prop} \label{ooc-counterex-prop}
 Let $R$ be a commutative ring and $S\subset R$ be a multiplicative
subset consisting of (some) nonzero-divisors in~$R$.
 Assume that\/ $\pd_RS^{-1}R\ge2$, and let $H$ be the first syzygy module
of the $R$\+module $S^{-1}R$.
 Then the $R$\+module $H$ is flat, the $R/sR$\+module $H/sH$ is
projective for every $s\in S$, and the $S^{-1}R$\+module $S^{-1}H$ is
projective, but the $R$\+module $H$ is \emph{not} $S$\+strongly flat.
\end{prop}

 Let us now formulate the positive results that we can prove.

\begin{thm} \label{one-countable-multsubset-thm}
 Let $S$ be a countable multiplicative subset in a commutative ring $R$.
 Then a flat $R$\+module $F$ is $S$\+strongly flat if and only if
the $R/sR$\+module $F/sF$ is projective for every $s\in S$ and
the $S^{-1}R$\+module $S^{-1}F$ is projective.
\end{thm}

\begin{thm} \label{one-regular-matlis-multsubset-thm}
 Let $R$ be a commutative ring and $S\subset R$ be a multiplicative
subset consisting of (some) nonzero-divisors in~$R$.
 Assume that the projective dimension of the $R$\+module $S^{-1}R$
does not exceed\/~$1$.
 Then a flat $R$\+module $F$ is $S$\+strongly flat if and only if
the $R/sR$\+module $F/sF$ is projective for every $s\in S$ and
the $S^{-1}R$\+module $S^{-1}F$ is projective.
\end{thm}

 Notice that one has $\pd_RS^{-1}R\le1$ for any countable multiplicative
subset $S\subset R$, but the converse is not true.
 For example, if $R$ is a Noetherian commutative ring of Krull
dimension~$1$, then the projective dimension of any flat $R$\+module
does not exceed~$1$ \cite[Corollaire~II.3.3.2]{RG}
(cf.~\cite[Corollary~13.7(a)]{Pcta}).

 Nevertheless, the following result is known~\cite{FS0,AHT}: if $R$ is
a commutative ring and $S\subset R$ is a multiplicative subset of
nonzero-divisors, then the projective dimension of the $R$\+module
$S^{-1}R$ does not exceed~$1$ if and only if the $R$\+module
$S^{-1}R/R$ is a direct sum of countably generated modules.
 Our Theorem~\ref{one-regular-matlis-multsubset-thm} is deduced from
this result together with the same computations with \emph{countable}
multiplicative subsets $S\subset R$ which we use to prove
Theorem~\ref{one-countable-multsubset-thm}.

 Moreover, the following slight generalization of
Theorem~\ref{one-regular-matlis-multsubset-thm} is provable using
the results of the papers~\cite{FS0,AHT}.
 Given a commutative ring $R$ and a multiplicative subset $S\subset R$,
we will say that the $S$\+torsion in $R$ is \emph{bounded} if there
exists an element $s_0\in S$ such that $sr=0$ for $s\in S$ and $r\in R$
implies $s_0r=0$.

\begin{thm} \label{one-bounded-torsion-matlis-multsubset-thm}
 Let $R$ be a commutative ring and $S\subset R$ be a multiplicative
subset such that the $S$\+torsion in $R$ is bounded.
 Assume that the projective dimension of the $R$\+module $S^{-1}R$
does not exceed\/~$1$.
 Then a flat $R$\+module $F$ is $S$\+strongly flat if and only if
the $R/sR$\+module $F/sF$ is projective for every $s\in S$ and
the $S^{-1}R$\+module $S^{-1}F$ is projective.
\end{thm}

 Notice that Theorem~\ref{one-countable-multsubset-thm} is \emph{not}
a particular case of
Theorem~\ref{one-bounded-torsion-matlis-multsubset-thm} as, for
a countable multiplicative subset $S$ in a commutative ring $R$,
the $S$\+torsion in $R$ does not need to be bounded.

 Finally, let us point out that yet another particular case
of Optimistic Conjecture~\ref{one-multsubset-conj} is shown to be
true in the paper~\cite{BP}.
 If $R$ is an \emph{$S$\+h-nil} commutative ring, that is, for every
element $s\in S$ the ring $R/sR$ is semilocal of Krull dimension~$0$,
then the projective dimension of the $R$\+module $S^{-1}R$ does not
exceed~$1$ \cite[Corollary~6.13]{BP}, and the assertion of our
Optimistic Conjecture~\ref{one-multsubset-conj} holds for
the multiplicative subset $S$ in the ring~$R$
\cite[Proposition~7.13]{BP}.

\subsection{{}} \label{one-multsubset-propositions-introd}
 Let us say a few words about our proofs of
Theorems~\ref{one-countable-multsubset-thm}\+-%
\ref{one-bounded-torsion-matlis-multsubset-thm}.
 Nothing (or almost nothing) is being done, in the course of these
proofs, with a flat $R$\+module $F$ satisfying the conditions of
these theorems.
 Instead, we work with an arbitrary $S$\+weakly cotorsion $R$\+module
$C$, proving that it can be obtained, using a certain set of rules
or operations, from $R$\+modules of simpler nature.

 Classically, the theory of cotorsion pairs~\cite{Sal,GT} is developed
as the theory of $\Ext^1$\+orthogonal classes of modules.
 Given any class of modules over an associative ring $R$, its left
$\Ext_R^1$\+orthogonal class is closed under transfinitely iterated
extesions (in the sense of the inductive limit)~\cite[Lemma~1]{ET}
and direct summands, while its right $\Ext_R^1$\+orthogonal class is
closed under transfinitely iterated extensions (in the sense of
the projective limit)~\cite[Proposition~18]{ET} and direct summands.
 In the important particular case of a \emph{hereditary} cotorsion
pair, the $\Ext^1$\+orthogonal class coincides with the
$\Ext^{\ge1}$\+orthogonal class.

 We change the point of view slightly and assign to a class of
$R$\+modules \emph{two} right orthogonal classes:
the $\Ext^{\ge1}$\+orthogonal class and the $\Ext^{\ge2}$\+orthogonal
class.
 In addition to the closedness with respect to transfinitely iterated
extensions (in the sense of the projective limit), the pair of classes
of modules so obtained has certain properties of closedness with
respect to kernels of surjections and cokernels of injections.
 Abstracting from these properties, we define the notion of
\emph{right\/ $1$\+obtainability} of an $R$\+module from a given
``seed class'' of $R$\+modules.

\begin{prop} \label{one-countable-multsubset-prop}
 Let $S$ be a countable multiplicative subset in a commutative ring~$R$.
 Then an $R$\+module is $S$\+weakly cotorsion if and only if it is
right\/ $1$\+obtainable from $R/sR$\+modules, $s\in S$, and
$S^{-1}R$\+modules.
\end{prop}

\begin{prop} \label{one-regular-matlis-multsubset-prop}
 Let $R$ be a commutative ring and $S\subset R$ be a multiplicative
subset consisting of (some) nonzero-divisors in~$R$.
 Assume that the projective dimension of the $R$\+module $S^{-1}R$
does not exceed\/~$1$.
 Then an $R$\+module is $S$\+weakly cotorsion if and only if it is
right\/ $1$\+obtainable from $R/sR$\+modules, $s\in S$, and
$S^{-1}R$\+modules.
\end{prop}

\begin{prop} \label{one-bounded-torsion-matlis-multsubset-prop}
 Let $R$ be a commutative ring and $S\subset R$ be a multiplicative
subset such that the $S$\+torsion in $R$ is bounded.
 Assume that the projective dimension of the $R$\+module $S^{-1}R$
does not exceed\/~$1$.
 Then an $R$\+module is $S$\+weakly cotorsion if and only if it is
right\/ $1$\+obtainable from $R/sR$\+modules, $s\in S$, and
$S^{-1}R$\+modules.
\end{prop}

 In this paper, Theorems~\ref{one-countable-multsubset-thm}\+-%
\ref{one-bounded-torsion-matlis-multsubset-thm} are (easily)
deduced from Propositions~\ref{one-countable-multsubset-prop}\+-%
\ref{one-bounded-torsion-matlis-multsubset-prop}, while the proofs
of the latter require some substantial work.

\subsection{{}} \label{several-multsubsets-introd}
 Let us consider the following generalization of the setting
of Sections~\ref{one-multsubset-theorems-introd}\+-%
\ref{one-multsubset-propositions-introd}.
 Let $S_1$,~\dots, $S_m\subset R$ be a finite collection of
multiplicative subsets in $R$.
 We will denote the collection $\{S_1,\dotsc,S_m\}$ by the single
letter $\S$ for brevity.
 Let us say that an $R$\+module $C$ is \emph{$\S$\+weakly cotorsion}
if $\Ext^1_R(S_j^{-1}R,C)=0$ for all $j=1$,~\dots,~$m$.
 An $R$\+module $F$ is said to be \emph{$\S$\+strongly flat} if
$\Ext_R^1(F,C)=0$ for all $\S$\+weakly cotorsion $R$\+modules~$C$.
 Equivalently, an $R$\+module $F$ is $\S$\+strongly flat if and only if
it is a direct summand of a transfinitely iterated extension, in
the sense of the inductive limit, of $R$\+modules isomorphic to $R$
or $S_j^{-1}R$, \ $1\le j\le m$ \cite[Corollary~6.14]{GT}.

 Given two multiplicative subsets $S$ and $T\subset R$, we denote
by $ST\subset R$ the multiplicative subset generated by $S$ and
$T$ in $R$; and similarly, given a finite collection of multiplicative
subsets $\{S_j\}$ in $R$, we denote by $\prod_jS_j\subset R$
the multiplicative subset generated by them.
 For any subset of indices $J\subset\{1,\dotsc,m\}$, denote by
$S_J\subset R$ the multiplicative subset $\prod_{j\in J} S_j\subset R$.
 Given a finite collection of multiplicative subsets
$\S=\{S_j\mid 1\le j\le m\}$ in a commutative ring $R$, we denote by
$\S^\times$ the finite collection of multiplicative subsets
$\{S_J\mid J\subset\{1,\dotsc,m\}\.\}$.

 Let $J\subset\{1,\dotsc,m\}$ be a subset of indices; denote by $K$
the complementary subset $K=\{1,\dotsc,m\}\setminus J$.
 Given a collection of multiplicative subsets $\S=\{S_1,\dotsc,S_m\}$
in a commutative ring $R$ and a subset of indices $K\subset
\{1,\dotsc,m\}$, denote by a single letter~$\s$ a collection of
elements $(s_k\in S_k)_{k\in K}$.
 Let $R_{J,\s}$ denote the quotient ring of the ring $S_J^{-1}R$ by
the ideal generated by all the elements $s_k\in S_J^{-1}R$, \,$k\in K$.
 So the ring $R_{J,\s}$ is the result of inverting all the multiplicative
subsets $S_j\subset R$, \,$j\in J$, and annihilating all the elements
$s_k\in R$, \,$k\in K$, in the ring~$R$.

 The following formulation is the analogue of
Optimistic Conjecture~\ref{one-multsubset-conj} for several
multiplicative subsets.

\begin{oc}
 Let $R$ be a commutative ring and $S_1$,~\dots, $S_m\subset R$ be
a finite collection of multiplicative subsets in~$R$.
 Assume that, for every subset of indices $J\subset\{1,\dotsc,m\}$,
the projective dimension of the $R$\+module $S_J^{-1}R$ does not
exceed\/~$1$.
 Then a flat $R$\+module $F$ is\/ $\S^\times$\+strongly flat if and only
if the $R_{J,\s}$\+module $R_{J,\s}\ot_RF$ is projective for every subset
of indices $J\subset\{1,\dotsc,m\}$ and any choice of elements
$s_k\in S_k$, \ $k\in\{1,\dotsc,m\}\setminus J$.
\end{oc}

 Here is the theorem that we can actually prove.

\begin{thm} \label{several-multsubsets-thm}
 Let $S_1$,~\dots, $S_m\subset R$ be a finite collection of
multiplicative subsets in a commutative ring~$R$.
 Assume that, for every subset of indices $J\subset\{1,\dotsc,m\}$,
the projective dimension of the $R$\+module $S_J^{-1}R$ does not
exceed\/~$1$.
 Furthermore, assume that for every $j=1$,~\dots,~$m$, one of the two
possibilities is realized: either $S_j$ is countable, or
the $S_j$\+torsion in $R$ is bounded.
 Then a flat $R$\+module $F$ is\/ $\S^\times$\+strongly flat if and
only if the $R_{J,\s}$\+module $R_{J,\s}\ot_RF$ is projective for every
subset of indices $J\subset\{1,\dotsc,m\}$ and any choice of
elements $s_k\in S_k$, \ $k\in\{1,\dotsc,m\}\setminus J$.
\end{thm}

 We deduce Theorem~\ref{several-multsubsets-thm} from the following
proposition.

\begin{prop} \label{several-multsubsets-prop}
 Let $S_1$,~\dots, $S_m\subset R$ be a finite collection of
multiplicative subsets in a commutative ring~$R$.
 Assume that, for every $J\subset\{1,\dotsc,m\}$, the projective
dimension of the $R$\+module $S_J^{-1}R$ does not exceed\/~$1$.
 Furthermore, assume that, for every\/ $1\le j\le m$, either $S_j$
is countable, or the $S_j$\+torsion in $R$ is bounded.
 Then an $R$\+module is\/ $\S^\times$\+weakly cotorsion if and only if
it is right\/ $1$\+obtainable from the class of all $R_{J,\s}$\+modules
(viewed as $R$\+modules via the restriction of scalars), where
$J$ runs over all the subsets of indices $J\subset\{1,\dotsc,m\}$
and\/ $\s$~runs over all the choices of elements $s_k\in S_k$, \
$k\in\{1,\dotsc,m\}\setminus J$.
\end{prop}

 Of course, the product $ST$ of any two countable multiplicative
subsets $S$, $T\subset R$ is countable; and one easily observes
that, for any two multiplicative subsets $S$, $T\subset R$,
the $ST$\+torsion in $R$ bounded whenever the $S$\+torsion is
bounded and $T$\+torsion is bounded.
 But the product of two multiplicative subsets, one of which is
countable and the other has bounded torsion, may satisfy neither
condition.

 On the other hand, if all the multiplicative subsets $S_1$,~\dots,
$S_m\subset R$ are countable, then the condition on the projective
dimension of the $R$\+modules $S_J^{-1}R$ in
Theorem~\ref{several-multsubsets-thm} and
Proposition~\ref{several-multsubsets-prop} is satisfied automatically.

\subsection{{}} \label{finite-dimensional-introd}
 Now let us formulate our descriptions of flat and Enochs cotorsion
$R$\+modules.
 The following two results were obtained in
the paper~\cite[Section~13]{Pcta} (see~\cite[Corollary~13.11 and
Theorem~13.9(b)]{Pcta}).

\begin{thm}
 Let $R$ be a Noetherian commutative ring with finite spectrum.
 Then there exists an element $s\in R$ such that, denoting by $S$
the multiplicative subset $\{1,s,s^2,s^3,\dotsc\}\subset R$ generated
by~$s$, all flat $R$\+modules are $S$\+strongly flat.
\end{thm}

\begin{thm}
 Let $R$ be a Noetherian commutative ring of Krull dimension\/~$1$.
 Then there exists a multiplicative subset $S\subset R$ such that
all flat $R$\+modules are $S$\+strongly flat.
\end{thm}

 In the rest of this introduction we discuss generalizations of
these theorems to more complicated Noetherian commutative rings.

 The following simple example is the starting point.
 Let $P$ be a principal ideal domain (PID).
 The \emph{content} of a polynomial $p(x)=p_nx^n+p_{n-1}x^{n-1}+
\dotsb+p_0\in P[x]$ is defined as the greatest common divisor of its
coefficients $p_n$,~\dots,~$p_0$.
 So the content of a polynomial in one variable~$x$ over the ring~$P$
is an element of $P$ defined up to a multiplication by an invertible
element of~$P$.
 By Gauss Lemma, the content of the product of two polynomials is
equal to the product of their contents.

 Denote by $S_1\subset P[x]$ the multiplicative subset of all nonzero
elements of $P$, viewed as elements of $P[x]$, and by $S_2\subset P[x]$
the multiplicative subset of all polynomials with content~$1$.
 Set $\S=\{S_1,S_2\}$ to be the collection of two multiplicative subsets
$S_1$ and $S_2$ in the ring $R=P[x]$.

\begin{thm} \label{polynomials-over-pid-thm}
 For any countable (commutative) principal ideal domain $P$, all flat
$P[x]$\+modules are\/ $\S^\times$\+strongly flat for the above
collection of two multiplicative subsets\/ $\S=\{S_1,S_2\}$ in $R=P[x]$.
\end{thm}

 The result of Theorem~\ref{polynomials-over-pid-thm}
generalizes to Noetherian rings of Krull dimension~$2$ as follows.

\begin{thm} \label{two-dimensional-ring-thm}
 Let $R$ be a Noetherian commutative ring of Krull dimension\/~$2$ with
countable spectrum.
 Then there exists a pair of countable multiplicative subsets $S_1$
and $S_2\subset R$ such that all flat $R$\+modules are\/
$\S^\times$\+strongly flat for\/ $\S=\{S_1,S_2\}$.
\end{thm}

 In other words, the assertions of
Theorems~\ref{polynomials-over-pid-thm}
and~\ref{two-dimensional-ring-thm} mean that, in their
respective assumptions, every flat $R$\+module is a direct summand
of a transfinitely iterated extension, in the sense of the inductive
limit, of $R$\+modules isomorphic to $R$, \ $S_1^{-1}R$, \ $S_2^{-1}R$,
or~$(S_1S_2)^{-1}R$.

 Furthermore, the construction of Theorem~\ref{two-dimensional-ring-thm}
generalizes to higher Krull dimensions~$d$ in the following way.
 Define a function $\mu\:\boZ_{\ge0}\rarrow\boZ_{\ge0}$ by the rule
$$
 \mu(d)=d+(d-2)+(d-4)+\dotsb=\sum\nolimits_{i\in\boZ}^{0\le 2i\le d} d-2i ,
$$
or equivalently, $\mu(d)$ is the closest integer to $(d+1)^2/4$.
 So we have $\mu(0)=0$, \ $\mu(1)=1$, \ $\mu(2)=2$, \ $\mu(3)=4$,
\ $\mu(4)=6$, etc.

\begin{thm} \label{finite-dimensional-ring-thm}
 Let $R$ be a Noetherian commutative ring of finite Krull dimension~$d$
with countable spectrum.
 Then there exists a collection of $m=\mu(d)$ countable multiplicative
subsets $S_1$,~\dots, $S_m\subset R$ such that all flat $R$\+modules
are\/ $\S^\times$\+strongly flat for\/ $\S=\{S_1,\dotsc,S_m\}$.
\end{thm}

 In all the three Theorems~\ref{polynomials-over-pid-thm},
\ref{two-dimensional-ring-thm}, and~\ref{finite-dimensional-ring-thm},
the multiplicative subsets $S_1$,~\dots, $S_m\subset R$ are constructed
in such a way that all the rings $R_{J,\s}$ (in the notation of
Section~\ref{several-multsubsets-introd}) are Artinian.
 Then all flat $R_{J,\s}$\+modules are projective, and the assertion
that all flat $R$\+modules are $\S^\times$\+strongly flat follows from
Theorem~\ref{several-multsubsets-thm}.

\subsection{{}} \label{quite-flat-introd}
 The following approach allows to generalize
Theorem~\ref{finite-dimensional-ring-thm} even further by getting rid
of the finite Krull dimension assumption on the ring~$R$.
 Besides, it allows to obtain some results applicable to Noetherian
rings with uncountable spectrum.

 Let $R$ be a commutative ring.
 Let us say that an $R$\+module $C$ is \emph{almost cotorsion} if
$\Ext^1_R(S^{-1}R,C)=0$ for all (at most) countable multiplicative
subsets $S\subset R$.
 We say that an $R$\+module $F$ is \emph{quite flat} if
$\Ext_R^1(F,C)=0$ for all almost cotorsion $R$\+mod\-ules~$C$.
 Equivalently, an $R$\+module $F$ is quite flat if and only if it is
a direct summand of a transfinitely iterated extension, in the sense
of the inductive limit, of $R$\+modules isomorphic to $S^{-1}R$,
where $S\subset R$ are countable multiplicative subsets.

\begin{thm} \label{countable-spectrum-thm}
 Let $R$ be a Noetherian commutative ring with countable spectrum.
 Then all flat $R$\+modules are quite flat.
\end{thm}

 Notice that, in an (infinite) Noetherian ring, the cardinality of
the set of all ideals does not exceed the cardinality of the ring.
 Thus any countable Noetherian ring has at most countable spectrum
(while the converse, of course, does not hold).

 Theorem~\ref{countable-spectrum-thm} can be deduced from
the following main lemma (which does not assume countability
of the spectrum).

\begin{ml} \label{noetherian-quite-flat-main-lemma}
 Let $R$ be a Noetherian commutative ring and $S\subset R$ be
a countable multiplicative subset.
 Then a flat $R$\+module $F$ is quite flat if and only if
the $R/sR$\+module $F/sF$ is quite flat for all $s\in S$ and
the $S^{-1}R$\+module $S^{-1}F$ is quite flat.
\end{ml}

 Alternatively, Theorem~\ref{countable-spectrum-thm} can be deduced
from the following proposition.

\begin{prop} \label{countable-spectrum-prop}
 Let $R$ be a Noetherian commutative ring with countable spectrum.
 Then an $R$\+module is almost cotorsion if and only if it is right\/
$1$\+obtainable from vector spaces over the residue fields of prime
ideals in~$R$ (viewed as $R$\+modules via the restriction of scalars).
\end{prop}

 One observes that all vector spaces over fields are (Enochs) cotorsion,
all restrictions of scalars take cotorsion modules to cotorsion modules,
and all modules right $1$\+obtainable from cotorsion modules are
cotorsion.
 Thus Proposition~\ref{countable-spectrum-prop} implies (or, if one
wishes, is equivalent to the combination of) the following
two corollaries.

\begin{cor} \label{countable-spectrum-almost-cotorsion-are-cotorsion}
 Let $R$ be a Noetherian commutative ring with countable spectrum.
 Then all almost cotorsion $R$\+modules are cotorsion.
\end{cor}

\begin{cor} \label{countable-spectrum-cotorsion-obtainable-cor}
 Let $R$ be a Noetherian commutative ring with countable spectrum.
 Then an $R$\+module is (Enochs) cotorsion if and only if it is right\/
$1$\+obtainable from vector spaces over the residue fields of prime
ideals in~$R$ (viewed as $R$\+modules via the restriction of scalars).
\end{cor}

 Corollary~\ref{countable-spectrum-almost-cotorsion-are-cotorsion} is,
of course, just a restatement of
Theorem~\ref{countable-spectrum-thm}, while
Corollary~\ref{countable-spectrum-cotorsion-obtainable-cor}
provides a description of cotorsion modules over a Noetherian
commutative ring with countable spectrum.
 Notice that a description of flat cotorsion modules over a Noetherian
commutative ring, obtained by Enochs in~\cite{En}, plays in important
role in the locally cotorsion contraherent cosheaf
theory~\cite[Section~5.1]{Pcosh}.
 A description of cotorsion modules over a Noetherian commutative ring
of Krull dimension~$1$ was obtained in~\cite[Corollary~13.12]{Pcta}.

 Both Main Lemma~\ref{noetherian-quite-flat-main-lemma}
and Proposition~\ref{countable-spectrum-prop} can be deduced from
the following stronger version of the main lemma (which, once again,
does not assume countability of the spectrum).

\begin{ml} \label{noetherian-almost-cotorsion-main-lemma}
 Let $R$ be a Noetherian commutative ring and $S\subset R$ be
a countable multiplicative subset.
 Then an $R$\+module is almost cotorsion if and only if it is
right $1$\+obtainable from almost cotorsion $R/sR$\+modules, $s\in S$,
and almost cotorsion $S^{-1}R$\+modules.
\end{ml}

\subsection{{}}
 Finally, we should mention that both
Main Lemmas~\ref{noetherian-quite-flat-main-lemma}
and~\ref{noetherian-almost-cotorsion-main-lemma} are generalizable
to some situations involving non-Noetherian commutative rings.
 Notice that, for any multiplicative subset $S$ in a Noetherian
commutative ring $R$, the $S$\+torsion in $R$ is necessarily bounded.

\begin{ml} \label{bounded-torsion-quite-flat-main-lemma}
 Let $R$ be a commutative ring and $S\subset R$ be a countable
multiplicative subset such that the $S$\+torsion in $R$ is bounded.
 Then a flat $R$\+module $F$ is quite flat if and only if
the $R/sR$\+module $F/sF$ is quite flat for all $s\in S$ and
the $S^{-1}R$\+module $S^{-1}F$ is quite flat.
\end{ml}

\begin{ml} \label{bounded-torsion-almost-cotorsion-main-lemma}
 Let $R$ be a commutative ring and $S\subset R$ be a countable
multiplicative subset such that the $S$\+torsion in $R$ is bounded.
 Then an $R$\+module is almost cotorsion if and only if it is right\/
$1$\+obtainable from almost cotorsion $R/sR$\+modules, $s\in S$, and
almost cotorsion $S^{-1}R$\+modules.
\end{ml}

 Here, as with the Noetherian versions of the main lemmas (and as with
the resuls of this paper generally), the description of (quite) flat
modules follows easily from the description of (almost) cotorsion ones,
i.~e., Main Lemma~\ref{bounded-torsion-quite-flat-main-lemma} is deduced
from Main Lemma~\ref{bounded-torsion-almost-cotorsion-main-lemma}.

\subsection{{}}
 This paper is a sequel to the paper~\cite{PS}, where largely the same
techniques are applied to a number of similar problems.
 The main difference is that the exposition in~\cite{PS} is restricted
to multiplicative subsets $S\subset R$ generated by a single element,
$S=\{1,r,r^2,r^3,\dotsc\}$, where $r\in R$ (as these are important for
the algebro-geometric applications to which \cite{PS}~aims).
 Thus many results of this paper are generalizations of the respective
results of~\cite{PS}.
 In particular, the technique of ``obtainable modules'' first appeared
in~\cite[Sections~3 and~6]{PS}.

 The reader may wish to glance into~\cite[Section~1]{PS} for an overview
of the results and an outline of the arguments in~\cite{PS} relevant in
the context of the present paper (the requisite definitions can be found
in~\cite[Section~0.5]{PS}).

 In particular, \cite[Toy Main Lemma~1.8]{PS} is notable as the simplest
version of Theorem~\ref{one-countable-multsubset-thm}
(while~\cite[Lemma~1.10]{PS} is an even simpler result whose proof is
dual in some sense).
 \cite[Toy Main Proposition~4.8]{PS} is a particular case
of Proposition~\ref{one-countable-multsubset-prop}, while
\cite[Theorem~1.9]{PS} is a particular case of
Theorem~\ref{several-multsubsets-thm}
and~\cite[Main Proposition~8.1]{PS} is a particular case
of Proposition~\ref{several-multsubsets-prop}.

 \cite[Main Lemma~1.4]{PS} is somewhat similar to (but simpler than)
Main Lemma~\ref{noetherian-quite-flat-main-lemma}, while
\cite[Main Proposition~7.1]{PS} is somewhat similar to (but simpler
than) Main Lemma~\ref{noetherian-almost-cotorsion-main-lemma}.
 Finally, \cite[Main Lemma~1.7]{PS} is somewhat similar to (but simpler
than) Main Lemma~\ref{bounded-torsion-quite-flat-main-lemma}, while
\cite[Main Proposition~7.2]{PS} is somewhat similar to (but simpler
than) Main Lemma~\ref{bounded-torsion-almost-cotorsion-main-lemma}.

\subsection{{}}
 The proofs of Proposition~\ref{ooc-counterex-prop},
Theorems~\ref{one-countable-multsubset-thm}\+-%
\ref{one-bounded-torsion-matlis-multsubset-thm}, and
Propositions~\ref{one-countable-multsubset-prop}\+-%
\ref{one-bounded-torsion-matlis-multsubset-prop} are presented
in Section~\ref{one-multsubset-secn}.
 The proofs of Theorem~\ref{several-multsubsets-thm} and
Proposition~\ref{several-multsubsets-prop} are given in
Section~\ref{several-multsubsets-secn}.
 Theorems~\ref{polynomials-over-pid-thm}\+-%
\ref{finite-dimensional-ring-thm} are proved in
Section~\ref{finite-dimensional-secn}.
 The proofs of Theorem~\ref{countable-spectrum-thm},
Proposition~\ref{countable-spectrum-prop},
Main Lemmas~\ref{noetherian-quite-flat-main-lemma}
and~\ref{noetherian-almost-cotorsion-main-lemma}\+-%
\ref{bounded-torsion-almost-cotorsion-main-lemma}, and
Corollaries~\ref{countable-spectrum-almost-cotorsion-are-cotorsion}\+-%
\ref{countable-spectrum-cotorsion-obtainable-cor} are contained
in the last Section~\ref{quite-flat-almost-cotorsion-secn}.

\medskip
\textbf{Acknowledgement.}
 The first author is grateful to Silvana Bazzoni for very helpful
discussions and communications.
 This paper grew out of the first author's visits to Padova and Prague,
and he wishes to thank Silvana Bazzoni and Jan Trlifaj for
the kind invitations.
 We also wish to thank Jan Trlifaj for very helpful discussions.
 The first author's research is supported by the Israel Science
Foundation grant~\#\,446/15 and by the Grant Agency of
the Czech Republic under the grant P201/12/G028.
 The second author's research is supported by the Grant Agency of
the Czech Republic under the grant 17-23112S and by the SVV project
under the grant SVV-2017-260456.

\Section{Obtainable Modules} \label{obtainable-secn}

 In this section, $R$ is an associative ring and $R\modl$ is
the abelian category of left $R$\+modules.
 The exposition below largely follows~\cite[Sections~3 and~6]{PS}.

 Given a class of objects (full subcategory) $\sF\subset R\modl$,
we denote by $\sF^{\perp_1}\subset R\modl$ the class of all objects
$C\in R\modl$ such that $\Ext_R^1(F,C)=0$ for all $F\in\sF$.
 Similarly, for any integer $n\ge0$, we denote by $\sF^{\perp_{\ge n}}
\subset R\modl$ the class of all objects $C\in R\modl$ such that
$\Ext_R^i(F,C)=0$ for all $F\in\sF$ and $i\ge n$.

 Let us also introduce notation for the dual operations on classes
of modules.
 Given a class of objects $\sE\subset R\modl$, we denote by
${}^{\perp_1}\.\sE\subset R\modl$ the class of all objects $F\in R\modl$
such that $\Ext_R^1(F,E)=0$ for all $E\in\sE$.
 For any integer $n\ge0$, we denote by ${}^{\perp_{\ge n}}\.\sE\subset
R\modl$ the class of all objects $F\in R\modl$ such that
$\Ext_R^i(F,E)=0$ for all $E\in\sE$ and all $i\ge n$.

 Now let us define transfinitely iterated extensions.
 Let $F$ be a left $R$\+module and $\gamma$~be an ordinal.
 Suppose that for every ordinal $\alpha\le\gamma$ we are given
an $R$\+submodule $F_\alpha\subset F$ such that the following conditions
are satisfied:
\begin{itemize}
\item $F_0=0$ and $F_\gamma=F$;
\item one has $F_\alpha\subset F_\beta$ for all $\alpha\le\beta\le\gamma$;
\item and one has $F_\beta=\bigcup_{\alpha<\beta}F_\alpha$ for all limit
ordinals $\beta\le\gamma$.
\end{itemize}
 Then one says that the $R$\+module $F$ is a \emph{transfinitely
iterated extension} (\emph{in the sense of the inductive limit}) of
the $R$\+modules $F_{\alpha+1}/F_\alpha$, where $0\le\alpha<\gamma$.

 Here is the dual definition.
 Let $G$ be a left $R$\+module and $\delta$~be an ordinal.
 Suppose that for every ordinal $\alpha\le\delta$ we are given a left
$R$\+module $G_\alpha$ and for every pair of ordinals $\alpha<\beta
\le\delta$ there is an $R$\+module morphism $G_\beta\rarrow G_\alpha$
such that the following conditions hold:
\begin{itemize}
\item $G_0=0$ and $G_\delta=G$;
\item the triangle diagrams $G_\gamma\rarrow G_\beta\rarrow G_\alpha$ are
commutative for all triples of ordinals $\alpha<\beta<\gamma
\le\delta$;
\item the induced morphism into the projective limit $G_\beta\rarrow
\varprojlim_{\alpha<\beta}G_\alpha$ is an isomorphism for all limit
ordinals $\beta\le\delta$;
\item the morphism $G_{\alpha+1}\rarrow G_\alpha$ is surjective for all
$\alpha<\delta$.
\end{itemize}
 Denote by $H_\alpha$ the kernel of the morphism $G_{\alpha+1}\rarrow
G_\alpha$.
 Then we will say that the left $R$\+module $G$ is a \emph{transfinitely
iterated exension} (\emph{in the sense of the projective limit}) of
the left $R$\+modules $H_\alpha$, where $0\le\alpha<\delta$.

\begin{lem} \label{eklof-lemma}
\textup{(a)} Let\/ $\sE\subset R\modl$ be a class of left $R$\+modules.
 Then the class of left $R$\+modules\/ $\sF={}^{\perp_1}\.\sE$ is closed
under transfinitely iterated extensions, in the sense of the inductive
limit, and direct summands. \par
\textup{(b)} Let\/ $\sF\subset R\modl$ be a class of left $R$\+modules.
 Then the class of left $R$\+modules\/ $\sC=\sF^{\perp_1}$ is closed
under transfinitely iterated extensions, in the sense of the projective
limit, and direct summands.
\end{lem}

\begin{proof}
 Closedness under direct summands is obvious.
 The assertions about closedness with respect to transfinitely iterated
extensions are known as the Eklof Lemma.
 Part~(a) is the classical version of
the Eklof Lemma~\cite[Lemma~1]{ET}, while part~(b) is
the dual version~\cite[Proposition~18]{ET}.
\end{proof}

\begin{lem} \label{positive-ext-orthogonal-closedness-properties}
\textup{(a)} Let\/ $\sE\subset R\modl$ be a class of left $R$\+modules
and $n\ge1$ be an integer.
 Then the class of left $R$\+modules\/ $\sF={}^{\perp_{\ge n}}\.\sE$ is
closed under direct summands, extensions, kernels of surjective
morphisms, infinite direct sums, and transfinitely iterated extensions
in the sense of the inductive limit. \par
\textup{(b)} Let\/ $\sF\subset R\modl$ be a class of left $R$\+modules
and $n\ge1$ be an integer.
 Then the class of left $R$\+modules\/ $\sC=\sF^{\perp_{\ge n}}$ is closed
under direct summands, extensions, cokernels of injective morphisms,
infinite products, and transfinitely iterated extensions in
the sense of the projective limit.
\end{lem}

\begin{proof}
 Notice that all the operations listed in part~(a), with the exception
of direct summands and kernels of surjective morphisms, are
particular cases of transfinitely iterated extensions in the sense of
the inductive limit, and similarly, all the operations listed in
part~(b), with the exception of direct summands and cokernels of
injective morphisms, are particular cases of transfinitely iterated
extensions in the sense of the projective limit.
 The closedness properties of the classes of modules in
Lemma~\ref{positive-ext-orthogonal-closedness-properties} with respect
to transfinitely iterated extensions follow from the similar properties
of the classes of modules in Lemma~\ref{eklof-lemma}, while closedness
with respect to kernels of surjections or cokernels of injections is
easily checked using the long exact sequence of~$\Ext_R^*$.
 We refer to~\cite[proofs of Lemmas~3.2 and~6.1]{PS} for the details.
\end{proof}

 Let $\sF\subset R\modl$ be a class of left $R$\+modules.
 For every $n\ge1$, consider the class of left $R$\+modules
$\sC_n=\sF^{\perp_{\ge n}}$.

\begin{lem} \label{classes-c-n-closedness-properties}
 For any class of objects\/ $\sF\subset R\modl$, the classes of
objects\/ $\sC_n=\sF^{\perp_{\ge n}}$ have the following properties:
\begin{enumerate}
\renewcommand{\theenumi}{\roman{enumi}}
\item for any $n\ge1$, one has\/ $\sC_n\subset\sC_{n+1}$;
\item for any $n\ge1$, the class of objects\/ $\sC_n\subset R\modl$
is closed under direct summands, extensions, cokernels of injective
morphisms, infinite products, and transfinitely iterated extensions
in the sense of the projective limit;
\item the kernel of any surjective morphism from an object of\/
$\sC_{n+1}$ to an object of\/ $\sC_n$ belongs to\/~$\sC_{n+1}$;
\item the cokernel of any injective morphism from an object of\/
$\sC_{n+1}$ to an object of\/ $\sC_n$ belongs to\/~$\sC_n$.
\end{enumerate}
\end{lem}

\begin{proof}
 This is~\cite[Lemma~6.1]{PS}.
 The property~(i) holds by the definition, (ii)~follows from
Lemma~\ref{positive-ext-orthogonal-closedness-properties}(b),
and~(iii\+iv) are easily provable using the long exact sequence
of~$\Ext_R^*$.
\end{proof}

 Let $\sE\subset R\modl$ be a fixed class of objects.
 Set $\sF={}^{\perp_{\ge1}}\.\sE$ and $\sC=\sF^{\perp_{\ge n}}$ for $n=1$
and~$2$.
 Our aim is to formulate a technique that would potentially allow
to describe the class $\sC_1=({}^{\perp_{\ge1}}\.\sE)^{\perp_{\ge1}}$,
at least in some special situations. 
 The following two definitions can be found in~\cite[Definition~3.3
and~6.2]{PS}.

 We will say that an object $C\in R\modl$ is \emph{simply right
obtainable} from a class $\sE\subset R\modl$ if $C$ belongs to
the (obviously, unique) minimal class of objects in $R\modl$ containing
$\sE$ and closed under direct summands, extensions, cokernels of
injective morphisms, infinite products, and transfinitely iterated
extensions in the sense of the projective limit.

 Furthermore, the pair of classes of objects \emph{right
$1$\+obtainable} from $\sE$ and \emph{right $2$\+obtainable} from
$\sE$ is defined as the (obviously, unique) minimal pair of classes
of objects in $R\modl$ satisfying the following generation rules:
\begin{enumerate}
\renewcommand{\theenumi}{\roman{enumi}}
\item all the objects of $\sE$ are right $1$\+obtainable; all
the right $1$\+obtainable objects are right $2$\+obtainable;
\item all the objects simply right obtainable from right
$n$\+obtainable objects are right $n$\+obtainable (for $n=1$ or~$2$);
\item the kernel of any surjective morphism from a right $2$\+obtainable
object to a right $1$\+obtainable object is right $2$\+obtainable;
\item the cokernel of any injective morphism from a right
$2$\+obtainable object to a right $1$\+obtainable object is right
$1$\+obtainable.
\end{enumerate}

\begin{lem} \label{obtainable-orthogonal-lemma}
 For any class of objects\/ $\sE\subset R\modl$, all the objects
right\/ $1$\+obtainable from\/ $\sE$ belong to the class\/
$\sC_1=({}^{\perp_{\ge1}}\.\sE)^{\perp_{\ge1}}\subset R\modl$.
 All the objects right\/ $2$\+obtain\-able from\/ $\sE$ belong to
the class\/ $\sC_2=({}^{\perp_{\ge1}}\.\sE)^{\perp_{\ge2}}\subset R\modl$.
\end{lem}

\begin{proof}
 This is~\cite[Lemma~6.3]{PS}.
 Both the assertions follow from
Lemma~\ref{classes-c-n-closedness-properties}.
\end{proof}

\begin{rem} \label{obtainability-limitations-of-use}
 We will never use the full power of the above definitions of
obtainability in our constructions in this paper.
 In particular, all our transfinitely iterated extensions (in
the sense of the projective limit) will be indexed by the ordinal
of nonnegative integers~$\omega$ (we will call these \emph{infinitely
iterated extensions}).
 However, we will sometimes use uncountably infinite products.
 Besides, the second part of the rule~(i) and the rule~(iii) will be
only used in the following weak combination:
\begin{enumerate}
\renewcommand{\theenumi}{\roman{enumi}$'$}
\setcounter{enumi}{2}
\item the kernel of any surjective morphism from a right
$1$\+obtainable object to a right $1$\+obtainable object is right
$2$\+obtainable.
\end{enumerate}
 We refer to~\cite[Remarks~3.7 and~6.4]{PS} for comparison.
\end{rem}

 In fact, the following lemma essentially captures all our uses of
the rule~(iii$'$) (and hence also of the rules~(i) and~(iii)).

\begin{lem} \label{2-obtainable-lem}
 Let\/ $\sC$ denote the class of all left $R$\+modules $C$ such that\/
$\Ext^1_R(Q,C)=0$ for all left $R$\+modules $Q$ of projective
dimension\/~$1$.
 Then all left $R$\+modules are right $2$\+obtainable from\/~$\sC$.
\end{lem}

\begin{proof}
 Any left $R$\+module $A$ can be embedded into an $R$\+module 
$B$ belonging to~$\sC$ (e.~g., an injective left $R$\+module~$B$).
 The quotient $R$\+module $B/A$ then also belongs to~$\sC$.
 Now $A$ is the kernel of the surjective morphism $B\rarrow B/A$, so
it is right $2$\+obtainable from $B$ and~$B/A$.
\end{proof}

\Section{$S$-Contramodule $R$-Modules} \label{contramodules-secn}

 Let $R$ be a commutative ring.
 We will use the notation $\pd_RM$ for the projective dimension of
an $R$\+module~$M$.

 Let $S\subset R$ be a multiplicative subset.
 We will use the notation $S^{-1}R$ for the localization of the ring $R$
at the multiplicative subset~$S$.
 For any $R$\+module $M$, we set $S^{-1}M=S^{-1}R\ot_RM$.

 An $R$\+module $M$ is said to be \emph{$S$\+torsion} if for every
element $x\in M$ there exists an element $s\in S$ such that $sx=0$
in~$M$.
 The (unique) maximal $S$\+torsion submodule of an $R$\+module $M$
is denoted by $\Gamma_S(M)\subset M$.
 An $R$\+module $M$ is said to have \emph{bounded $S$\+torsion} if
there exists an element $s_0\in S$ such that $s_0\Gamma_S(M)=0$.

 Assume that the projective dimension $\pd_RS^{-1}R$ of the $R$\+module
$S^{-1}R$ does not exceed~$1$.
 In this case, an $R$\+module $C$ is said to be
an \emph{$S$\+contramodule}~\cite{PMat} if
$\Hom_R(S^{-1}R,C)=0=\Ext_R^1(S^{-1}R,C)$.
 The full subcategory of $S$\+contramodule $R$\+modules is denoted by
$R\modl_{S\ctra}\subset R\modl$.

 Given a complex of $R$\+modules $K^\bu$ and an $R$\+module $A$, we will
use the simplified notation $\Ext_R^i(K^\bu,A)$ for the $R$\+modules
$\Hom_{\sD(R\modl)}(K^\bu,A[i])$ of morphisms in the derived category
of $R$\+modules $\sD(R\modl)$.

 The two-term complex of $R$\+modules $R\rarrow S^{-1}R$ will be
particularly important for us.
 We denote it by $K^\bu_{R,S}$ and place in the cohomological
degrees~$-1$ and~$0$ (so one has $K^{-1}_{R,S}=R$ and $K^0_{R,S}=S^{-1}R$).
 The functor $A\longmapsto\Ext^1_R(K^\bu_{R,S},A)$ on the category
of $R$\+modules is denoted by~$\Delta_{R,S}$.

\begin{lem} \label{S-contramodule-category-lem}
 Let $R$ be a commutative ring and $S\subset R$ be a multiplicative
subset such that\/ $\pd_RS^{-1}R\le1$.  Then \par
\textup{(a)} the full subcategory $R\modl_{S\ctra}$ is closed under
kernels, cokernels, extensions, and infinite products in
$R\modl$; so, in particular, $R\modl_{S\ctra}$ is an abelian category and
its embedding $R\modl_{S\ctra}\rarrow R\modl$ is an exact functor; \par
\textup{(b)} for any $R$\+module $A$, the $R$\+module\/
$\Delta_{R,S}(A)=\Ext_R^1(K^\bu_{R,S},A)$ is an $S$\+contra\-module;
the functor\/ $\Delta_{R,S}\:R\modl\rarrow R\modl_{S\ctra}$ is left
adjoint to the embedding functor $R\modl_{S\ctra}\rarrow R\modl$.
\end{lem}

\begin{proof}
 This is~\cite[Theorem~3.4]{PMat}.
\end{proof}

 Let us also introduce notation for the $S$\+completion functor
$\Lambda_{R,S}\:R\modl\rarrow R\modl$
$$
 \Lambda_{R,S}(A)=\varprojlim\nolimits_{s\in S}A/sA,
$$
where the projective limit is taken over the set $S$ endowed with
the preorder $t\ge s$, \ $s$, $t\in S$ if there exists $r\in R$ such
that $t=rs$.
 For any $R$\+module $A$, there is a natural $R$\+module morphism
$\beta_{R,S,A}\:\Delta_{R,S}(A)\rarrow\Lambda_{R,S}(A)$
\cite[Lemma~2.1(b)]{PMat}.

 The rest of this section is devoted to the study of right obtainability
properties of $S$\+contramodule $R$\+modules (in the sense of
Section~\ref{obtainable-secn}).
 We consider three situations (corresponding to the assumption sets of
three Theorems~\ref{one-countable-multsubset-thm},
\ref{one-regular-matlis-multsubset-thm},
and~\ref{one-bounded-torsion-matlis-multsubset-thm}) separately.

\subsection{Countable multiplicative subsets}
\label{countable-multsubsets-subsecn}
 Let $R$ be a commutative ring and $S\subset R$ be an (at most)
countable multiplicative subset.
 Firstly we observe that the projective dimension of the $R$\+module
$S^{-1}R$ does not exceed~$1$ in this case~\cite[Lemma~1.9]{PMat}.

 Let $s_1$, $s_2$, $s_2$,~\dots\ be a sequence of elements of $S$ such
that every element of $S$ appears infinitely many times in this
sequence.
 Set $t_0=1$ and $t_n=s_1\dotsm s_n$ for all $n\ge1$.
 Then for any $R$\+module $A$ we have $\Lambda_{R,S}(A)=
\varprojlim_{n\ge1}A/t_nA$, where the projective limit is taken over
the natural surjective maps $A/t_nA\rarrow A/t_{n-1}A$.

 For any $R$\+module $A$ and an element $r\in R$, denote by
${}_rA\subset A$ the submodule of all elements annihilated by~$r$
in $A$ (i.~e., the kernel of the multiplication map $r\:A\rarrow A$).
 Then the $R$\+modules ${}_{t_n}\!\.A$, \ $n\ge1$ form a projective
system with the projection maps $s_n\:{}_{t_n}A\rarrow {}_{t_{n-1}}A$.

\begin{lem} \label{S-delta-lambda-sequence}
 Let $R$ be a commutative ring and $S\subset R$ be an (at most)
countable multiplicative subset.
 Then for any $R$\+module $A$ there is a natural short exact sequence
of $R$\+modules
$$
 0\lrarrow\varprojlim\nolimits_{n\ge1}^1\.{}_{t_n}\!\.A\lrarrow
 \Delta_{R,S}(A)\lrarrow\varprojlim\nolimits_{n\ge1}A/t_nA\lrarrow0.
$$
\end{lem}

 In other words, the natural morphism $\beta_{R,S,A}\:\Delta_{R,S}(A)
\rarrow\Lambda_{R,S}(A)$ is surjective with the kernel isomorphic
to $\varprojlim_{n\ge1}^1\.{}_{t_n}\!\.A$.

\begin{proof}
 This is a generalization of~\cite[Lemma~6.7]{Pcta}
and~\cite[Sublemma~4.6]{PS} (which, in turn, is a particular case
of~\cite[Lemma~7.5]{Pcta}).
 The reason why it holds is, essentially, because the complex
$K_{R,S}^\bu=(R\to S^{-1}R)$ is the inductive limit of the complexes
$R\overset{t_n}\rarrow R$ over the inductive system formed by
the morphisms
$$
 (R\overset{t_{n-1}}\rarrow R)\overset{(1,\.s_n)}\lrarrow
 (R\overset{t_n}\rarrow R).
$$

 To give a formal proof, consider the two-term complex
$$
 \bigoplus\nolimits_{n=0}^\infty R\lrarrow
 \bigoplus\nolimits_{n=1}^\infty R
$$
with the differential taking an eventually vanishing sequence
$x_0$, $x_1$, $x_2$,~\dots~$\in R$ to the eventually vanishing sequence
$y_1$, $y_2$, $y_3$,~\dots~$\in R$ with $y_n=x_n-s_nx_{n-1}$, \ $n\ge1$.
 We denote this complex by $T^\bu$ and place it in the cohomological
degrees~$0$ and~$1$ (so $T^0=\bigoplus_{n=0}^\infty R$ and
$T^1=\bigoplus_{n=1}^\infty R$).
 Furthermore, denote by $T_n^\bu$ the subcomplex
$$
 \bigoplus\nolimits_{i=0}^{n-1} R\lrarrow
 \bigoplus\nolimits_{i=1}^n R
$$
in the complex $T^\bu$.
 Then the complex $T_n^\bu$ is homotopy equivalent to the complex
$R\overset{t_n}\rarrow R$; the homotopy equivalence is provided by
the morphism of complexes taking a sequence of elements
$x_0$,~\dots,~$x_{n-1}\in R$ to the element $x_0\in R$ and
a sequence of elements $y_1$,~\dots,~$y_n\in R$ to the element
$-s_ns_{n-1}\dotsm s_2y_1-s_ns_{n-1}\dotsm s_3y_2-
\dotsb-y_n\in R$.
 It follows that the complex $T^\bu$ is quasi-isomorphic to
the complex $R\rarrow S^{-1}R$; the quasi-isomorphism is provided
by the morphism of complexes taking an eventually vanishing sequence
$(x_n\in R)_{n=0}^\infty$ to the element $x_0\in R$ and
an eventually vanishing sequence $(y_n\in R)_{n=1}^\infty$ to
the element $-\sum_{n=1}^\infty y_n/t_n\in S^{-1}R$.
 To be more precise, the complex $T^\bu$ is quasi-isomorphic to
the complex $K_{R,S}^\bu[-1]$.

 Now we can compute the $R$\+module $\Delta_{R,S}(A)=
\Ext_R^1(K_{R,S}^\bu,A)$ as $\Delta_{R,S}(A)=H_0(\Hom_R(T^\bu,A))$.
 Furthermore, the complex $\Hom_R(T^\bu,A)$ is the projective limit
of the complexes $\Hom_R(T^\bu_n,A)$, which form a countable directed
projective system of complexes with termwise surjective morphisms
between them.
 Hence we have a natural short exact sequence of $R$\+modules
\begin{multline*}
 0\lrarrow\varprojlim\nolimits_n^1H_1(\Hom_R(T^\bu_n,A))
 \lrarrow H_0(\varprojlim\nolimits_n\Hom_R(T^\bu_n,A)) \\
 \lrarrow\varprojlim\nolimits_n H_0(\Hom_R(T^\bu_n,A))\lrarrow0,
\end{multline*}
and it remains to recall that $H_0(\Hom_R(T^\bu_n,A))=
H_0(A\overset{t_n}\to A)=A/t_nA$ and $H_1(\Hom_R(T^\bu_n,A))=
H_1(A\overset{t_n}\to A)={}_{t_n}\!\.A$.
\end{proof}

\begin{lem} \label{projlim-obtainable-lem}
 Let $R$ be a commutative ring and $S\subset R$ be an (at most)
countable multiplicative subset.
 Let $D_1\larrow D_2\larrow D_3\larrow\dotsb$ be a projective system
of $R$\+modules such that $D_n$ is an $R/t_nR$\+module for every
$n\ge1$.
 Then the $R$\+modules \textup{(a)}~$\varprojlim_n D_n$ and
\textup{(b)}~$\varprojlim_n^1 D_n$ are simply right obtainable from
$R/sR$\+modules, $s\in S$.
\end{lem}

\begin{proof}
 This is a generalization of~\cite[Sublemma~4.7]{PS}.
 Part~(a): denote by $D_n'\subset D_n$ the image of the projection
map $\varprojlim_m D_m\rarrow D_n$.
 Then we have $\varprojlim_n D_n=\varprojlim_n D'_n$, and the maps
$D'_n\rarrow D'_{n-1}$ are surjective.
 Hence the $R$\+module $\varprojlim_n D_n$ is an infinitely iterated
extension, in the sense of the projective limit, of the $R$\+modules
$D'_1$ and $\ker(D'_n\to D'_{n-1})$, \ $n\ge2$.
 These are, obviously, $R/t_nR$\+modules.
 The proof of part~(b) is the same as in~\cite[Sublemma~4.7(b)]{PS}.
\end{proof}

\begin{lem} \label{countable-S-contramodules-obtainable}
 Let $R$ be a commutative ring and $S\subset R$ be an (at most)
countable multiplicative subset.
 Then all $S$\+contramodule $R$\+modules are simply right obtainable
from $R/sR$\+modules, $s\in R$.
\end{lem}

\begin{proof}
 By Lemma~\ref{S-contramodule-category-lem}(b), any $S$\+contramodule
$R$\+module $C$ has the form $C=\Delta_{R,S}(A)$ for some $R$\+module
$A$; in fact, one has $C=\Delta_{R,S}(C)$.
 According to Lemma~\ref{S-delta-lambda-sequence}, the $R$\+module
$C$ is an extension of two $R$\+modules, both of which, according
to Lemma~\ref{projlim-obtainable-lem}(a\+b), are simply right
obtainable from $R/sR$\+modules.
\end{proof}

\begin{rem} \label{obtainable-are-contramodules-remark}
 Notice that the converse assertion to
Lemma~\ref{countable-S-contramodules-obtainable} obviously holds for
any multiplicative subset $S$ in a commutative ring $R$ such that
$\pd_RS^{-1}R\le1$.
 In fact, all the $R$\+modules right $1$\+obtainable from
$R/sR$\+modules, and even all the $R$\+modules right $2$\+obtainable
from $R/sR$\+modules, $s\in S$, are $S$\+contramodules.
 This follows simply from the fact that the class of $S$\+contramodule
$R$\+modules is closed under kernels, cokernels, extensions,
and infinite products (which implies closedness under projective
limits, which can be expressed as kernels of morphisms between
products) by Lemma~\ref{S-contramodule-category-lem}(a),
together with the fact that all $R/sR$\+modules are
$S$\+contramodules~\cite[Lemma~1.6(b)]{PMat}.
\end{rem}

\subsection{Matlis multiplicative subsets of nonzero-divisors}
\label{regular-matlis-multsubset-subsecn}
 A \emph{Matlis multiplicative subset} $S$ in a commutative ring $R$
is a multiplicative subset $S\subset R$ such that $\pd_RS^{-1}R\le1$.
 All the results concerning uncountable Matlis multiplicative subsets
in this paper are based on the following theorem from
the paper~\cite{AHT}.

\begin{thm} \label{aht-theorem}
 Let $R$ be a commutative ring and $S\subset R$ be a multiplicative
subset consisting of (some) nonzero-divisors in~$R$.
 Then the projective dimension of the $R$\+module $S^{-1}R$ does not
exceed\/~$1$ if and only if the quotient $R$\+module $S^{-1}R/R$ is
isomorphic to a direct sum of countably presented $R$\+modules.
\end{thm}

\begin{proof}
 This is~\cite[Theorem~1.1\,(1)$\Leftrightarrow$(5)]{AHT}.
 The argument combines two techniques: Hamsher's
``restrictions''~\cite{Ham} and tight systems.
 Another proof, using tilting theory rather than tight systems, can be
found in~\cite[Theorem~14.59(a)$\Leftrightarrow$(c)]{GT}
(the required result from tilting theory is ``deconstruction to
countable type'', another proof of which can be found
in~\cite[Theorem~3.6]{SS}).
 For the particular case when $R$ is an integral domain,
see~\cite[Corollary~2.8]{FS0}.
\end{proof}

 Let us have a little discussion of what Theorem~\ref{aht-theorem}
entails.
 Let $R$ be a commutative ring and $S\subset R$ be a multiplicative
subset of nonzero-divisors.
 Then, for any multiplicative subset $T\subset S\subset R$, one has
$T^{-1}R\subset S^{-1}R$ and $T^{-1}R/R\subset S^{-1}R/R$.
 Obviously, any countable subset in $S^{-1}R/R$ is contained in
$T^{-1}R/R$ for some countable multiplicative subset $T\subset S$.

 Furthermore, if an $R$\+submodule $M\subset S^{-1}R/R$ is a direct
summand in $S^{-1}R/R$ and $T\subset S$ is a multiplicative subset
such that $M\subset T^{-1}R/R\subset S^{-1}R/R$, then $M$ is a direct
summand in $T^{-1}R/R$ as well.
 Thus, it follows from Theorem~\ref{aht-theorem} that, for any Matlis
multiplicative subset of nonzero-divisors $S$ in a commutative ring
$R$, the quotient module $S^{-1}R/R$ is a direct sum of direct summands
of the quotient modules $T^{-1}R/R$, where $T$ runs over countable
multiplicative subsets in $R$ contained in~$S$.

\begin{prop} \label{regular-matlis-delta-product-decomposition}
 Let $R$ be a commutative ring and $S\subset R$ be a multiplicative
subset of nonzero-divisors such that\/ $\pd_RS^{-1}R\le1$.
 Then the functor\/ $\Delta_{R,S}\:R\modl\rarrow R\modl_{S\ctra}$ is
isomorphic to an infinite product of functors\/~$\Delta_{R,S}^\alpha$,
$$
 \Delta_{R,S}(A)=\prod\nolimits_\alpha\Delta_{R,S}^\alpha(A)
 \qquad\text{for all \,$A\in R\modl$},
$$
indexed by some set of indices~$\{\alpha\}$ and such that for
every\/~$\alpha$ there is a countable multiplicative subset
$T_\alpha\subset S$ for which the functor\/ $\Delta_{R,S}^\alpha$ is
a direct summand in the functor\/~$\Delta_{R,T_\alpha}$,
$$
 \Delta_{R,T_\alpha}(A)=\Delta_{R,S}^\alpha(A)\.\oplus\.
 {}'\!\Delta_{R,S}^\alpha(A) \qquad\text{for all \,$A\in R\modl$}.
$$
\end{prop}

\begin{proof}
 For any multiplicative subset of nonzero-divisors $S\subset R$,
the complex $K^\bu_{R,S}=(R\to S^{-1}R)$ is quasi-isomorphic to
the quotient module $S^{-1}R/R$, so one has $\Delta_{R,S}(A)=
\Ext^1_R(S^{-1}R/R,\.A)$.
 Now if $\pd_RS^{-1}R\le1$ then, according to the preceding discussion,
we have $S^{-1}R/R=\bigoplus_\alpha L_\alpha$, where the $R$\+modules
$L_\alpha$ are direct summands in the quotient modules
$T_\alpha^{-1}R/R$ for some countable multiplicative subsets
$T_\alpha\subset S$.
 It remains to set $\Delta_{R,S}^\alpha=\Ext^1_R(L_\alpha,{-})$.
\end{proof}

\begin{lem} \label{regular-matlis-contramodules-obtainable}
 Let $R$ be a commutative ring and $S\subset R$ be a multiplicative
subset of nonzero-divisors such that\/ $\pd_RS^{-1}R\le1$.
 Then all $S$\+contramodule $R$\+modules are simply right obtainable
from $R/sR$\+modules, $s\in S$.
\end{lem}

\begin{proof}
 Let $C$ be an $S$\+contramodule $R$\+module.
 By Lemma~\ref{S-contramodule-category-lem}(b) and
Proposition~\ref{regular-matlis-delta-product-decomposition},
we have $C=\Delta_{R,S}(C)=\prod_\alpha\Delta_{R,S}^\alpha(C)$,
where the $R$\+modules $\Delta_{R,S}^\alpha(C)$ are direct summands
in the $R$\+modules $\Delta_{R,T_\alpha}(C)$ for some countable
multiplicative subsets $T_\alpha\subset S\subset R$.
 By Lemma~\ref{countable-S-contramodules-obtainable}, every $R$\+module
$\Delta_{R,T_\alpha}(C)$ is simply right obtainable from $R/tR$\+modules,
$t\in T_\alpha\subset S$, hence it follows that the $R$\+module $C$ is
simply right obtainable from $R/sR$\+modules, $s\in S$.
\end{proof}

\subsection{Matlis multiplicative subsets with bounded torsion}
 Let $R$ be a commutative ring and $S\subset R$ be a multiplicative
subset.
 Then the maximal $S$\+torsion submodule $\Gamma_S(R)\subset R$ is
the kernel ideal of the ring homomorphism $R\rarrow S^{-1}R$.
 Set $I=\Gamma_S(R)$ and denote by $\oR$ the quotient ring $\oR=R/I$.

 For any multiplicative subset $T\subset R$, the image of $T$ under
the ring homomorphism $R\rarrow\oR$ is a multiplicative subset
$\oT\subset\oR$.
 The localization of the $R$\+module $\oR$ at the multiplicative
subset $T\subset R$ is the same thing as the localization of the ring
$\oR$ at the multiplicative subset $\oT\subset\oR$, that is
$T^{-1}\oR=\oT^{-1}\oR$.

 In particular, we denote by $\oS\subset\oR$ the image of
the multiplicative subset $S\subset R$ in the ring~$\oR$.
 Then we have a natural isomorphism of rings $S^{-1}R=
\oS^{-1}\oR$ and a commutative diagram of ring homomorphisms
$R\rarrow\oR\rarrow S^{-1}\oR=S^{-1}R$.
 The multiplicative subset $\oS\subset\oR$ is a multiplicative
subset of nonzero-divisors in~$\oR$.

 We are indebted to Silvana Bazzoni for the idea to use the simple
observation that is formulated in the next lemma.

\begin{lem} \label{reduction-remains-matlis}
 Let $R$ be a commutative ring and $S\subset R$ be a multiplicative
subset such that\/ $\pd_RS^{-1}R\le1$.
 Then one has\/ $\pd_\oR\oS^{-1}\oR\le1$.
\end{lem}

\begin{proof}
 Clearly, $\oS^{-1}\oR=\oR\ot_RS^{-1}R$.
 For any ring homomorphism $R\rarrow R'$ and any flat $R$\+module $F$,
one has $\pd_{R'}(R'\ot F)\le\pd_RF$, as tensoring with $R'$ over $R$
transforms a projective resolution of the $R$\+module $F$ into
a projective resolution of the $R'$\+module $R'\ot_RF$.
\end{proof}

 The following two results from the paper~\cite{AHT}, forming together
a slightly more precise version of Theorem~\ref{aht-theorem}, will be
needed for the purposes of this section.

\begin{thm} \label{aht-enlarge-countable-multsubset}
 Let $R$ be a commutative ring and $S\subset R$ be a multiplicative
subset of nonzero-divisors such that\/ $\pd_RS^{-1}R\le1$.
 Then for any countable multiplicative subset $T_0\subset S$ there
exists a countable multiplicative subset $T_0\subset T_1\subset S$
such that the $R$\+submodule $T_1^{-1}R/R\subset S^{-1}R/R$ is
a direct summand.
\end{thm}

\begin{proof}
 This is~\cite[Corollary~4.4]{AHT}.
 (For the particular case when $R$ is an integral domain,
see~\cite[sentence after Theorem~2.7]{FS0}.)
\end{proof}

\begin{thm} \label{aht-decompose-direct-summand}
 Let $R$ be a commutative ring and $S\subset R$ be a multiplicative
subset of nonzero-divisors such that\/ $\pd_RS^{-1}R\le1$.
 Then every direct summand in the $R$\+module $S^{-1}R/R$ is isomorphic
to a direct sum of countably presented $R$\+modules.
\end{thm}

\begin{proof}
 This is~\cite[Theorem~7.11]{AHT}.
\end{proof}

\begin{lem} \label{complex-K-R-S-direct-sum-decomposition}
 Let $R$ be a commutative ring and $S\subset R$ be a multiplicative
subset such that the $S$\+torsion in $R$ is bounded and\/
$\pd_\oR\oS^{-1}\oR\le1$.
 Then the two-term complex of $R$\+modules $K^\bu_{R,S}$ is isomorphic,
as an object of the derived category\/ $\sD(R\modl)$, to a direct sum
of complexes of $R$\+modules $L^\bu_\alpha$ such that,
for every~$\alpha$, there exists a countable multiplicative subset
$T_\alpha\subset S\subset R$ for which the complex $L^\bu_\alpha$,
viewed as an object of\/ $\sD(R\modl)$, is a direct summand in
the complex $K^\bu_{R,T_\alpha}$.
\end{lem}

\begin{proof}
 Let $t_0\in S$ be an element such that $t_0I=0$.
 By Theorem~\ref{aht-enlarge-countable-multsubset} applied to
the commutative ring $\oR$ with a multiplicative subset of
nonzero-divisors $\oS\subset\oR$, there exists a countable
multiplicative subset $\oT_1\subset\oS$ containing the image
$\bar t_0\in\oR$ of the element $t_0\in R$ and such that
the $\oR$\+submodule $\oT_1^{-1}\oR/\oR$ is a direct summand in
the $\oR$\+module $\oS^{-1}\oR/\oR$.
 Lifting every element of $\oT_1$ to an element of $S\subset R$ and
taking the multiplicative closure, we obtain a countable multiplicative
subset $T_1\subset R$ such that $t_0\in T_1\subset S$ and the image of
$T_1$ in $\oR$ coincides with $\oT_1\subset\oR$.

 Let $M\subset\oS^{-1}\oR/\oR$ denote an $\oR$\+submodule such that
$\oS^{-1}\oR/\oR=\oT_1^{-1}\oR/\oR\oplus M$.
 By Theorem~\ref{aht-decompose-direct-summand}, the $\oR$\+module $M$
is a direct sum of countably generated $\oR$\+modules~$L_\beta$.
 According to the discussion in
Section~\ref{regular-matlis-multsubset-subsecn}, for every~$\beta$
there exists a countable multiplicative subset $\oT_\beta\subset\oS$
such that $L_\beta$ is a direct summand in $\oT_\beta^{-1}\oR/\oR$.
 Arguing as above, we lift the multiplicative subset $\oT_\beta
\subset\oS\subset\oR$ to a countable multiplicative subset
$T_\beta\subset S\subset R$.
 Enlarging the multiplicative subsets $\oT_\beta$ and $T_\beta$ if
necessary, we can assume that $\oT_1\subset\oT_\beta$ and
$T_1\subset T_\beta$ for all~$\beta$.

 Now, two-term complexes of $R$\+modules $K^{-1}\rarrow K^0$, viewed
as objects of the derived category $\sD(R\modl)$, are classified
by triples $(H^{-1},H^0,\xi)$, where $H^{-1}=H^{-1}(K^\bu)$ and
$H^0=H^0(K^\bu)$ are the cohomology $R$\+modules, and $\xi$~is
an extension class $\xi\in\Ext^2_R(H^0,H^{-1})$.
 In the case of the complex $K^\bu_{R,S}$, we have
$H^{-1}(K^\bu_{R,S})=I$ and $H^0(K^\bu_{R,S})=\oS^{-1}\oR/\oR$.
 For any $T_1^{-1}R$\+module $L$, one has $\Ext^*_R(L,I)=0$, because
the element $t_0\in R$ acts by an automorphism of~$L$ and by zero
in~$I$.

 In particular, $M=\oS^{-1}\oR/\oT_1^{-1}\oR$ is
a $\oT_1^{-1}\oR$\+module.
 So we have $\Ext^2_R(M,I)=0$ and $\oS^{-1}\oR/\oR=
\oT_1^{-1}\oR/\oR\oplus M$, hence it follows that the complex
$K^\bu_{R,S}$ is isomorphic to the direct sum of its subcomplex
$K^\bu_{R,T_1}$ and the $R$\+module $M$ in $\sD(R\modl)$.

 Similarly, in the case of the complex $K^\bu_{R,T_\beta}$, we have
$H^{-1}(K^\bu_{R,T_\beta})=I$ and $H^0(K^\bu_{R,T_\beta})=
\oT_\beta^{-1}\oR/\oR$.
 Since $\Ext^2_R(L_\beta,I)=0$ and the $\oR$\+module $L_\beta$ is
a direct summand in the $\oR$\+module $\oT_\beta^{-1}\oR/\oR$,
it follows that the $R$\+module $L_\beta$ is a direct summand of
the complex $K^\bu_{R,T_\beta}$ in $\sD(R\modl)$.

 We have shown that $K^\bu_{R,S}\simeq K^\bu_{R,T_1}\oplus\bigoplus_\beta
L_\beta$ in $\sD(R\modl)$, where $L_\beta$ is a direct summand
of $K^\bu_{R,T_\beta}$ in $\sD(R\modl)$.
 It remains to define the set of indices $\{\alpha\}$ to be the disjoint
union of the set of indices $\{\beta\}$ and the one-element set~$\{1\}$,
that is $\{\alpha\}=\{\beta\}\sqcup\{1\}$.
 Put $L_\beta^\bu=L_\beta$ and $L_1^\bu=K^\bu_{R,T_1}$.
\end{proof}

 The following corollary provides a partial converse assertion to
Lemma~\ref{reduction-remains-matlis}.

\begin{cor} \label{reduction-reflects-matlis-cor}
 Let $R$ be a commutative ring and $S\subset R$ be a multiplicative
subset such that the $S$\+torsion in $R$ is bounded and\/
$\pd_\oR\oS^{-1}\oR\le1$.
 Then\/ $\pd_RS^{-1}R\le1$.
\end{cor}

\begin{proof}
 For any bounded complex of $R$\+modules $K^\bu$, denote by $\pd_RK^\bu$
the supremum of all the integers~$n$ for which there exists
an $R$\+module $A$ such that $\Ext_R^n(K^\bu,A)\ne0$.
 So $\pd_RK^\bu\in\boZ\cup\{+\infty\}$ if $K^\bu\ne0$ in $\sD(R\modl)$,
and $\pd_RK^\bu=-\infty$ when the complex $K^\bu$ is acyclic.
 For any distinguished triangle $K^\bu\rarrow L^\bu\rarrow M^\bu\rarrow
K^\bu[1]$ in the bounded derived category $\sD^\b(R\modl)$, we have
$\pd_RL^\bu\le\max(\pd_RK^\bu,\pd_RM^\bu)$.

 In particular, for any commutative ring $R$ with a multiplicative
subset $S\subset R$ we have $\pd_RK_{R,S}^\bu\le1$ if and only if
$\pd_RS^{-1}R\le1$, because $\pd_RR=0$.
 Now if the $S$\+torsion in $R$ is bounded and $\pd_\oR\oS^{-1}\oR\le1$,
then we can apply Lemma~\ref{complex-K-R-S-direct-sum-decomposition},
obtaining a direct sum decomposition
$K^\bu_{R,S}\simeq\bigoplus_\alpha L^\bu_\alpha$ in $\sD^\b(R\modl)$,
where $L^\bu_\alpha$ is a direct summand of $K^\bu_{R,T_\alpha}$ in
$\sD^\b(R\modl)$ and the multiplicative subset $T_\alpha\subset S
\subset R$ is countable.
 By~\cite[Lemma~1.9]{PMat}, we have $\pd_RT_\alpha^{-1}R\le1$, hence
$\pd_RK^\bu_{R,T_\alpha}\le1$ and $\pd_RL^\bu_\alpha\le1$
for all~$\alpha$.
 Thus $\pd_RK^\bu_{R,S}\le1$, and it follows that $\pd_RS^{-1}R\le1$.
\end{proof}

 The next proposition is a generalization of
Proposition~\ref{regular-matlis-delta-product-decomposition} to
the case of bounded torsion.

\begin{prop} \label{bounded-torsion-matlis-delta-product-decomposition}
 Let $R$ be a commutative ring and $S\subset R$ be a multiplicative
subset such that the $S$\+torsion in $R$ is bounded and\/
$\pd_RS^{-1}R\le1$.
 Then the functor\/ $\Delta_{R,S}\:R\modl\rarrow R\modl_{S\ctra}$ is
isomorphic to an infinite product of functors\/~$\Delta_{R,S}^\alpha$,
$$
 \Delta_{R,S}(A)=\prod\nolimits_\alpha\Delta_{R,S}^\alpha(A)
 \qquad\text{for all \,$A\in R\modl$},
$$
indexed by some set of indices\/~$\{\alpha\}$ and such that for
every~$\alpha$ there is a countable multiplicative subset
$T_\alpha\subset S$ for which the functor\/ $\Delta_{R,S}^\alpha$ is
a direct summand in the functor\/~$\Delta_{R,T_\alpha}$,
$$
 \Delta_{R,T_\alpha}(A)=\Delta_{R,S}^\alpha(A)\.\oplus\.
 {}'\!\Delta_{R,S}^\alpha(A) \qquad\text{for all \,$A\in R\modl$}.
$$
\end{prop}

\begin{proof}
 In view of Lemma~\ref{reduction-remains-matlis},
Lemma~\ref{complex-K-R-S-direct-sum-decomposition} is applicable,
and it remains to recall that $\Delta_{R,S}=\Ext^1_R(K^\bu_{R,S},{-})$
and $\Delta_{R,T_\alpha}=\Ext^1_R(K^\bu_{R,T_\alpha},{-})$.
\end{proof}

\begin{lem} \label{bounded-torsion-matlis-contramodules-obtainable}
 Let $R$ be a commutative ring and $S\subset R$ be a multiplicative
subset such that the $S$\+torsion in $R$ is bounded and\/
$\pd_RS^{-1}R\le1$.
 Then all $S$\+contramodule $R$\+modules are simply right obtainable
from $R/sR$\+modules, $s\in S$.
\end{lem}

\begin{proof}
 Similar to the proof of
Lemma~\ref{regular-matlis-contramodules-obtainable},
using Lemma~\ref{S-contramodule-category-lem}(b),
Proposition~\ref{bounded-torsion-matlis-delta-product-decomposition},
and Lemma~\ref{countable-S-contramodules-obtainable}.
\end{proof}

\Section{Rings with a Multiplicative Subset} \label{one-multsubset-secn}

 We start with a proof of Proposition~\ref{ooc-counterex-prop} and
then proceed to prove
Propositions~\ref{one-countable-multsubset-prop}\+-%
\ref{one-bounded-torsion-matlis-multsubset-prop} and
Theorems~\ref{one-countable-multsubset-thm}\+-%
\ref{one-bounded-torsion-matlis-multsubset-thm}
(which is the aim of this section).

 Let $R$ be a commutative ring and $S\subset R$ be a multiplicative
subset.
 Recall that an $R$\+module $C$ is said to be \emph{$S$\+weakly
cotorsion} if $\Ext_R^1(S^{-1}R,C)=0$, and an $R$\+module $F$ is
said to be \emph{$S$\+strongly flat} if $\Ext_R^1(F,C)=0$ for all
$S$\+weakly cotorsion $R$\+modules~$C$.

 An $R$\+module $M$ is said to be \emph{$S$\+divisible} if $sM=M$
for all $s\in S$.
 An $R$\+module $M$ is said to be \emph{$S$\+h-divisible} if there
exists a surjective $R$\+module morphism onto $M$ from
an $S^{-1}R$\+module.
 Any $S$\+h-divisible $R$\+module is $S$\+divisible.
 An $R$\+module $M$ is $S$\+h-divisible if and only if the natural
morphism $\Hom_R(S^{-1}R,M)\allowbreak\rarrow M$ is surjective.
 More precisely, for any $R$\+module $M$ let us denote the image
of the morphism $\Hom_R(S^{-1}R,M)\rarrow M$ by $h_S(M)\subset M$.
 Then $h_S(M)$ is the (unique) maximal $S$\+h-divisible submodule
of~$M$.
 An $R$\+module $M$ is said to be \emph{$S$\+h-reduced} if it has
no $S$\+h-divisible submodules, that is $h_S(M)=0$.
 An $R$\+module $M$ is $S$\+h-reduced if and only if
$\Hom_R(S^{-1}R,M)=0$.
 We refer to~\cite[Section~1]{PMat} and the references therein for
further discussion.

\begin{proof}[Proof of Proposition~\ref{ooc-counterex-prop}]
 For any commutative ring $R$ with a multiplicative subset $S\subset R$,
the class of all flat $R$\+modules $F$ such that the $R/sR$\+module
$F/sF$ is projective for all $s\in S$ and the $S^{-1}R$\+module
$S^{-1}F$ is projective is closed under kernels of surjective
morphisms (as well as under extensions and infinite direct sums, and
more generally, transfinitely iterated extensions in the sense of
the inductive limit).
 Since the $R$\+module $S^{-1}R$ belongs to this class, so does its first
syzygy $R$\+module~$H$.

 On the other hand, assume that the multiplicative subset $S\subset R$
consists of (some) nonzero-divisors in a~$R$.
 Suppose that the $R$\+module $H$ is $S$\+strongly flat.
 Then we have $\Ext^2_R(S^{-1}R,C)=\Ext^1_R(H,C)=0$ for all $S$\+weakly
cotorsion $R$\+modules~$C$.
 Hence the cokernel of any injective morphism of $S$\+weakly cotorsion
$R$\+modules is an $S$\+weakly cotorsion $R$\+module.

 For the purposes of this proof, let $K$ denote the quotient
$R$\+module $S^{-1}R/R$.
 Following the argument in~\cite[proof of
Lemma~7.53(c)$\Rightarrow$(a)]{GT} and using~\cite[Lemma~1.7(a)]{PMat},
one shows that the $R$\+module $\Ext^1_R(K,M)$ is $S$\+h-reduced for any
$R$\+module $M$, hence the $R$\+module $M/h_S(M)$ is $S$\+h-reduced for
any $M$, and therefore the class of all $S$\+h-divisible $R$\+modules
is closed under extensions.
 Finally, it remains to apply~\cite[Lemma~1.8(b)]{PMat} in order to
conclude that the projective dimension of the $R$\+module $S^{-1}R$
does not exceed~$1$.
\end{proof}

 As in Section~\ref{contramodules-secn}, we denote by $K^\bu_{R,S}$
the two-term complex of $R$\+modules $R\rarrow S^{-1}R$, with the term
$R$ sitting in the cohomological degree~$-1$ and the term $S^{-1}R$
in the cohomological degree~$0$.
 We also denote simply by $S^{-1}R/R$ the cokernel of the $R$\+module
morphism $R\rarrow S^{-1}R$.
 The functor $A\longmapsto\Ext_R^1(K^\bu_{R,S},A)=
\Hom_{\sD(R\modl)}(K^\bu_{R,S},A[1])$ is denoted by
$\Delta_{R,S}\:R\modl\rarrow R\modl$.

\begin{lem} \label{matlis-exact-sequence}
 Let $R$ be commutative ring and $S\subset R$ be a multiplicative
subset.
 Then for any $R$\+module $A$ there is a natural\/ $5$\+term exact
sequence of $R$\+modules
\begin{multline}
 0\lrarrow\Hom_R(S^{-1}R/R,A)\lrarrow\Hom_R(S^{-1}R,A) \\ \lrarrow A
 \lrarrow\Delta_{R,S}(A)\lrarrow\Ext_R^1(S^{-1}R,A)\lrarrow0,
\end{multline}
which in the case of an $S$\+weakly cotorsion $R$\+module $C$ reduces
to a\/ $4$\+term exact sequence
\begin{equation} \label{weakly-cotorsion-sequence}
 0\lrarrow\Hom_R(S^{-1}R/R,C)\lrarrow\Hom_R(S^{-1}R,C)\lrarrow C
 \lrarrow\Delta_{R,S}(C)\lrarrow0.
\end{equation}
\end{lem}

\begin{proof}
 Apply the cohomological functor $\Hom_{\sD(R\modl)}(-,A[*])$ to
the distinguished triangle
$$
 R\lrarrow S^{-1}R\lrarrow K^\bu_{R,S}\lrarrow R[1]
$$
in the derived category $\sD(R\modl)$.
\end{proof}

 The next lemma uses the definitions from
Section~\ref{contramodules-secn}.

\begin{lem} \label{hom-from-torsion-is-contramodule}
 Let $R$ be a commutative ring and $S\subset R$ be a multiplicative
subset such that\/ $\pd_RS^{-1}R\le1$.
 Then for any $S$\+torsion $R$\+module $D$ and any $R$\+module $A$
the $R$\+module\/ $\Hom_R(D,A)$ is an $S$\+contramodule.
\end{lem}

\begin{proof}
 Recall the notation ${}_rD\subset D$ for the kernel of
the multiplication map $r\:D\rarrow D$, \ $r\in R$.
 Then one has $D=\varinjlim_{s\in S}\.{}_sD$, hence
$\Hom_R(D,A)=\varprojlim_{s\in S}\Hom_R({}_sD,A)$.
 Now the $R$\+module $\Hom_R({}_sD,A)$ is annihilated by the action
of the element $s\in S$, so it is an $S$\+contramodule
by~\cite[Lemma~1.6(b)]{PMat}.
 By Lemma~\ref{S-contramodule-category-lem}(a), the class of
$S$\+contramodule $R$\+modules is closed under infinite products and
kernels of morphisms in $R\modl$, hence it is also closed under
all projective limits.

 In fact, the assumption of projective dimension of $S^{-1}R$ not
exceeding~$1$ is not needed for the validity of this lemma;
see~\cite[Theorem~2.1]{Mat} and~\cite[Lemma~1.4 or proof of
Lemma~1.7(a)]{PMat} (but we have only defined $S$\+contramodule
$R$\+modules in the assumption of $\pd_RS^{-1}R\le1$).
\end{proof}

\begin{lem} \label{change-of-scalars-ext}
 Let $R\rarrow R'$ be a homomorphism of commutative rings and $F$
be a flat $R$\+module.
 Then for any $R'$\+module $C'$ and all $i\ge0$ there is a natural
isomorphism of Ext groups/modules\/
$\Ext_R^i(F,C')\simeq\Ext_{R'}^i(R'\ot_RF,\>C')$.
\end{lem}

\begin{proof}
 This is a particular case of~\cite[Lemma~4.1(a)]{PS}.
\end{proof}

\begin{proof}[Proof of
Propositions~\ref{one-countable-multsubset-prop}\+-%
\ref{one-bounded-torsion-matlis-multsubset-prop}]
 In all the cases covered by the assumptions of these three
propositions, one has $\pd_RS^{-1}R\le 1$, so $\Delta_{R,S}(A)$
is an $S$\+contramodule $R$\+mod\-ule for any $R$\+module $A$
(see Lemma~\ref{S-contramodule-category-lem}(b)
and/or~\cite[Lemma~1.7(c)]{PMat}).

 Now let $C$ be an $S$\+weakly cotorsion $R$\+module.
 We will prove a stronger assertion, viz., that $C$ is simply right
obtainable from $R/sR$\+modules, $s\in S$, and an $S^{-1}R$\+module
(see Section~\ref{obtainable-secn} for the definitions).

 The exact sequence~\eqref{weakly-cotorsion-sequence} represents $C$
as an extension of an $S$\+contramodule $\Delta_{R,S}(C)$ and
the cokernel of an injective morphism from an $S$\+contramodule
$\Hom_R(S^{-1}R/R,C)$ (see Lemma~\ref{hom-from-torsion-is-contramodule})
into an $S^{-1}R$\+module $\Hom_R(S^{-1}R,C)$.
 So any $S$\+weakly cotorsion $R$\+module is simply right obtainable
from two $S$\+contramodule $R$\+modules and one
$S^{-1}R$\+module.

 It remains to recall that, as we already know, in the assumptions of
any of the three propositions all $S$\+contramodule $R$\+modules are
simply right obtainable from $R/sR$\+modules, $s\in S$.
 In the case of Proposition~\ref{one-countable-multsubset-prop}, this
is the result of Lemma~\ref{countable-S-contramodules-obtainable};
in the case of Proposition~\ref{one-regular-matlis-multsubset-prop},
we have Lemma~\ref{regular-matlis-contramodules-obtainable};
and in the case of
Proposition~\ref{one-bounded-torsion-matlis-multsubset-prop},
we need to apply
Lemma~\ref{bounded-torsion-matlis-contramodules-obtainable}.

 This proves the ``only if'' assertions of the three propositions.
 The ``if'' assertion holds for any multiplicative subset $S$ in
a commutative ring~$R$.
 To prove as much, denote by $\sE\subset R\modl$ the class of all
$R/sR$\+modules, $s\in S$, and all $S^{-1}R$\+modules (viewed as
$R$\+modules via the restriction of scalars).

 By Lemma~\ref{change-of-scalars-ext}, for any $R$\+module $E\in\sE$
we have $\Ext^i_R(S^{-1}R,E)=0$ for all $i>0$ (since $S^{-1}R$ is
a flat $R$\+module, $R/sR\ot_RS^{-1}R=0$, and $S^{-1}R\ot_RS^{-1}R=
S^{-1}R$ is a free $S^{-1}R$\+module).
 It remains to apply Lemma~\ref{obtainable-orthogonal-lemma} in order
to conclude that $\Ext^1_R(S^{-1}R,C)=0$ for all $R$\+modules $C$ right
$1$\+obtainable from~$\sE$.
\end{proof}

\begin{proof}[Proof of
Theorems~\ref{one-countable-multsubset-thm}\+-%
\ref{one-bounded-torsion-matlis-multsubset-thm}]
 The ``only if'' assertion holds for any multiplicative subset $S$ in
a commutative ring $R$ and follows immediately from the description of
$S$\+strongly flat $R$\+modules in terms of the exact
sequence~\eqref{strongly-flat-sequence}.

 To prove the ``if'', we, as above, denote by $\sE\subset R\modl$
the class of all $R/sR$\+modules, $s\in S$, and all $S^{-1}R$\+modules
(viewed as $R$\+modules via the restriction of scalars).
 Let $F$ be a flat $R$\+module such that the $R/sR$\+module $F/sF$ is
projective for all $s\in S$ and the $S^{-1}R$\+module $S^{-1}F$ is
projective.
 Then, by Lemma~\ref{change-of-scalars-ext}, we have $\Ext_R^i(F,E)=0$
for all $E\in\sE$ and all $i>0$.

 Applying Lemma~\ref{obtainable-orthogonal-lemma}, we learn that
$\Ext_R^i(F,C)=0$ for all $R$\+modules $C$ right $1$\+obtainable from
$\sE$ and all $i>0$.
 Depending on which of the three
Theorems~\ref{one-countable-multsubset-thm},
\ref{one-regular-matlis-multsubset-thm},
or~\ref{one-bounded-torsion-matlis-multsubset-thm}
we want to prove, we can use
one of the respective Propositions~\ref{one-countable-multsubset-prop},
\ref{one-regular-matlis-multsubset-prop},
or~\ref{one-bounded-torsion-matlis-multsubset-prop}, which tells that
all $S$\+weakly cotorsion $R$\+modules $C$ are right $1$\+obtainable
from~$\sE$.
 Thus $\Ext_R^1(F,C)=0$ for all $S$\+weakly cotorsion $R$\+modules $C$,
that is, $F$ is an $S$\+strongly flat $R$\+module.
\end{proof}

\Section{Rings with Several Multiplicative Subsets}
\label{several-multsubsets-secn}

 The aim of this section is to prove
Proposition~\ref{several-multsubsets-prop} and
Theorem~\ref{several-multsubsets-thm}.
 The arguments here generalize those in~\cite[Section~8]{PS}.

\begin{lem} \label{change-of-ring-multsubset}
 Let $f\: R\rarrow R'$ be a homomorphism of commutative rings,
$S\subset R$ be a multiplicative subset, and $S'=f(S)\subset R$
be the image of $S$ in~$R'$.
 In this setting: \par
\textup{(a)} if\/ $\pd_RS^{-1}R\le1$, then\/ $\pd_{R'}S'{}^{-1}R'\le1$;
\par
\textup{(b)} an $R'$\+module is an $S$\+weakly cotorsion $R$\+module
if and only if it is an $S'$\+weakly cotorsion $R'$\+module; \par
\textup{(c)} an $R'$\+module is an $S$\+contramodule
$R$\+module if and only if it is an $S'$\+contra\-module $R'$\+module.
\end{lem}

\begin{proof}
 Part~(a) was already explained in the proof of
Lemma~\ref{reduction-remains-matlis}, while parts~(b) and~(c)
follow from Lemma~\ref{change-of-scalars-ext}.
\end{proof}

\begin{lem} \label{hom-into-contramodule-is-contramodule}
 Let $R$ be a commutative ring, $s\in R$ be an element, and $S\subset R$
be a multiplicative subset such that\/ $\pd_RS^{-1}R\le1$.
 Let $K^\bu$ be a complex of $R$\+modules and $C$ be an $R$\+module.
 Then \par
\textup{(a)} the $R$\+module\/ $\Hom_R(K^\bu,C[i])$ is annihilated
by~$s$ for all $i\in\boZ$ whenever the $R$\+module $C$ is annihilated
by~$s$; \par
\textup{(b)} the $R$\+module\/ $\Hom_R(K^\bu,C[i])$ is
an $S$\+contramodule for all $i\in\boZ$ whenever the $R$\+module $C$
is an $S$\+contramodule.
\end{lem}

\begin{proof}
 Part~(a) holds because, for any $R$\+module $C$, the multiplication
map $s\:\Hom_R(K^\bu,C)\rarrow\Hom_R(K^\bu,C)$ is induced by
the multiplication map $s\:C\rarrow C$.
 Part~(b) is a (partial) generalization of the related assertion
of~\cite[Lemma~6.2(b)]{Pcta}.
 Since $\pd_RS^{-1}R\le1$, for any complex of $R$\+modules $B^\bu$
there are short exact sequences
\begin{multline*}
 0\lrarrow\Ext_R^1(S^{-1}R,H^{i-1}(B^\bu))\lrarrow \\
 H^i(\boR\Hom_R(S^{-1}R,B^\bu))\lrarrow
 \Hom_R(S^{-1}R,H^i(B^\bu))\lrarrow0,
\end{multline*}
where $\boR\Hom_R$ denotes the derived functor of $R$\+module
homomorphisms, viewed as a functor acting on the derived category
of $R$\+modules.
 Therefore, the $R$\+modules $H^i(B^\bu)$ are $S$\+contramodules for
all $i\in\boZ$ if and only if $\boR\Hom_R(S^{-1}R,B^\bu)=0$ in
$\sD(R\modl)$ (cf.~\cite[Lemma~4.5(b)]{PMat}).

 Now one has
$$
 \boR\Hom_R(S^{-1}R,\.\boR\Hom_R(K^\bu,C))=
 \boR\Hom_R(K^\bu,\.\boR\Hom_R(S^{-1}R,C)),
$$
hence acyclicity of the the complex $\boR\Hom_R(S^{-1}R,C)$ implies
acyclicity of the complex $\boR\Hom_R(S^{-1}R,\.\boR\Hom_R(K^\bu,C))$.
\end{proof}

 Let $R$ be a commutative ring and $\S=\{S_1,\dotsc,S_m\}$ be
a finite collection of multiplicative subsets in~$R$.
 Assume that $\pd_RS_j^{-1}R\le1$ for all $j=1$,~\dots,~$m$.
 We will say that an $R$\+module $C$ is an \emph{$\S$\+contramodule}
if it is an $S_j$\+contramodule for every $1\le j\le m$.

 Let $s_1\in S_1$,~\dots, $s_m\in S_m$ be a sequence of elements,
which will denote for brevity by a single letter~$\s$.
 Then we denote by $R_\s$ the quotient ring $R_\s=R/(s_1R+\dotsb+s_mR)$
of the ring $R$ by the ideal generated by the elements
$s_1$,~\dots,~$s_m$.

\begin{lem} \label{several-multsubsets-contramodule-obtainable}
 Let $R$ be a commutative ring and\/ $\S=\{S_1,\dotsc,S_m\}$ be
a finite collection of multiplicative subsets in $R$ such that, for
every\/ $1\le j\le m$, either $S_j$ is countable, or
the $S_j$\+torsion in $R$ is bounded and\/ $\pd_RS_j^{-1}R\le1$.
 Then an $R$\+module $C$ is an\/ $\S$\+contramodule if and only if
it is simply right obtainable from $R_\s$\+modules, where\/
$\s$~runs over all the sequences of elements $s_1\in S_1$,~\dots,
$s_m\in S_m$.
\end{lem}

\begin{proof}
 The ``if'' assertion holds for any collection of Matlis multiplicative
subsets $S_j\subset\nobreak R$.
 In fact, any $R$\+module right $1$\+obtainable from $R_\s$\+modules
is an $S_j$\+con\-tramodule for every~$j$ by
Remark~\ref{obtainable-are-contramodules-remark}, hence
it is an $\S$\+contramodule.
 The ``only if'' assertion is the nontrivial part of the lemma.

 The following proof is a generalization of the ``alternative'' proof
of~\cite[Lemma~8.2]{PS}.
 Arguing by induction on $n=0$,~\dots, $m$ for a fixed ring $R$ with
$m$~multiplicative subsets $S_1$,~\dots, $S_m\subset R$, we will prove
the following assertion: any $R$\+module which is an $S_j$\+contramodule
for all $1\le j\le n$ and which is annihilated by the ideal
$s_{n+1}R+\dotsb+s_mR\subset R$ for some elements $s_k\in S_k$, \
$n+1\le k\le m$, is simply right obtainable from $R_\s$\+modules.
 For $n=0$, this is a trivial assertion, which provides the induction
base.
 For $n=m$, this is the assertion of the lemma.

 We do \emph{not} want to pass to quotient rings $R'=R/(s_{n+1}R+\dotsb
+s_mR)$ for the purposes of the current induction procedure because, if
the $S_j$\+torsion in $R$ is bounded and $\pd_RS_j^{-1}R\le1$ for some
$1\le j\le n$, and $S_j'$ is the image of $S_j$ in $R'$, it does
\emph{not} follow that the $S_j'$\+torsion in $R'$ is bounded.
 So we will not really use Lemma~\ref{change-of-ring-multsubset} in
this proof, but will rather use
Lemma~\ref{hom-into-contramodule-is-contramodule}(a) instead.

 Let $C$ be an $R$\+module that is an $S_j$\+contramodule for all
$1\le j\le n$, and let $s_k\in S_k$, \ $n+1\le k\le m$, be some
elements such that $s_kC=0$ for all $n+1\le k\le m$.
 We have two cases, depending on whether the multiplicative subset
$S_n\subset R$ is countable, or it is a Matlis multiplicative subset
with bounded torsion.

 If $S_n$ is countable, we follow the proof of
Lemma~\ref{countable-S-contramodules-obtainable} for $S=S_n$ in order to
observe that $C$ is simply right obtainable from $R/sR$\+modules,
$s\in S_n$, each of which, in turn, can be obtained from $C$ in
a functorial way using the operations of the passages to the kernels
and cokernels of (natural) $R$\+module morphisms and infinite products
of $R$\+modules.
 All such $R/sR$\+modules, therefore, are $S_j$\+contramodule
$R$\+modules for all $1\le j\le n-1$ and are annihilated by~$s_k$
for $n+1\le k\le m$.
 It remains to set $s_n=s$ and use the induction assumption.

 If the $S_n$\+torsion in $R$ is bounded and $\pd_RS_n^{-1}R\le1$,
we apply Lemma~\ref{S-contramodule-category-lem}(b) and
Proposition~\ref{bounded-torsion-matlis-delta-product-decomposition},
obtaining a direct product decomposition
$$
 C=\Delta_{R,S_n}(C)=\prod\nolimits_\alpha\Delta_{R,S}^\alpha(C),
$$
where $\Delta_{R,S}^\alpha(C)$ are direct summands in the $R$\+modules
$\Delta_{R,T_\alpha}(C)$ for some countable multiplicative subsets
$T_\alpha\subset S_n$.
 By Lemma~\ref{hom-into-contramodule-is-contramodule}(a\+b),
the $R$\+modules $\Delta_{R,T_\alpha}(C)$ are annihilated by~$s_k$
for all $n+1\le k\le m$ and are $S_j$\+contramodules for all
$1\le j\le n-1$.

 Applying the proof of Lemma~\ref{countable-S-contramodules-obtainable}
for $S=T_\alpha$ to the $T_\alpha$\+contramodules
$\Delta_{R,T_\alpha}(C)$ and arguing as above, we see that the $R$\+modules
$\Delta_{R,T_\alpha}(C)$ are simply right obtainable from $R$\+modules
that are $S_j$\+contramodules for $1\le j\le n-1$, annihilated by
the elements~$s_k$ for $n+1\le k\le m$, and also annihilated by some
elements $s\in T_\alpha\subset S_n$.
 Once again, it remains to set $s_n=s$ and use the induction assumption.
\end{proof}

\begin{rem}
 Notice that no products of multiplicative subsets $S_j\subset R$, \
$j=1$,~\dots,~$m$ are involved in the formulation of
Lemma~\ref{several-multsubsets-contramodule-obtainable}, though
they are necessary in Proposition~\ref{several-multsubsets-prop}
and Theorem~\ref{several-multsubsets-thm}.
 The following observations may help the reader feel more comfortable:
if $T\subset S\subset R$ are two embedded multiplicative subsets,
then any $T$\+contramodule $R$\+module is an $S$\+contramodule
$R$\+module.
 Indeed, $S^{-1}R$ is a $T^{-1}R$\+module, so~\cite[Lemma~1.2]{PMat}
applies.
 On the other hand, a $T$\+weakly cotorsion $R$\+module does not need
to be $S$\+weakly cotorsion (consider the case of a trivial
multiplicative subset $T=\{1\}$).
 Cf.~\cite[Remark~5.2]{Pcta}.

 Furthermore, for a pair of multiplicative subsets $S$ and $T$ in
a commutative ring $R$, an $S$\+weakly cotorsion and $T$\+weakly
cotorsion $R$\+module does not have to be $ST$\+weakly cotorsion.
 Indeed, otherwise the $R$\+module $(ST)^{-1}R$ would be a direct
summand in a transfinitely iterated extension of copies of
the three $R$\+modules $R$, \ $S^{-1}R$, and $T^{-1}R$.
 However, for any such transfinitely iterated extension $F$ and
a nonzero element~$x\in F$ there exists an $R$\+submodule $G\subset F$,
\ $x\in G$ and an $R$\+module morphism from $G$ into either $R$,
or $S^{-1}R$, or $T^{-1}R$ taking~$x$ to a nonzero element.
 In addition, the quotient module $F/G$ is also a transfinitely iterated
extension of the three $R$\+modules $R$, \ $S^{-1}R$, and $T^{-1}R$.
 For example, consider the case of the ring of integers $R=\boZ$
with the multiplicative subsets $S=\{p^n\mid n\in\boZ_{\ge0}\}$
and $T=\{q^n\mid n\in\boZ_{\ge0}\}$, where $p$ and $q$ are two
distinct prime numbers.
 Then no transfinitely iterated extension of the three
$\boZ$\+modules $\boZ$, \ $\boZ[p^{-1}]$, and $\boZ[q^{-1}]$
contains elements infinitely divisible by~$pq$.
\end{rem}

\begin{lem} \label{localization-torsion-remains-bounded}
 Let $R$ be a commutative ring and $S$, $T\subset R$ be two
multiplicative subsets.
 Denote by $S'\subset T^{-1}R$ the image of $S$ under the localization
morphism $R\rarrow T^{-1}R$.
 Assume that the $S$\+torsion in $R$ is bounded.
 Then the $S'$\+torsion in $T^{-1}R$ is bounded.
\end{lem}

\begin{proof}
 Let $s_0\in S$ be an element such that $s_0\Gamma_S(R)=0$.
 Denote by $s_0'\in S'$ the image of the element $s_0\in R$ under
the ring homomorphism $R\rarrow T^{-1}R$.
 Then $s_0'\Gamma_{S'}(T^{-1}R)=0$.
 Indeed, if $s'r/t=0$ in $T^{-1}R$ for some $s'\in S'$, $r\in R$, and
$t\in T$, then, denoting by $s\in S$ a preimage of the element~$s'$,
there exists an element $u\in T$ such that $sru=0$ in~$R$.
 It follows that $s_0ru=0$ in $R$, hence $s'_0r/t=0$ in~$T^{-1}R$.
\end{proof}

\begin{lem} \label{hom-is-weakly-cotorsion}
 Let $R$ be a commutative ring and $S$, $T\subset R$ be two
multiplicative subsets.
 Assume that an $R$\+module $C$ is $ST$\+weakly cotorsion.
 Then the $R$\+module\/ $\Hom_R(T^{-1}R,C)$ is $S$\+weakly cotorsion.
\end{lem}

\begin{proof}
 According to~\cite[proof of Lemma~8.6]{PS}, for any flat $R$\+modules
$F$ and $G$ the $R$\+module $\Ext^1_R(F,\Hom_R(G,C))$ is a submodule
in the $R$\+module $\Ext_R^1(F\ot_RG,\>C)$.
 In particular, the $R$\+module $\Ext_R^1(S^{-1}R,\.\Hom_R(T^{-1}R,C))$ is
a submodule in the $R$\+module $\Ext_R^1((ST)^{-1}R,\.C)$, so
the former vanishes whenever the latter does.
\end{proof}

 Now we return to the notation of
Section~\ref{several-multsubsets-introd}.
 Suppose that we have a commutative ring $R$ with several multiplicative
subsets $S_1$,~\dots, $S_m\subset R$.
 For any subset of indices $J\subset\{1,\dotsc,m\}$ we denote by
$S_J$ the multiplicative subset $\prod_{j\in J}S_j\subset R$.
 The collection of $m$~multiplicative subsets $\{S_j\mid 1\le j\le m\}$
is denoted by~$\S$, and the collection of $2^m$~multiplicative subsets
$\{S_J\mid J\subset\{1,\dotsc,m\}\.\}$ is denoted by~$\S^\times$.

 Given a subset of indices $J\subset\{1,\dotsc,m\}$, we denote its
complement by $K=\{1,\dotsc,m\}\setminus J$.
 The symbol~$\s$ denotes an arbitrary collection of elements
$s_k\in S_k$, \ $k\in K$, and the ring $R_{J,\s}$ is defined as
the quotient ring of $S_J^{-1}R$ by the ideal generated by
the elements~$s_k$, \ $k\in K$.

\begin{proof}[Proof of Proposition~\ref{several-multsubsets-prop}]
 This proof is a generalization of the proof
of~\cite[Main Proposition~8.1]{PS}, and the reader may wish to
glance into the discussion in~\cite{PS} for motivation related to
the inductive argument below.

 We will prove the following assertion by induction on a pair of
integers $0\le n\le m$: any $R$\+module that is $S_J$\+weakly
cotorsion for all $J\subset\{1,\dotsc,n\}$ and
an $S_k$\+contramodule for all $n+1\le k\le m$ is right $1$\+obtainable
from $R_{J,\s}$\+modules, where $J\subset\{1,\dotsc,n\}\subset
\{1,\dotsc,m\}$ and $\s$~runs over all the collections of elements
$s_k\in S_k$, \ $k\in K=\{1,\dotsc,m\}\setminus J$.
 For $n=m$, this is the ``only if'' assertion of the proposition.

 Induction base: for $n=0$, this is the assertion of
Lemma~\ref{several-multsubsets-contramodule-obtainable}.

 Let $C$ be an $R$\+module that is $S_J$\+weakly cotorsion for all
$J\subset\{1,\dotsc,n\}$ and an $S_k$\+contramodule for all
$n+1\le k\le m$.
 Consider the exact sequence~\eqref{weakly-cotorsion-sequence}
(from Lemma~\ref{matlis-exact-sequence}) for the multiplicative
subset $S=S_n$:
\begin{equation}
 0\lrarrow\Hom_R(S_n^{-1}R/R,C)\lrarrow\Hom_R(S_n^{-1}R,C)\lrarrow C
 \lrarrow\Delta_{R,S_n}(C)\lrarrow0.
\end{equation}
 Denote by $\sE_n\subset R\modl$ the class of all $R_{J,\s}$\+modules
with $J\subset\{1,\dotsc,n\}\subset\{1,\dotsc,m\}$.
 To show that $C$ is right $1$\+obtainable from $\sE_n$, it suffices
to check that the $R$\+modules $\Hom_R(S_n^{-1}R,C)$ and
$\Delta_{R,S_n}(C)$ are right $1$\+obtainable from $\sE_n$ and
the $R$\+module $\Hom_R(S_n^{-1}R/R,C)$ is right $2$\+obtainable
from~$\sE_n$.
 Let us consider these three $R$\+modules one by one.

 The $R$\+module $\Hom_R(S_n^{-1}R,C)$ is an $S_n^{-1}R$\+module.
 By Lemma~\ref{hom-into-contramodule-is-contramodule}(b), it is also
an $S_k$\+contramodule for all $n+1\le k\le m$, and by 
Lemma~\ref{hom-is-weakly-cotorsion}, it is $S_J$\+weakly cotorsion
for all $J\subset\{1,\dotsc,n-1\}$.
 In order to conclude that the $R$\+module $\Hom_R(S_n^{-1}R,C)$ is
right $1$\+obtainable from $\sE_n$, we apply the induction
assumption to the pair of integers $(n',m')=(n-1,m-1)$, the ring
$R'=S_n^{-1}R$, and the collection of $m-1$ multiplicative subsets
$S_1'$,~\dots, $S_{n-1}'$, $S_{n+1}'$,~\dots, $S'_m\subset R'$ obtained
as the images of the multiplicative subsets $S_1$,~\dots, $S_{n-1}$,
$S_{n+1}$,~\dots, $S_m\subset R$ under the localization morphism
$R\rarrow S_n^{-1}R$.

 It needs to be observed that, if the multiplicative subset
$S_j\subset R$ is countable, then the multiplicative subset
$S_j'\subset R'$ is countable; if $\pd_RS_j^{-1}R\le1$, then
$\pd_{R'}S_j'{}^{-1}R'\le1$ by Lemma~\ref{change-of-ring-multsubset}(a);
and if the $S_j$\+torsion in $R$ is bounded, then the $S_j'$\+torsion
in $R'$ is bounded by Lemma~\ref{localization-torsion-remains-bounded}.
 The $R'$\+module $\Hom_R(S_n^{-1}R,C)$ is $S'_J$\+weakly cotorsion
for all $J\subset\{1,\dotsc,n-1\}$ and an $S'_k$\+contramodule for
all $n+1\le k\le m$ by Lemma~\ref{change-of-ring-multsubset}(b\+c).
 Finally, for any subset of indices $J'\subset\{1,\dotsc,n-1\}
\subset\{1,\dotsc,\allowbreak n-1,n+1,\dotsc,m\}$ and any
collection~$\s'$ of elements $s'_{n+1}\in S'_{n+1}$,~\dots,
$s'_m\in S'_m$, the $R$\+algebra $R'_{J',\s'}$ is isomorphic to $R_{J,s}$,
where $J=J'\cup\{n\}\subset\{1,\dotsc,n\}\subset\{1,\dotsc,m\}$ is
the related subset of indices and $\s$~is a collection of preimages
$s_{n+1}\in S_{n+1}$,~\dots, $s_m\in S_m$ of the elements $s'_k\in S_k'$,
\ $n+1\le k\le m$.

 The $R$\+module $\Delta_{R,S_n}(C)$ is $S_J$\+weakly cotorsion for all
$J\subset\{1,\dotsc,n-1\}$, since it is a quotient module of
an $S_J$\+weakly cotorsion $R$\+module $C$ and $\pd_RS_J^{-1}R\le1$.
 It is also an $S_n$\+contramodule by
Lemma~\ref{S-contramodule-category-lem}(b), and
an $S_k$\+contramodule for all $n+1\le k\le m$ by
Lemma~\ref{hom-into-contramodule-is-contramodule}(b).
 So $\Delta_{R,S_n}(C)$ is an $S_k$\+contramodule $R$\+module for all
$n\le k\le m$.
 The claim that the $R$\+module $\Delta_{R,S_n}(C)$ is right
$1$\+obtainable from~$\sE_n$ (in fact, even from~$\sE_{n-1}$) now
follows from the induction assumption applied to the pair of
integers $(n',m')=(n-1,m)$, the same ring $R'=R$, and the same
collection of multiplicative subsets $S_1'=S_1$,~\dots, $S_m'=S_m$.

 The $R$\+module $\Hom_R(S_n^{-1}R/R,C)$ is an $S_n$\+contramodule by
Lemma~\ref{hom-from-torsion-is-contramodule} and an $S_k$\+contramodule
for all $n+1\le k\le m$ by
Lemma~\ref{hom-into-contramodule-is-contramodule}(b).
 By Lemma~\ref{several-multsubsets-contramodule-obtainable},
the $R$\+module $\Hom_R(S_n^{-1}R/R,C)$ is simply right obtainable
from modules over the quotient rings $R''=R/(s_nR+s_{n+1}R+\dotsb+s_mR)$
for various combinations of elements $s_n\in S_n$, \
$s_{n+1}\in S_{n+1}$,~\dots, $s_m\in S_m$.

 Denote by $S''_J\subset R''$ the images of the multiplicative subsets
$S_J\subset R$, \ $J\subset\{1,\dotsc,n-1\}$ in the ring~$R''$.
 By Lemma~\ref{change-of-ring-multsubset}(a), we have
$\pd_{R''}S_J^{\prime\prime\.-1}R''\le1$.
 Therefore, by Lemma~\ref{2-obtainable-lem}, all $R''$\+modules are
right $2$\+obtainable from $R''$\+modules that are $S''_J$\+weakly
cotorsion for all $J\subset\{1,\dotsc,n-1\}$. 
 Applying Lemma~\ref{change-of-ring-multsubset}(b), we see that
all such $R''$\+modules are $S_J$\+weakly cotorsion as $R$\+modules.
 Thus all $R$\+modules annihilated by $s_nR+\dotsb+s_mR$ are right
$2$\+obtainable from $R$\+modules annihilated by $s_nR+\dotsb+s_mR$
that are $S_J$\+weakly cotorsion for all $J\subset\{1,\dotsc,n-1\}$.
 Any $R$\+module annihilated by an element $s_k\in S_k$ is
an $S_k$\+contramodule.

 Now we apply the induction assumption to the pair of integers
$(n',m')=(n-1,m)$, the ring $R'=R$, and the collection of
multiplicative subsets $S_1'=S_1$,~\dots, $S_m'=S_m$ in order to
conclude that all $R$\+modules annihilated by $s_nR+\dotsb+s_mR$, \
$s_k\in S_k$, are right $2$\+obtainable from~$\sE_{n-1}$.
 It follows that the $R$\+module $\Hom_R(S_n^{-1}R/R,C)$ is also
right $2$\+obtainable from~$\sE_{n-1}$.

 This proves the ``only if'' assertion of the proposition.
 The ``if'' assertion holds for any collection of multiplicative
subsets $S_1$,~\dots, $S_m$ in a commutative ring~$R$.
 To prove as much, let $\sE\subset R\modl$ denote the class of all
$R_{J,\s}$\+modules, viewed as $R$\+modules via the restriction of
scalars, where $J\subset\{1,\dotsc,m\}$ and $\s$ runs over all
the collections of elements $s_k\in S_k$, \ $k\in K=\{1,\dotsc,m\}
\setminus J$.

 By Lemma~\ref{change-of-scalars-ext}, for any $R$\+module $E\in\sE$
and any subset of indices $I\subset\{1,\dotsc,m\}$, we have
$\Ext_R^i(S_I^{-1}R,E)=0$ for all $i>0$, since $S_I^{-1}R$ is a flat
$R$\+module and
$$
 R_{J,\s}\ot_R S_I^{-1}R = \begin{cases}
 R_{J,\s} &\text{if $I\subset J$,} \\
 0,   &\text{if $I\not\subset J$} \end{cases} 
$$
is always a free $R_{J,\s}$\+module (with $1$ or~$0$ generators).
 It remains to apply Lemma~\ref{obtainable-orthogonal-lemma} in order
to conclude that $\Ext_R^1(S_I^{-1}R,C)=0$ for all $R$\+modules~$C$
right $1$\+obtainable from~$\sE$.
\end{proof}

\begin{proof}[Proof of Theorem~\ref{several-multsubsets-thm}]
 The ``only if'' assertion holds for any collection of multiplicative
subsets $S_1$,~\dots, $S_m$ in a commutative ring~$R$.
 Indeed, an $R$\+module is $\S^\times$\+strongly flat if and only if
it is a direct summand of a transfinitely iterated extension, in
the sense of the inductive limit, of copies of the $R$\+modules
$S_I^{-1}R$, \ $I\subset\{1,\dotsc,m\}$ \cite[Corollary~6.14]{GT}.
 Now, for any subset $J\subset\{1,\dotsc,m\}$ and any collection~$\s$
of elements $s_k\in S_k$, \ $k\in\{1,\dotsc,m\}\setminus J$,
the functor of extension of scalars $R_{J,\s}\ot_R{-}\:R\modl\rarrow
R_{J,\s}\modl$ takes transfinitely iterated extensions (in the sense
of the inductive limit) of flat $R$\+modules to transfinitely iterated
exensions of flat $R_{J,\s}$\+modules, and the $R$\+modules $S_I^{-1}R$
to free $R_{J,\s}$\+modules (with $1$ or~$0$ generators).

 To prove the ``if'', denote, as above, by $\sE\subset R\modl$ the class
of all $R_{J,\s}$\+modules (viewed as $R$\+modules via the restriction
of scalars).
 Let $F$ be a flat $R$\+module such that the $R_{J,\s}$\+module
$R_{J,\s}\ot_RF$ is projective for all subsets of indices
$J\subset\{1,\dots,m\}$ and all collections~$\s$ of elements
$s_k\in S_k$, \ $k\in\{1,\dotsc,m\}\setminus J$.
 Then, by Lemma~\ref{change-of-scalars-ext}, we have $\Ext_R^i(F,E)=0$
for all $E\in\sE$ and all $i>0$.

 Applying Lemma~\ref{obtainable-orthogonal-lemma}, we see that
$\Ext_R^i(F,C)=0$ for all $R$\+modules $C$ right $1$\+obtainable from
$\sE$ and all $i>0$.
 According to Proposition~\ref{several-multsubsets-prop}, all
$\S^\times$\+weakly cotorsion $R$\+modules are right $1$\+obtainable
from~$\sE$.
 Thus $\Ext_R^1(F,C)=0$ for all $\S^\times$\+weakly cotorsion
$R$\+modules $C$, so $F$ is an $\S^\times$\+strongly flat $R$\+module.
\end{proof}

\Section{Finite-Dimensional Noetherian Rings with Countable Spectrum}
\label{finite-dimensional-secn}

 The aim of this section is to prove
Theorems~\ref{polynomials-over-pid-thm},
\ref{two-dimensional-ring-thm}, and~\ref{finite-dimensional-ring-thm}.
 The essential idea of these arguments was already stated in
Section~\ref{finite-dimensional-introd}; the following corollary,
formulated in the notation of Section~\ref{several-multsubsets-introd},
summarizes it.

\begin{cor} \label{artinian-rings-r-j-s-cor}
 Let\/ $\S=\{S_1,\dotsc,S_m\}$ be a finite collection of countable
multiplicative subsets in a commutative ring~$R$.
 Assume that, for any subset of indices $J\subset\{1,\dotsc,m\}$ and
any collection of elements $s_k\in S_k$, \ $k\in K=\{1,\dotsc,m\}
\setminus J$, the ring $R_{J,\s}$ is Artinian.
 Then all flat $R$\+modules are\/ $\S^\times$\+strongly flat.
\end{cor}

\begin{proof}
 This is a corollary of Theorem~\ref{several-multsubsets-thm}.
 Since we assume the multiplicative subsets $S_1$,~\dots, $S_m\subset R$
to be countable, the condition that the projective dimension of
the $R$\+module $S_J^{-1}R$ does not exceed~$1$ for all
$J\subset\{1,\dotsc,m\}$ holds automatically.
 All flat modules over an Artinian commutative ring are projective,
so our assumptions guarantee that all flat $R_{J,\s}$\+modules are
projective, and the assertion of the corollary follows immediately
from Theorem~\ref{several-multsubsets-thm}.
\end{proof}

 Let $P$ be a (commutative) principal ideal domain and $R=P[x]$ be
the ring of polynomials in one variable~$x$ with coefficients in~$P$.
 Changing the notation of Section~\ref{finite-dimensional-introd}
slightly, denote by $S\subset P[x]$ the multiplicative subset of
all nonzero elements of $P$, viewed as elements of $P[x]$, and by
$T\subset P[x]$ the multiplicative subset of all polynomials with
content~$1$.

\begin{prop} \label{pid-tensor-products-artinian}
 For any principal ideal domain $P$, and any two elements $s\in S$ and
$t\in T$, the four rings $R/sR\ot_RR/tR$, \ $R/sR\ot_R T^{-1}R$, \
$S^{-1}R\ot_RR/tR$, and $S^{-1}R\ot_R T^{-1}R$ are Artinian.
\end{prop}

\begin{proof}
 The multiplicative subset $ST\subset P[x]$ consists of all
the nonzero elements of the commutative domain $R=P[x]$, so $S^{-1}R
\ot_R T^{-1}R=(ST)^{-1}R$ is a field.
 Furthermore, denoting by $Q$ the field of fractions of
the domain~$P$, one has $S^{-1}R\ot_RR/tR=Q[x]/tQ[x]$, where
$t$ is a nonzero polynomial in one variable~$x$ over the field~$Q$.
 This is obviously an Artinian ring.
 Finally, for any commutative ring $A$ with a finitely generated
nilpotent ideal $\n\subset A$, the ring $A$ is Artinian if and only
if the ring $A/\n$ is.
 Besides, the direct sum of a finite number of Artinian rings is
Artinian.
 This reduces the task of proving that the rings $R/sR\ot_RR/tR$
and $R/sR\ot_RT^{-1}R$ are Artinian to the case of a prime element
$s\in P$.
 In this case, $k=P/sP$ is a field, and the ring $R/sR\ot_RT^{-1}R$
is isomorphic to the field of rational functions $k(x)$, since
every unital polynomial $f(x)=x^n+f_{n-1}x^{n-1}+\dotsb+f_0$ in
the variable~$x$ with the coefficients in~$k$ can be lifted
to a unital polynomial in~$x$ with the coefficients in~$P$, which
belongs to~$T$.
 On the other hand, any polynomial $t=t(x)\in T$ reduces to
a nonzero polynomial $g(x)\in k[x]$ modulo~$s$, so the ring
$R/sR\ot_RR/tR$ is isomorphic to the quotient ring $k[x]/gk[x]$,
which is Artinian.
\end{proof}

\begin{proof}[Proof of Theorem~\ref{polynomials-over-pid-thm}]
 When the principal ideal domain $P$ is countable, so are both
the multiplicative subsets $S_1=S$ and $S_2=T\subset R=P[x]$.
 According to Proposition~\ref{pid-tensor-products-artinian}, all
the rings $R_{J,\s}$, \,$J\subset\{1,2\}$, are Artinian, so
Corollary~\ref{artinian-rings-r-j-s-cor} is applicable.
\end{proof}

\begin{rem}
 A commutative algebra $R'$ over a commutative ring $R$ is said to be
\emph{finite} (over~$R$) if it is finitely generated \emph{as
an $R$\+module}.
 Let $P$ be a countable PID and $R'$ be a finite commutative algebra
over the ring $R=P[x]$.
 Denote by $S'_1=S'$ and $S_2'=T'\subset R'$ the images of our two
multiplicative subsets $S_1=S$ and $S_2=T\subset P[x]$ in
the ring~$R'$.
 Then, for any elements $s'\in S'$ and $t'\in T'$, the four rings
$R'/s'R'\ot_{R'}R'/t'R'$, \ $R'/s'R'\ot_{R'}T'{}^{-1}R'$, \
$S'{}^{-1}R'\ot_{R'}R'/t'R'$, and $S'{}^{-1}R'\ot_{R'}T'{}^{-1}R'$ are
Artinian, because the similar rings related to the multiplicative
subsets $S$ and $T\subset R$ are Artinian by
Proposition~\ref{pid-tensor-products-artinian} and a commutative ring
finite over an Artinian commutative ring in Artinian.

 This provides an elementary explicit construction of pairs of
multiplicative subsets $\S'=\{S_1',S_2'\}$ in some countable
Noetherian commutative rings $R'$ of Krull dimension~$2$ such that
all flat $R'$\+modules are $\S'{}^\times$\+strongly flat.
 In particular, if $R'$ is a finitely generated commutative algebra
of transcendence degree (\,$=$~Krull dimension)~$2$ over a countable
field~$k$, then, by the Noether normalization lemma, one can find
two elements $x$, $y\in R'$ such that $R'$ is finite over its
$k$\+subalgebra generated by $x$ and~$y$, which is
the polynomial algebra $k[x,y]\subset R'$.
 Applying the above construction to the principal ideal domain
$P=k[y]$ and the ring $R=P[x]=k[x,y]$, we obtain a pair of
multiplicative subsets $S'_1$, $S'_2\subset R'$ with the desired
properties.
\end{rem}

 Let $R$ be a commutative ring, $S\subset R$ be a multiplicative
subset, and $\p\varsubsetneq\q$ be two prime ideals in $R$, one of
which is properly contained in the other one.
 We will say that the multiplicative subset $S$
\emph{distinguishes}~$\p$ from~$\q$ if $\p\cap S=\varnothing$ and
$\q\cap S\ne\varnothing$.
 We will say that a collection of multiplicative subsets $S_1$,~\dots,
$S_m\subset R$ distinguishes~$\p$ from~$\q$ if there exists
$1\le j\le m$ such that $S_j$ distinguishes~$\p$ from~$\q$.

\begin{lem} \label{distinguishing-prime-ideals-lemma}
 Let $R$ be a Noetherian commutative ring and $S_1$,~\dots,
$S_m\subset R$ be a collection of multiplicative subsets in~$R$.
 Then the ring $R_{J,\s}$ is Artinian for every subset of indices
$J\subset\{1,\dotsc,m\}$ and every collection of elements $s_k\in S_k$,
\ $k\in\{1,\dotsc,m\} \setminus J$ if and only if for every pair of
prime ideals\/ $\p\varsubsetneq\q$ in $R$ there exists\/ $1\le j\le m$
for which\/ $\p\cap S_j=\varnothing$, while $\q\cap S_j\ne\varnothing$.
\end{lem}

\begin{proof}
 For any commutative ring $R$ with a collection of multiplicative
subsets $S_1$,~\dots, $S_m\subset R$, and any $J\subset\{1,\dotsc,m\}$
and $s_k\in S_k$, the map of spectra $\Spec R_{J,\s}\rarrow\Spec R$
induced by the commutative ring homomorphism $R\rarrow R_{J,\s}$
identifies the spectrum of the ring $R_{J,\s}$ with the subset in 
$\Spec R$ consisting of all the prime ideals $\p\subset R$ such that
$\p\cap S_j=\varnothing$ for all $j\in J$ and $s_k\in\p$ for all
$k\notin J$.
 If the collection of multiplicative subsets $S_1$,~\dots, $S_m$
distinguishes all the pairs of prime ideals one of which is properly
contained in the other one in the ring $R$, the above conditions cannot
be simultaneously satisfied for two such prime ideals
$\p\varsubsetneq\q\subset R$.
 Hence the Krull dimension of the ring $R_{J,\s}$ is equal to~$0$.
 If the ring $R$ is Noetherian, then so is the ring $R_{J,\s}$; and
it follows that $R_{J,\s}$ is Artinian.

 Conversely, let $\p\varsubsetneq\q$ be a pair of prime
ideals in $R$ not distinguished by the multiplicative subsets
$S_1$,~\dots,~$S_m$.
 Denote by $J$ the set of all indices $1\le j\le m$ for which
$\q\cap S_j=\varnothing$.
 For every $k\in\{1,\dotsc,m\}\setminus J$, the intersection
$\p\cap S_k$ is nonempty, so the exists an element $s_k\in\p\cap S_k$.
 Consider the related ring $R_{J,\s}$.
 Then both the prime ideals $\p$ and $\q\subset R$ belong to
the image of the map $\Spec R_{J,\s}\rarrow\Spec R$, so they correspond
to a pair of prime ideals one of which is properly contained in
the other one in the ring~$R_{J,\s}$.
 Consequently, the ring $R_{J,\s}$ has positive Krull dimension and is
not Artinian.
\end{proof}

 The next lemma is a simple elementary version of the one following
after it.

\begin{lem} \label{one-dimensional-ring-lemma}
 Let $R$ be a Noetherian ring of Krull dimension~$1$.
 Then there exists a multiplicative subset $S\subset R$ such that
all the rings $R/sR$, \,$s\in S$, and $S^{-1}R$ are Artinian.
 If the spectrum of $R$ is countable, one can choose $S$ to be
a countable multiplicative subset in~$R$.
\end{lem}

\begin{proof}
 The point is that the set of all minimal prime ideals in
a Noetherian commutative ring is finite.
 To obtain just some multiplicative subset $S\subset R$ for which
the rings $R/sR$ and $S^{-1}R$ are Artinian, it suffices to set
$S=R\setminus\bigcup_{i=1}^k\q_i$, where $\q_i$ are the minimal prime
ideals in~$R$.
 Then, by prime avoidance, for any prime ideal $\p\subset R$ of
height~$1$ there exists an element $s\in \p\cap S$, so
the multiplicative subset $S\subset R$ distinguishes all the pairs of
prime ideals one of which is properly contained in the other one
in the ring~$R$ (cf.~\cite[Section~13]{Pcta}).
 When $\Spec R$ is countable, one can choose, for every prime ideal
$\p\subset R$ of height~$1$, an element $s\in R$ not belonging to
any of the minimal prime ideals, and set $S$ to be the multiplicative
subset generated by all the elements~$s$ so obtained.
 Then $S$ is countable and all the pairs of prime ideals one of which
is properly contained in the other one in $R$ are still distinguished
by $S$, so the rings $R/sR$, \,$s\in S$, and $S^{-1}R$ are Artinian.
\end{proof}

\begin{lem} \label{two-dimensional-ring-construction-lemma}
 Let $R$ be a Noetherian commutative ring of Krull dimension~$2$ with
countable spectrum.
 Then there exists a pair of countable multiplicative subsets
$S$ and $T\subset R$ such that, for any two elements $s\in S$ and
$t\in T$, the four rings $R/sR\ot_RR/tR$, \ $R/sR\ot_R T^{-1}R$, \
$S^{-1}R\ot_RR/tR$, and $S^{-1}R\ot_R T^{-1}R$ are Artinian.
\end{lem}

\begin{proof}
 Let $\p_1$, $\p_2$,~\dots\ be all the prime ideals of the ring $R$,
numbered by the positive integers in an arbitrary order.
 We proceed by induction, constructing two nondecreasing sequences
of finite subsets of elements $S'_0\subset S'_1\subset S'_2\subset
\dotsb\subset R$ and $T'_0\subset T'_1\subset T'_2\subset\dotsb
\subset R$ in the following way.

 Throughout the construction, the following conditions will be always
satisfied.
 Prime ideals of height~$0$ in $R$ do not intersect the subsets
$S_n'$ and $T'_n\subset R$.
 A prime ideal of height~$1$ may intersect either $S'_n$, or $T'_n$,
but not both.
 Prime ideals of height~$2$ are allowed to intersect both $S'_n$
and~$T'_n$.

 For the purposes of this proof, given a finite subset $U=S_n'$
or $T_n'\subset R$, we will say that a prime ideal $\p\subset R$ of
height~$1$ is a \emph{$U$\+prime} if $\p$ intersects~$U$.
 Since the primes of height~$0$ do not intersect $U$, such a prime $\p$
is a minimal prime over the principal ideal $(u)\subset R$ generated
by some element $u\in U$.
 Since the set $U$ is finite and the set of all minimal prime ideals
in a Noetherian commutative ring is finite, the set of all
$U$\+primes in $R$ will be always finite.
 (Notice that, by Krull's Hauptidealsatz, all minimal primes over
a principal ideal in $R$ have height at most~$1$; but we do not need
to use this result, as all our ``$U$\+primes'' have height~$1$
by definition.)

 To begin with, set $S'_0=\varnothing=T'_0$.
 For any $n\ge1$, suppose that the subsets $S'_{n-1}$ and $T'_{n-1}
\subset R$ have been constructed already, and consider the prime ideal
$\p_n\subset R$.
 If $\p_n$ is a prime ideal of height~$0$, we set $S'_n=S'_{n-1}$ and
$T'_n=T'_{n-1}$.

 If $\p_n$~is a prime ideal of height~$1$, it may intersect one of
the two subsets $S'_{n-1}$ and $T'_{n-1}$, or it may intersect
neither of them.
 In the former case, we set $S'_n=S'_{n-1}$ and $T'_n=T'_{n-1}$.
 In the latter case, we use prime avoidance to find an element
$s\in\p_n$ not belonging to any of the $T'_{n-1}$\+primes and any
of the primes of height~$0$ in~$R$.
 Set $S'_n=S'_{n-1}\cup\{s\}$ and $T'_n=T_{n-1}$.

 If $\p_n$~is a prime ideal of height~$2$, it may intersect both
the subsets $S'_{n-1}$ and $T'_{n-1}\subset R$, or only one of them,
or neither one.
 In any case, we construct the sets $S_n'$ and $T_n'$ so that
$\p_n$~intersects both of them.
 Specifically, if $\p_n\cap S'_{n-1}=\varnothing$, we find an element
$s\in\p_n$ not belonging to any of the $T'_{n-1}$\+primes and any of
the primes of height~$0$ in $R$, and set $S'_n=S'_{n-1}\cup\{s\}$.
 If $\p_n\cap S'_{n-1}\ne\varnothing$, we set $S_n'=S'_{n-1}$.
 Then, if $\p_n\cap T'_{n-1}=\varnothing$, we find an element $t\in\p_n$
not belonging to any of the $S'_n$\+primes and any of the primes of
height~$0$ in $R$, and set $T'_n=T'_{n-1}\cup\{t\}$.
 If $\p_n\cap T'_{n-1}\ne\varnothing$, we set $T_n'=T'_{n-1}$.

 After the whole inductive process has been finished, we set
$S'=\bigcup_{n=1}^\infty S'_n$ and $T'=\bigcup_{n=1}^\infty T'_n$, and
observe that any prime ideal $\p\subset R$ of height~$h$ intersects
exactly $h$~of the two subsets $S'$ and $T'\subset R$.
 Finally, we set $S\subset R$ to be the multiplicative subset
generated by $S'$ and $T\subset R$ to be the multiplicative subset
generated by $T'$ in $R$, and again observe that any prime ideal of
height~$h$ in $R$ intersects exactly $h$~of the two multiplicative
subsets $S$ and $T\subset R$.

 Applying Lemma~\ref{distinguishing-prime-ideals-lemma}, we conclude
that the pair of countable multiplicative subsets $S$ and $T\subset R$
has the desired properties.
\end{proof}

\begin{proof}[Proof of Theorem~\ref{two-dimensional-ring-thm}]
 Denote the two multiplicative subsets $S$ and $T$ constructed in
Lemma~\ref{two-dimensional-ring-construction-lemma} by $S_1=S$ and
$S_2=T$, and apply Corollary~\ref{artinian-rings-r-j-s-cor}.
\end{proof}

\begin{rem}
 One might wish to extend the construction of
Lemma~\ref{two-dimensional-ring-construction-lemma} to Noetherian
rings $R$ of Krull dimension~$3$ with countable spectrum by
producing a collection of three multiplicative subsets $S$, $T$,
and $U\subset R$ such that any prime ideal of height~$h$ in $R$
intersects exactly $h$~of them.
 However, the most straightforward attempt to do so breaks down on
the following problem.
 Arguing as in the proof of
Lemma~\ref{two-dimensional-ring-construction-lemma}, suppose that
we have a prime ideal $\p_n\subset R$ of height~$1$ which does not
intersect any of the three sets $S'_{n-1}$, $T'_{n-1}$, and
$U'_{n-1}$.
 We need to add an element of the ideal~$\p_n$ to one of these three
sets in order to produce a new set $S'_n$, $T'_n$, or~$U'_n$ that would
interstect~$\p_n$.
 The problem occurs when the prime ideal~$\p_n$ is contained in
three prime ideals of height~$2$, say $\q'$, $\q''$, and~$\q'''$,
such that $\q'$~intersects $S'_{n-1}$ and $T'_{n-1}$, while
$\q''$~intersects $S'_{n-1}$ and $U'_{n-1}$, and at the same time
$\q'''$~intersects $T'_{n-1}$ and $U'_{n-1}$.
 Then, into whichever of the three sets $S'_n$, $T'_n$, or $U'_n$ we
decide to include an element of the prime ideal~$\p_n$, the condition
that any prime ideal of height~$2$ intersects at most two of these
three sets would be violated.

 Another approach might be to do several waves of the induction process,
treating all the prime ideals of height~$1$ before starting with
the prime ideals of height~$2$.
 But then the sets $S'$, $T'$, and $U'$ produced after an infinite
number of such steps become infinite themselves, and the argument with
prime avoidance, based on there being only a finite number of prime
ideals of height~$1$ intersecting either of these sets, would no
longer work (of course).
 Yet another approach, which we successfully develop below, is to use
more than~$d$ multiplicative subsets for a Noetherian ring of Krull
dimension~$d$ with countable spectrum.

 In fact, more than~$d$ multiplicative subsets for a Noetherian ring
of Krull dimension~$d$ are necessary in some cases.
 Let $R=k[x,y,z]$ be the ring of polynomials in three variables over
a countable field~$k$, or alternatively, $R=\boZ[x,y]$ be the ring
of polynomials in two variables with integer coefficients.
 Then the intersection of any finite set of prime ideals of
height~$2$ in $R$ contains infinitely many primes of height~$1$
\cite[Corollary~11]{McA}, and in particular, the intersection of any
three prime ideals of height~$2$ contains a prime ideal of height~$1$.
 For this reason, one cannot find $3$~multiplicative subsets in $R$
such that any prime ideal of height~$h$ intersects exactly~$h$ of them.
 So the result of
Proposition~\ref{mu-of-d-multsubsets-construction-prop} below, providing
$4$~multiplicative subsets in a Noetherian ring of Krull dimension~$3$
with countable spectrum, is really the best one could hope for in
these cases.
\end{rem}

 In the sequel, when we say that ``a collection of multiplicative
subsets $S_1$,~\dots, $S_m\subset R$ in a commutative ring $R$
distinguishes all prime ideals belonging to a certain set of primes
$P\subset \Spec R$ from all prime ideals belonging to a certain other
set of primes $Q\subset\Spec R$\,'', we mean that every pair of prime
ideals $\p\varsubsetneq\q$ with $\p$ and~$\q$ belonging to
the respective sets of prime ideals in $R$ is distinguished by
the collection of multiplicative subsets $S_1$,~\dots, $S_m\subset R$.
 So, ``distinguishing prime ideals'' always means distinguishing pairs
of prime ideals one of which is properly contained in the other one.

\begin{lem} \label{height-l-and-lminusone-prime-ideals}
 Let $R$ be a Noetherian commutative ring with countable spectrum,
and let $l\ge2$ be an integer.
 Then there exists a collection of $l$~multiplicative subsets
$S_1$,~\dots, $S_l\subset R$ distinguishing all the prime ideals of
height~$l$ and $l-1$ in $R$ from each other, and from all the prime
ideals of smaller height.
\end{lem}

\begin{proof}
 The assertion of the lemma means that, for every pair of
prime ideals $\p\varsubsetneq\q$ in $R$ such that the height
of~$\q$ is equal to $l$ or~$l-1$, there exists $1\le j\le l$ for
which $S_j\cap\p=\varnothing$ and $S_j\cap\q\ne\varnothing$.

 The proof is similar to that of
Lemma~\ref{two-dimensional-ring-construction-lemma}.
 Let $\p_1$, $\p_2$,~\dots\ be all the prime ideals of height~$l$
or~$l-1$ in the ring $R$, numbered by the positive integers in
an arbitrary order.
 We proceed by induction, constructing $l$~nondecreasing sequences
of finite subsets of elements $S'_{j,0}\subset S'_{j,1}\subset S'_{j,2}
\subset\dotsb$, \ $1\le j\le l$ in the ring $R$ in the way
described below.

 Throughout the construction, the following condition will be always
satisfied.
 Any prime ideal of height~$h$, \ $0\le h\le l$ in the ring $R$ may
intersect at most~$h$ of the $l$~finite sets $S'_{1,n}$,~\dots,
$S'_{l,n}\subset R$ (but not more).

 We will denote the collections of $l$~sets $S'_{1,n}$,~\dots,
$S'_{l,n}$ by $\S'_n$ for brefity, and say that a prime ideal
$\p\subset R$ of height~$h$ is an \emph{$\S'_n$\+prime} if
$\p$~intersects $h$~of the $l$~sets $S'_{1,n}$,~\dots, $S'_{l,n}$.
 Since a prime ideal of height less than~$h$ cannot intersect
$h$~of these sets by assumption, such a prime~$\p$ is a minimal
prime over the ideal generated by some collection of $h$~elements
$s_j\in S'_{j,n}$, \ $j\in J$, where $J\subset\{1,\dotsc,l\}$ is
a subset of indices of the cardinality $|J|=h$.
 Since all the sets $S'_{1,n}$,~\dots, $S'_{l,n}$ are finite and
the set of all minimal prime ideals in a Noetherian commutative ring
is finite, the set of all $\S'_n$\+primes in $R$ will be always finite.

 To begin with, set $S'_{j,0}=\varnothing$ for all $1\le j\le l$.
 For any $n\ge1$, suppose that the subsets $S'_{j,n-1}\subset R$, \
$1\le j\le l$ have been constructed already; and consider
the minimal element~$i_n$ in the set of all positive integers~$i$
such that the prime ideal $\p_i\subset R$ is not an
$\S'_{n-1}$\+prime.

 Since $\p_{i_n}$ is a prime ideal of height $l$ or~$l-1$, there exists
an index~$k$, \ $1\le k\le l$ such that
$\p_{i_n}\cap S'_{k,n-1}=\varnothing$.
 By prime avoidance, we can find an element $s\in\p_{i_n}$ not belonging
to any of the $\S'_{n-1}$\+primes of height~$\le l-1$ in~$R$.
 Then we set $S'_{k,n}=S'_{k,n-1}\cup\{s\}$ and $S'_{j,n}=S'_{j,n-1}$ for
all $j=1$,~\dots, $k-1$, $k+1$,~\dots,~$l$.

 The sequence of integers~$i_n$ is nondecreasing, $1\le i_1\le i_2\le
i_2\le i_3\le \dotsb$; it may visit a given integer~$i$ more than
once, but at most~$l$ times.
 If for a certain~$n\ge1$ it turns out that all primes of height~$l$
and~$l-1$ are $\S'_{n-1}$\+primes already, we stop the induction
at this point and set $S'_{j,n-1}=S'_{j,n}=S'_{j,n+1}=\dotsb$ for all
$1\le j\le l$.
 Otherwise, the construction proceeds for all $n\ge1$.

 After the whole inductive process has been finished, we set
$S'_j=\bigcup_{n=1}^\infty S'_{j,n}$ for all $j=1$,~\dots,~$l$, and
observe that any prime ideal $\p\subset R$ of height $h\le l-2$
intersects at most~$h$ of the $l$~subsets $S'_1$,~\dots,
$S'_l\subset R$, while any prime ideal $\q\subset R$ of height
$h=l$ or~$l-1$ intersects exactly~$h$ of these $l$~subsets.
 Finally, we set $S_j\subset R$, \ $1\le j\le l$, to be
the multiplicative subset generated by $S'_j$, and again observe
that any prime ideal of height $h\le l-2$ in $R$ intersects at most~$h$
of the multiplicative subsets $S_1$,~\dots, $S_l\subset R$, while
any prime ideal of height $h=l$ or~$l-1$ in $R$ intersects
exactly~$h$ of these $l$~multiplicative subsets.
\end{proof}

 We recall the notation $\mu\:\boZ_{\ge0}\rarrow\boZ_{\ge0}$ for
the function taking a nonnegative integer~$d$ to the nonnegative
integer $\mu(d)=d+(d-2)+(d-4)+\dotsb+(d-2k)$, where $d-2k=0$
or~$1$ (see Section~\ref{finite-dimensional-introd}).

\begin{prop} \label{mu-of-d-multsubsets-construction-prop}
 Let $R$ be a Noetherian commutative ring of finite Krull dimension~$d$
with countable spectrum.
 Then there exists a collection of $m=\mu(d)$ countable multiplicative
subsets $S_1$,~\dots, $S_m\subset R$ such that, for every pair of
prime ideals\/ $\p\varsubsetneq\q$ in $R$ there exists\/ $1\le j\le m$
for which\/ $\p\cap S_j=\varnothing$, while\/
$\q\cap S_j\ne\varnothing$.
\end{prop}

\begin{proof}
 If $d=0$, then there are no pairs of prime ideals one of which is
properly contained in the other one in the ring $R$, so $m=0$ and
an empty collection of multiplicative subsets in $R$ is sufficient.
 If $d=1$, one has $m=1$ and the one countable multiplicative
subset $S\subset R$ constructed in
Lemma~\ref{one-dimensional-ring-lemma} is sufficient.
 If $d=2$, one has $m=2$ and there is a pair of multiplicative subsets
$S$ and $T\subset R$ constructed in
Lemma~\ref{two-dimensional-ring-construction-lemma}.

 Generally for an arbitrary integer~$d\ge0$, let $k\ge0$ be
an integer such that $d-2k=0$ or~$1$.
 We apply Lemma~\ref{height-l-and-lminusone-prime-ideals} \,$k$~times
to the same ring $R$ with the parameter $l=d$, \ $l=d-2$,~\dots,
$l=d-2k+2$.
 This produces the total of $m'=d+(d-2)+\dotsb+(d-2k+2)$ countable
multiplicative subsets $S_1$,~\dots, $S_{m'}\subset R$.

 The first~$d$ of these $m'$~multiplicative subsets distinguish
prime ideals of height~$d$ and~$d-1$ from each other and from
prime ideals of lower height; the next~$d-2$ of these multiplicative
subsets distinguish prime ideals of height~$d-2$ and~$d-3$ from
each other and from prime ideals of lower height, etc.
 The last~$d-2k+2=2$ or~$3$ of these $m'$~multiplicative subsets
distinguish prime ideals of height $d-2k+2$ and~$d-2k+1$ from each
other and from prime ideals of height~$d-2k$ and lower.

 If $d$~is even, then $d-2k=0$, \ $m=m'$, and we are done.
 If $d$~is odd, then $d-2k=1$, \ $m=m'+1$, and we still need one more
multiplicative subset $S_m\subset R$ to distinguish prime ideals
of height~$1$ from prime ideals of height~$0$ in~$R$.
 The argument here repeats the one in the proof of
Lemma~\ref{one-dimensional-ring-lemma}.
 For every prime ideal $\p$ of height~$1$ in $R$, we use prime
avoidance to choose an element $s\in\p$ not belonging to any of
(the finite set of) prime ideals of height~$0$.
 Then set $S_m\subset R$ to be the multiplicative subset generated
by all the elements $s\in R$ so obtained.
\end{proof}

\begin{proof}[Proof of Theorem~\ref{finite-dimensional-ring-thm}]
 We apply Proposition~\ref{mu-of-d-multsubsets-construction-prop} to
produce a collection of $m=\mu(d)$ multiplicative subsets $S_1$,~\dots,
$S_m\subset R$ distinguishing all the prime ideals in~$R$.
 Then, by Lemma~\ref{distinguishing-prime-ideals-lemma}, all the rings
$R_{J,\s}$ are Artinian.
 According to Corollary~\ref{artinian-rings-r-j-s-cor}, it follows that
all flat $R$\+modules are $\S^\times$\+strongly flat.
\end{proof}

\Section{Contramodule Approximation Sequences}
\label{approximations-secn}

 The aim of this section is to prepare ground for the proofs of
Main Lemmas~\ref{noetherian-almost-cotorsion-main-lemma}
and~\ref{bounded-torsion-almost-cotorsion-main-lemma} in the next
Section~\ref{quite-flat-almost-cotorsion-secn}.
 The constructions of contramodule approximation sequences in this
section generalize those in~\cite[Sections~5.2\+-5.3]{PS}.

 We refer to~\cite[Section~5.1]{PS} and the references therein for
a general discussion of cotorsion theories in abelian categories,
including, first of all, the definitions of a \emph{cotorsion theory}
(or~\emph{pair}), the \emph{approximation sequences} and
a \emph{complete cotorsion theory}, and also a \emph{hereditary
complete cotorsion theory}.

 Specifically, we are interested in what we call the \emph{flat} and
the \emph{quite flat} cotorsion theories in the abelian category
of $S$\+contramodule $R$\+modules $R\modl_{S\ctra}$, where $R$ is
a commutative ring and $S\subset R$ is a countable multiplicative
subset (see Lemma~\ref{S-contramodule-category-lem}(a)).
 Our constructions of these cotorsion theories will require making
some assumptions on the ring $R$ and/or the multiplicative subset~$S$.

\subsection{Flat cotorsion theory in $S$\+contramodule $R$\+modules
for a Noetherian commutative ring~$R$}
\label{noetherian-flat-cotorsion-theory-subsecn}
 We refer to Section~\ref{historical-introd} of the introduction and
the paper~\cite{En} for the definition of an (\emph{Enochs}) 
\emph{cotorsion} $R$\+module.
 In this section, as well as below, we will be particularly interested
in $S$\+contramodule $R$\+modules that are flat, cotorsion, etc.\
\emph{as $R$\+modules}.
 To emphasize this aspect, we will call such $S$\+contramodule
$R$\+modules \emph{$R$\+flat}, \emph{$R$\+cotorsion}, etc.

 The aim of this section is to prove the following theorem
(which generalizes~\cite[Theorem~5.2]{PS}).

\begin{thm} \label{noetherian-contramodule-flat-cotorsion-theory-thm}
 Let $R$ be a Noetherian commutative ring and $S\subset R$ be
a countable multiplicative subset.
 Then the pair of full subcategories\/ $(R$\+flat $S$\+contramodule
$R$\+modules, $R$\+cotorsion $S$\+contramodule $R$\+modules) is
a hereditary complete cotorsion theory in the abelian category
$R\modl_{S\ctra}$.
\end{thm}

 Let $R$ be a commutative ring and $S\subset R$ be a multiplicative
subset.
 We refer to Section~\ref{one-multsubset-secn} for the discussion of
\emph{$S$\+divisible} and \emph{$S$\+h-divisible} $R$\+modules.

\begin{lem} \label{countable-h-divisible}
 Let $R$ be a commutative ring and $S\subset R$ be a countable
multiplicative subset.
 Then an $R$\+module is $S$\+divisible if and only if it is
$S$\+h-divisible.
\end{lem}

\begin{proof}
 This well-known fact has a straightforward proof.
 One observes that, in the notation of
Section~\ref{countable-multsubsets-subsecn}, the $R$\+module
$S^{-1}R$ is the inductive limit of the inductive system of
$R$\+modules $R\overset{s_1}\rarrow R\overset{s_2}\rarrow R
\overset{s_3}\rarrow\dotsb$ (cf.~\cite[proof of Lemma~1.9]{PMat}).

 Now let $M$ be an $S$\+divisible $R$\+module and $x\in M$ be
an element.
 Set $x_0=x$ and, proceeding by induction, choose elements
$x_1$, $x_2$,~\dots~$\in M$ such that $s_nx_n=x_{n-1}$ for all $n\ge1$.
 Then there exists a unique $R$\+module morphism $S^{-1}R\rarrow M$
taking the element $t_n^{-1}\in S^{-1}R$ to the element
$x_n\in M$ for all $n\ge0$.
\end{proof}

 The next three lemmas form a version of the theory developed
in~\cite[Section~B.9]{Pweak} and~\cite[Section~10]{Pcta}
(see also the paper~\cite{Yek}).
 We refer to Section~\ref{contramodules-secn} for the definition of
the $S$\+completion functor $\Lambda_{R,S}$, as well as for
the definition of a partial order on the set $S$ and the related
discussion.

\begin{lem} \label{artin-rees-for-S-limits}
 Let $R$ be a Noetherian commutative ring, $S\subset R$ be
a multiplicative subset, $K\subset L$ be a finitely generated
$R$\+module and its submodule, and $F$ be an $R$\+module.
 Then the natural map between the projective limits
$$
 \varprojlim\nolimits_{s\in S} K/sK\ot_R F\lrarrow
 \varprojlim\nolimits_{s\in S} K/(K\cap sL)\ot_R F
$$
is an isomorphism.
\end{lem}

\begin{proof}
 For any $S$\+indexed projective system of $R$\+modules
$(M_s)_{s\in S}$, there is a natural isomorphism of projective limits
$$
 \varprojlim\nolimits_{s\in S}M_s=\varprojlim\nolimits_{s\in S}
 \varprojlim\nolimits_{n\ge1}M_{s^n}.
$$
 Thus it suffices to show that the natural map
$$
 \varprojlim\nolimits_{n\ge1} K/s^nK\ot_R F\lrarrow
 \varprojlim\nolimits_{n\ge1} K/(K\cap s^nL)\ot_R F
$$
is an isomorphism for every fixed $s\in S$.
 Now, according to the Artin--Rees Lemma applied to the pair of
embedded finitely generated $R$\+modules $K\subset L$ and
the principal ideal $(s)\subset R$ generated by the element $s\in R$,
there exists an integer $m\ge1$ such that $K\cap s^nL=
s^{n-m}(K\cap s^mL)$ for every $n\ge m$.
 Hence $s^nK\subset K\cap s^nL\subset s^{n-m}K$, implying
the desired isomorphism of projective limits over~$n$.
\end{proof}

\begin{lem} \label{noetherian-lambda-flat-lemma}
 Let $R$ be a Noetherian commutative ring, $S\subset R$ be
a countable multiplicative subset, and $F$ be an $R$\+module such that
the $R/sR$\+module $F/sF$ is flat for all $s\in S$.
 Then the $R$\+module\/ $\Lambda_{R,S}(F)=\varprojlim_{s\in S}F/sF$
is flat.
\end{lem}

\begin{proof}
 Consider the functor $M\longmapsto\varprojlim_{s\in S} M\ot_R F/sF$
acting from the category of finitely generated $R$\+modules to
the category of $R$\+modules.
 Let us show that this functor is exact.

 Indeed, for any short exact sequence of finitely generated
$R$\+modules $0\rarrow K\rarrow L\rarrow M\rarrow0$ there are short
exact sequences of $R/sR$\+modules $0\rarrow K\cap s^nL\rarrow s^nL
\rarrow s^nM\rarrow0$.
 Applying the functor ${-}\ot_R F$ preserves exactness of these
short exact sequences, since $F/sF$ is a flat $R/sR$ module.
 The passage to the projective limits over $s\in S$ preserves exactness
of the resulting short exact sequences of tensor products, because
these are countable filtered projective limits of surjective maps.
 It remains to take into account Lemma~\ref{artin-rees-for-S-limits}.

 Furthermore, for any finitely generated $R$\+module $M$ we have
a natural $R$\+module morphism
$$
 M\ot_R\varprojlim\nolimits_{s\in S}F/sF\lrarrow
 \varprojlim\nolimits_{s\in S}M\ot_RF/sF,
$$
which is clearly an isomorphism for finitely generated free
$R$\+modules~$M$.
 Both the functors being right exact on the category of finitely
generated $R$\+modules $M$, it follows that the morphism is
an isomorphism for all such $M$ and the functor
$M\longmapsto M\ot_R\Lambda_{R,s}(F)$ is exact on the category of
finitely gnerated $R$\+modules.
\end{proof}

 We recall that, for any countable multiplicative subset $S$ in
a commutative ring $R$, the natural $R$\+module morphism
$\Delta_{R,S}(M)\rarrow\Lambda_{R,S}(M)$ is surjective for all
$R$\+modules (see Lemma~\ref{S-delta-lambda-sequence}).

\begin{lem} \label{noetherian-delta=lambda-lemma}
 Let $R$ be a Noetherian commutative ring, $S\subset R$ be a countable
multiplicative subset, and $F$ be an $R$\+module such that
the $R$\+module $F/sF$ is flat for all $s\in S$.
 Then the natural $R$\+module map $\Delta_{R,S}(F)\rarrow
\Lambda_{R,S}(F)$ is an isomorphism.
\end{lem}

\begin{proof}
 Denote by $K$ the kernel of the natural map $\Delta_{R,S}(F)\rarrow
\Lambda_{R,S}(F)$; so we have a short exact sequence of $R$\+modules
$0\rarrow K\rarrow\Delta_{R,S}(F)\rarrow\Lambda_{R,S}(F)\rarrow0$.
 Applying the functor $R/sR\ot_R{-}$ and taking into account the fact
that the $R$\+module $\Lambda_{R,S}(F)$ is flat by
Lemma~\ref{noetherian-lambda-flat-lemma}, we get a short exact
sequence of $R/sR$\+modules
$$
 0\lrarrow K/sK\lrarrow\Delta_{R,S}(F)/s\Delta_{R,S}(F)\lrarrow
 \Lambda_{R,S}(F)/s\Lambda_{R,S}(F)\lrarrow0.
$$
 Now we have $\Delta_{R,S}(F)/s\Delta_{R,S}(F)=F/sF$
by~\cite[Lemma~1.11]{PMat} and $\Lambda_{R,S}(F)/s\Lambda_{R,S}(F)
\allowbreak=F/sF$ by~\cite[Proposition~2.2(b) and
Theorem~2.3\,(i)$\Leftrightarrow$(iv)]{PMat},
hence it follows that $K/sK=0$ and $K=sK$.
 As this holds for all $s\in S$, we can apply
Lemma~\ref{countable-h-divisible} and conclude that
the natural morphism $\Hom_R(S^{-1}R,K)\rarrow K$ is surjective.

 On the other hand, both $\Delta_{R,S}(F)$ and $\Lambda_{R,S}(F)$ are
$S$\+contramodule $R$\+modules (see
Lemma~\ref{S-contramodule-category-lem}(b)
and~\cite[Lemma~2.1(a)]{PMat}), hence $K$ is an $S$\+contramodule
$R$\+module and $\Hom_R(S^{-1}R,K)=0$.
 Therefore, $K=0$.
\end{proof}

 Let $R$ be a commutative ring and $S\subset R$ be a multiplicative
subset such that $\pd_RS^{-1}R\le1$.
 Let $0\rarrow A\rarrow B\rarrow C\rarrow0$ be a short exact sequence
of $R$\+modules.
 Applying the cohomological functor $\Hom_{\sD(R\modl)}(K^\bu_{R,S},{-})$,
we obtain an exact sequence
\begin{multline} \label{derived-delta-sequence}
 0\lrarrow\Hom_R(S^{-1}R/R,\.A)\lrarrow\Hom_R(S^{-1}R/R,\.B)\lrarrow
 \\ \Hom_R(S^{-1}R/R,\.C) \lrarrow
 \Delta_{R,S}(A)\lrarrow\Delta_{R,S}(B)\lrarrow\Delta_{R,S}(C)\lrarrow0
\end{multline}
(see Section~\ref{one-multsubset-secn} for the notation).

 Notice that, when $R$ is a Noetherian commutative ring,
the $S$\+torsion in $R$ is bounded for any multiplicative subset
$S\subset R$.

\begin{lem} \label{bounded-torsion-flat-modules-lemma}
 Let $R$ be a commutative ring and $S\subset R$ be a multiplicative
subset such that the $S$\+torsion in $R$ is bounded.
 Then one has\/ $\Hom_R(S^{-1}R/R,\.F)=0$ for any flat $R$\+module~$F$.
\end{lem}

\begin{proof}
 If the $S$\+torsion in a commutative ring $R$ is bounded, then
the $S$\+torsion in any flat $R$\+module $F$ is bounded,
too~\cite[proof of Corollary~2.7]{PMat}.
 Furthermore, one has $\Hom_R(S^{-1}R/R,\.M)=
\Hom_R(S^{-1}R/R,\.\Gamma_S(M))=0$ for any $R$\+module $M$ with
bounded $S$\+torsion, since $\Hom_R(S^{-1}R/R,\.N)\subset
\Hom_R(S^{-1}R,N)=0$ for any $R$\+module $N$ annihilated by
an element $s\in S$.
\end{proof}

 Let $R$ be a Noetherian commutative ring and $S\subset R$ be
a countable multiplicative subset.
 The pair of full subcategories ($R$\+flat $S$\+contramodule
$R$\+modules, $R$\+cotorsion $S$\+contramodule $R$\+modules) is called
the \emph{flat cotorsion theory} in the abelian category
$R\modl_{S\ctra}$.
 Having finished the preparatory work, we can now proceed to construct
the approximation sequences in the category $R\modl_{S\ctra}$ proving
that this is indeed a complete cotorsion theory (as it was promised in
Theorem~\ref{noetherian-contramodule-flat-cotorsion-theory-thm}).

\begin{lem} \label{noetherian-flat-precover-lemma}
 Let $R$ be a Noetherian commutative ring and $S\subset R$ be
a countable multiplicative subset.
 Let $C$ be an $S$\+contramodule $R$\+module, and let\/ $0\rarrow K
\rarrow F\rarrow C\rarrow0$ be a short exact sequence of $R$\+modules
with a flat $R$\+module $F$.
 Then there is a short exact sequence of $S$\+contramodule $R$\+modules
$$
 0\lrarrow\Delta_{R,S}(K)\lrarrow\Delta_{R,S}(F)\lrarrow C\lrarrow0
$$
with a flat $R$\+module $\Delta_{R,S}(F)$.
 If $K$ is a cotorsion $R$\+module, then the $R$\+module
$\Delta_{R,S}(K)$ is also cotorsion.
\end{lem}

\begin{proof}
 By Lemma~\ref{S-contramodule-category-lem}(b), we have
$\Delta_{R,S}(C)=C$.
 Furthermore, the $R$\+module $\Hom_R(S^{-1}R/R,\.C)$ is a submodule in
the $R$\+module $\Hom_R(S^{-1}R,C)$, which vanishes by virtue of $C$
being an $S$\+contramodule.
 Hence the desired short exact sequence is a particular case of
the exact sequence~\eqref{derived-delta-sequence}.
 By Lemmas~\ref{noetherian-lambda-flat-lemma}
and~\ref{noetherian-delta=lambda-lemma}, the $R$\+module
$\Delta_{R,S}(F)$ is flat for any flat $R$\+module~$F$.

 Finally, we have $\Hom_R(S^{-1}R/R,\.F)=0$ by
Lemma~\ref{bounded-torsion-flat-modules-lemma},
and consequently $\Hom_R(S^{-1}R/R,\.K)=0$.
 Assume that the $R$\+module $K$ is $S$\+weakly cotorsion.
 Then the exact sequence~\eqref{weakly-cotorsion-sequence}
(from Lemma~\ref{matlis-exact-sequence}) for the $R$\+module $K$
reduces to a short exact sequence $0\rarrow\Hom_R(S^{-1}R,K)\rarrow K
\rarrow\Delta_{R,S}(K)\rarrow0$.

 Now, since the $R$\+module $S^{-1}R$ is flat, the $R$\+module
$\Hom_R(S^{-1}R,K)$ is cotorsion whenever the $R$\+module $K$ is
\cite[Lemma~1.3.2(a)]{Pcosh}.
 Therefore, the $R$\+module $\Delta_{R,S}(K)$ is also cotorsion in this
case, as the cokernel of an injective morphism of cotorsion
$R$\+modules. 
\end{proof}

\begin{lem} \label{noetherian-cotorsion-envelope-lemma}
 Let $R$ be a Noetherian commutative ring and $S\subset R$ be
a countable multiplicative subset.
 Let $C$ be an $S$\+contramodule $R$\+module, and let\/ $0\rarrow C
\rarrow K\rarrow F\rarrow0$ be a short exact sequence of $R$\+modules
with a flat $R$\+module $F$.
 Then there is a short exact sequence of $S$\+contramodule $R$\+modules
$$
 0\lrarrow C\lrarrow\Delta_{R,S}(K)\lrarrow\Delta_{R,S}(F)\lrarrow 0
$$
with a flat $R$\+module $\Delta_{R,S}(F)$.
 If $K$ is a cotorsion $R$\+module, then the $R$\+module
$\Delta_{R,S}(K)$ is also cotorsion.
\end{lem}

\begin{proof}
 For the reasons mentioned in the proof of
Lemma~\ref{noetherian-flat-precover-lemma}, we have
$\Delta_{R,S}(C)=C$ and $\Hom_R(S^{-1}R/R,\.F)=0$.
 Hence the desired short exact sequence is a particular case of
the exact sequence~\eqref{derived-delta-sequence}.
 The $R$\+module $\Delta_{R,S}(F)$ is flat for any flat $R$\+module
$F$, as it was pointed out.
 Also for the reasons explained in the proof of
Lemma~\ref{noetherian-flat-precover-lemma}, we have
$\Hom_R(S^{-1}R/R,\.C)=0$.
 Hence $\Hom_R(S^{-1}R/R,\.K)=0$, and in the same way as in
the proof of Lemma~\ref{noetherian-flat-precover-lemma} one
deduces that the $R$\+module $\Delta_{R,S}(K)$ is cotorsion
whenever the $R$\+module $K$ is.
\end{proof}

\begin{proof}[Proof of
Theorem~\ref{noetherian-contramodule-flat-cotorsion-theory-thm}]
 All the substantial work has been done already in
Lemmas~\ref{noetherian-flat-precover-lemma}
and~\ref{noetherian-cotorsion-envelope-lemma}, which produce
the required approximation sequences out of the approximation
sequences for the flat cotorsion theory on the category of
$R$\+modules $R\modl$ (which exist by~\cite[Theorem~10]{ET}
and~\cite[Proposition~2]{BBE}).

 One can further observe that the functor $\Ext^1$ in the abelian
category $R\modl_{S\ctra}$ agrees with the functor $\Ext^1$ in
$R\modl$, since $R\modl_{S\ctra}\subset R\modl$ is a full subcategory
closed under kernels, cokernels, and extensions.
 Besides, the full subcategory of $R$\+flat objects in $R\modl_{S\ctra}$
is closed under direct summands and kernels of epimorphisms,
while the full subcategory of $R$\+cotorsion objects in
$R\modl_{S\ctra}$ is closed under direct summands and cokernels
of monomorphisms, since the full subcategories of flat $R$\+modules
and cotorsion $R$\+modules in the abelian category $R\modl$ have
similar properties.
 This observations are sufficient to imply that the pair of full
subcategories in $R\modl_{S\ctra}$ which we are interested in is
a cotorsion theory/pair, and that this cotorsion theory is complete
and hereditary.
\end{proof}

\begin{rem}
 Notice that, in the context of
Lemma~\ref{noetherian-flat-precover-lemma}, if the $R$\+module $K$ is,
at least, $S$\+weakly cotorsion, then so is the $R$\+module $F$
(because the $R$\+module $C$, being an $S$\+contramodule, is
consequently $S$\+weakly cotorsion, and $F$ is an extension of
$C$ and~$K$).
 Hence one has $\Delta_{R,S}(K)=K/h_S(K)$ and $\Delta_{R,S}(F)=F/h_S(F)$.
 Similarly, in the context of
Lemma~\ref{noetherian-cotorsion-envelope-lemma}, if the $R$\+module $K$
is $S$\+weakly cotorsion, then so is the $R$\+module $F$ (as
a quotient $R$\+module of~$K$).
 Hence one also has
$\Delta_{R,S}(K)=K/h_S(K)$ and $\Delta_{R,S}(F)=F/h_S(F)$.
 In other words, for a countable multiplicative subset $S$ in
a Noetherian commutative ring $R$, in order to produce the approximation
sequences for the flat cotorsion theory in $R\modl_{S\ctra}$ from
the approximation sequences for the flat cotorsion theory in $R\modl$,
all one needs to do is to quotient out the maximal $S$\+divisible
submodules.
\end{rem}

\subsection{Quite flat cotorsion theory in the bounded torsion case}
\label{bounded-torsion-quite-flat-cotorsion-theory-subsecn}
 Let $R$ be a commutative ring, $S\subset R$ be a multiplicative
subset, and $M$ be an $R$\+module.
 The \emph{$S$\+topology} on $M$ is the topology with the base of
neighborhoods of zero formed by the submodules $sM\subset M$,
where $s\in S$.
 The $R$\+module $\Lambda_S(M)=\varprojlim_{s\in S}M/sM$ is
the completion of $M$ with respect to the $S$\+topology.
 When $S$ is countable, the topology of projective limit (of discrete
$R$\+modules $M/sM$) on $\Lambda_S(M)$ always coincides with
the $S$\+topology of the $R$\+module $\Lambda_S(M)$
\cite[Proposition~2.2(b)
and Theorem~2.3\,(i)$\Leftrightarrow$(ii)]{PMat}.

 We refer to~\cite[Section~1.2]{Pweak} or~\cite[Section~5]{PR} for
the definition of a left contramodule over a complete, separated
topological associative ring $\fR$ with a base of neighborhoods of
zero fomed by open right ideals.
 In the context of this paper, we set $\fR=\Lambda_S(R)$, where $R$
is a commutative ring and $S\subset R$ is a countable multiplicative
subset.
 The commutative ring $\fR$ is endowed with its topology of
projective limit of discrete commutative rings
$\varprojlim_{s\in S}R/sR$ or, which is the same, its $S$\+topology
(as an $R$\+module).
 We denote the abelian category of $\fR$\+contramodules by
$\fR\contra$.

 The following result is a generalization of~\cite[Theorem~5.9]{PS}.

\begin{thm} \label{contramodule-category-equivalence}
 Let $R$ be a commutative ring and $S\subset R$ be a countable
multiplicative subset such that the $S$\+torsion in $R$ is bounded.
 Then the forgetful functor\/ $\fR\contra\rarrow R\modl$ induces
an equivalence of abelian categories\/ $\fR\contra\simeq
R\modl_{S\ctra}$.
\end{thm}

\begin{proof}
 In fact, this assertion holds under in the greater generality of
a multiplicative subset $S\subset R$ such that $\pd_RS^{-1}R\le1$
and the $S$\+torsion in $R$ is bounded.
 See~\cite[Example~2.4(3)]{Pper}.
 For a countable multiplicative subset $S\subset R$ such that
the $S$\+torsion in $R$ is unbounded, the abelian category
$\fR\contra$ is, generally speaking, a full subcategory in
the abelian category $R\modl_{S\ctra}$ with an exact embedding
functor $\fR\contra\rarrow R\modl_{S\ctra}$
\cite[Example~3.7(2)]{Pper} (cf.~\cite[Theorem~5.20]{PS}).
\end{proof}

 It is important for us that, when the multiplicative subset
$S\subset R$ is countable, the topological ring $\fR$ has
a countable base of neighborhoods of zero, so the results
of~\cite[Sections~6\+-7]{PR} are applicable.

 An $\fR$\+contramodule $F$ is called \emph{flat} if
the $R/sR$\+modules $F/sF$ are flat for all elements $s\in S$
(see~\cite[Section~D.1]{Pcosh} or~\cite[Sections~5\+-6]{PR}).
 Unlike in the Noetherian case of
Section~\ref{noetherian-flat-cotorsion-theory-subsecn}, there is
\emph{no} claim that flat $\fR$\+contramodules are flat $R$\+modules
in our present generality (when $S$ is just a countable multiplicative
subset in a commutative ring $R$ such that the $S$\+torsion in $R$
is bounded).

\begin{lem} \label{delta-of-flat-is-flat-contra-lem}
 Let $R$ be a commutative ring and $S\subset R$ be a countable
multiplicative subset such that the $S$\+torsion in $R$ is bounded.
 Then, for any flat $R$\+module $F$, the $S$\+contramodule $R$\+module
$\Delta_{R,S}(F)$ is a flat\/ $\fR$\+contramodule.
\end{lem}

\begin{proof}
 By~\cite[Lemma~1.11]{PMat}, we have $\Delta_{R,S}(F)/s\Delta_{R,S}(F)=
F/sF$, which is a flat $R/sR$\+module.
\end{proof}

 The definitions of an \emph{almost cotorsion} and a \emph{quite flat}
$R$\+module were given in Section~\ref{quite-flat-introd}.
 The following definitions extend these concepts to the realm of
$S$\+contramodule $R$\+modules.
 For simplicity of notation, let us recall and use the fact that
the functor $\Ext^1$ in the abelian category $R\modl_{S\ctra}$ agrees
with the functor $\Ext^1$ in the abelian category $R\modl$, as
$R\modl_{S\ctra}\subset R\modl$ is a full subcategory closed under
kernels, cokernels, and extensions.

 Let $R$ be a commutative ring and $S\subset R$ be a countable
multiplicative subset such that the $S$\+torsion in $R$ is bounded.
 For any countable multiplicative subset $T\subset R$, consider
the $S$\+contramodule $R$\+module $\Delta_{R,S}(T^{-1}R)$.
 An $S$\+contramodule $R$\+module $C$ is said to be \emph{almost
cotorsion} if $\Ext_R^1(\Delta_{R,S}(T^{-1}R),C)=0$ for all countable
multiplicative subsets $T\subset R$.
 An $S$\+contramodule $R$\+module $F$ is said to be \emph{quite flat}
if $\Ext_R^1(F,C)=0$ for all almost cotorsion $S$\+contramodule
$R$\+modules~$C$.

 In other words, these definitions mean that the pair of full
subcategories (quite flat $S$\+contramodule $R$\+modules, almost
cotorsion $S$\+contramodule $R$\+modules) is defined as
the cotorsion theory/pair \emph{generated by} the objects
$\Delta_{R,S}(T^{-1}R)$ in the abelian category $R\modl_{S\ctra}$.
 This cotorsion theory is called the \emph{quite flat} cotorsion
theory in the abelian category $R\modl_{S\ctra}$.

 Once again, there is \emph{no} claim that quite flat $S$\+contramodule
$R$\+modules are quite flat as $R$\+modules (even for a Noetherian
ring~$R$).
 On the other hand, the following assertions hold.

\begin{prop} \label{quite-flat-contramodules-are-flat}
 Let $R$ be a commutative ring and $S\subset R$ be a countable
multiplicative subset such that the $S$\+torsion in $R$ is bounded.
 Then all quite flat $S$\+contramodule $R$\+modules are flat\/
$\fR$\+contramodules.
\end{prop}

\begin{proof}
 Both the arguments in~\cite[proof of Proposition~5.12]{PS} are
applicable to the situation at hand just as well (proving also in
addition that an $S$\+contramodule $R$\+module is quite flat if and
only if it is a direct summand of a transfinitely iterated extension
of the objects $\Delta_{R,S}(T^{-1}R)$ in the category $R\modl_{S\ctra}$,
in the sense of the inductive limit, or more precisely,
of~\cite[Definition~4.3]{PR}).
\end{proof}

\begin{thm} \label{almost-cotorsion-contramodules-thm}
 Let $R$ be a commutative ring and $S\subset R$ be a countable
multiplicative subset such that the $S$\+torsion in $R$ is bounded.
 Then an $S$\+contramodule $R$\+module is almost cotorsion if and only
if it is almost cotorsion as an $R$\+module.
\end{thm}

\begin{proof}
 The proof of this theorem does not really use the condition of
countability of the multiplicative subset $S\subset R$, but only
the conditions that the $S$\+torsion in $R$ is bounded and
that the projective dimension of the $R$\+module $S^{-1}R$ does
not exceed~$1$.
 We will construct an isomorphism of the Ext modules
$$
 \Ext^1_R(F,C)\simeq\Ext_{R\modl_{S\ctra}}^1(\Delta_{R,S}(F),C)
$$
for any flat $R$\+module $F$ and any $S$\+contramodule $R$\+module~$C$.
 Specializing to the case of the flat $R$\+module $F=T^{-1}R$ will then
immediately imply the assertion of the theorem.

 Suppose that we have short exact sequence of $R$\+modules
\begin{equation} \label{R-module-extension-sequence}
 0\lrarrow C\lrarrow B\lrarrow F\lrarrow0,
\end{equation}
where $F$ is a flat $R$\+module and $C$ is an $S$\+contramodule
$R$\+module.
 Applying the functor $\Delta_{R,S}$ and recalling the exact
sequence~\eqref{derived-delta-sequence}
from Section~\ref{noetherian-flat-cotorsion-theory-subsecn}
together with Lemma~\ref{bounded-torsion-flat-modules-lemma}, we
get a short exact sequence of $S$\+contramodule $R$\+modules
\begin{equation} \label{induced-S-contramodule-extension-sequence}
 0\lrarrow C\lrarrow\Delta_{R,S}(B)\lrarrow\Delta_{R,S}(F)\lrarrow0,
\end{equation}
because $\Delta_{R,S}(C)=C$ by
Lemma~\ref{S-contramodule-category-lem}(b).
 There is a natural adjunction morphism from the short exact
sequence~\eqref{R-module-extension-sequence} to the short exact
sequence~\eqref{induced-S-contramodule-extension-sequence}, hence
it follows that the sequence~\eqref{R-module-extension-sequence} is
the pullback of
the sequence~\eqref{induced-S-contramodule-extension-sequence}
with respect to the adjunction morphism $F\rarrow\Delta_{R,S}(F)$.

 Conversely, given a short exact sequence of $S$\+contramodule
$R$\+modules
\begin{equation} \label{arbitrary-S-contramodule-extension-sequence}
 0\lrarrow C\lrarrow B'\lrarrow\Delta_{R,S}(F)\lrarrow0,
\end{equation}
one can take the pullback with respect to the morphism $F\rarrow
\Delta_{R,S}(F)$ in order to produce a short exact sequence
of $R$\+modules~\eqref{R-module-extension-sequence}.
 Then there is a natural morphism of short exact sequences of
$R$\+modules from the exact
sequence~\eqref{R-module-extension-sequence} into the exact
sequence~\eqref{arbitrary-S-contramodule-extension-sequence};
and the adjunction morphisms provide an isomorphism from
the exact sequence~\eqref{induced-S-contramodule-extension-sequence} to
the exact sequence~\eqref{arbitrary-S-contramodule-extension-sequence}.
\end{proof}

\begin{thm} \label{quite-flat-cotorsion-theory-complete}
 Let $R$ be a commutative ring and $S\subset R$ be a countable
multiplicative subset such that the $S$\+torsion in $R$ is bounded.
 Then the pair of full subcategories (quite flat $S$\+contramodule
$R$\+modules, almost cotorsion $S$\+contramodule $R$\+modules) is
a hereditary complete cotorsion theory in the abelian category
$R\modl_{S\ctra}$.
\end{thm}

\begin{proof}
 Since the $R$\+module $T^{-1}R$ is flat for any (countable)
multiplicative subset $T\subset R$, the $\fR$\+contramodule
$\Delta_{R,S}(T^{-1}R)$ is flat by
Lemma~\ref{delta-of-flat-is-flat-contra-lem}.
 Identifying the category $R\modl_{S\ctra}$ with the category
$\fR\contra$ by Theorem~\ref{contramodule-category-equivalence}, we can
apply the result of~\cite[Corollary~7.11]{PR}, according to which
any cotorsion theory generated by a set of flat $\fR$\+contramodules
is complete in $\fR\contra$.

 To show that the quite flat cotorsion theory in $R\modl_{S\ctra}$ is
hereditary, we will check that the objects $\Delta_{R,S}(T^{-1}R)$
have projective dimension at most~$1$ in $R\modl_{S\ctra}$.
 This will prove the stronger claim that the class of almost
cotorsion $S$\+contramodule $R$\+modules is closed under the passages
to arbitrary quotient objects in $R\modl_{S\ctra}$, while all
quite flat $S$\+contramodule $R$\+modules have projective dimension
at most~$1$ as objects of $R\modl_{S\ctra}$.

 Indeed, let $0\rarrow Q\rarrow P\rarrow T^{-1}R\rarrow0$ be
a projective resolution of the $R$\+module $T^{-1}R$.
 Applying the functor $\Delta_{R,S}$ and using the exact
sequence~\eqref{derived-delta-sequence}
together with Lemma~\ref{bounded-torsion-flat-modules-lemma}, we
obtain a short exact sequence of $S$\+contramodule $R$\+modules
$0\rarrow \Delta_{R,S}(Q)\rarrow\Delta_{R,S}(P)\rarrow
\Delta_{R,S}(T^{-1}R)\rarrow0$.
 The functor $\Delta_{R,S}\:R\modl\rarrow R\modl_{S\ctra}$ is left
adjoint to an exact functor, so it takes projectives to projectives.
 Thus we have obtained the desired two-term projective resolution
of our object $\Delta_{R,S}(T^{-1}R)\in R\modl_{S\ctra}$.
\end{proof}

\subsection{Separated $S$\+contramodule $R$\+modules}
\label{separated-contramodules-subsecn}
 Let $R$ be a commutative ring and $S\subset R$ be a multiplicative
subset.
 An $R$\+module $C$ is said to be \emph{$S$\+complete} if
the natural map
$$
 \lambda_{R,S,C}\:C\lrarrow\Lambda_{R,S}(C) =
 \varprojlim\nolimits_{s\in S} C/sC
$$
is surjective, and \emph{$S$\+separated} if the map~$\lambda_{R,S,C}$
is injective.

 Clearly, an $R$\+module $C$ is $S$\+separated if and only if
the intersection $\bigcap_{s\in S}sC\subset C$ vanishes.
 It follows that any $R$\+submodule of an $S$\+separated $R$\+module
is $S$\+separated. {\hbadness=1200\par}

 For any multiplicative subset $S$ in a commutative ring $R$ such
that $\pd_RS^{-1}R\le1$, any $S$\+separated and $S$\+complete
$R$\+module is an $S$\+contramodule~\cite[Lemma~2.1(a)]{PMat}.
 It follows from our Lemmas~\ref{S-contramodule-category-lem}(b)
and~\ref{S-delta-lambda-sequence} that, for any countable multiplicative
subset $S\subset R$, any $S$\+contramodule $R$\+module is $S$\+complete.

 The following corollaries pick out the aspects of the results of
Sections~\ref{noetherian-flat-cotorsion-theory-subsecn}\+-%
\ref{bounded-torsion-quite-flat-cotorsion-theory-subsecn}
relevant for the proofs of
Main Lemmas~\ref{noetherian-almost-cotorsion-main-lemma}
and~\ref{bounded-torsion-almost-cotorsion-main-lemma} in
Section~\ref{quite-flat-almost-cotorsion-secn}.
 We recall that, according to our terminology introduced in
the beginning of Section~\ref{noetherian-flat-cotorsion-theory-subsecn},
an $S$\+contramodule $R$\+module is called \emph{$R$\+almost cotorsion}
if it is almost cotorsion \emph{as an $R$\+module}.

\begin{cor} \label{noetherian-cokernel-of-almost-cotorsion-separated}
 Let $R$ be a Noetherian commutative ring, $S\subset R$ be a countable
multiplicative subset, and $C$ be an $R$\+almost cotorsion
$S$\+contramodule $R$\+module.
 Then the $R$\+module $C$ can be presented as the cokernel of
an injective morphism of $R$\+almost cotorsion $S$\+separated
$S$\+complete $R$\+modules.
\end{cor}

\begin{proof}
 By Theorem~\ref{noetherian-contramodule-flat-cotorsion-theory-thm},
or more specifically by Lemma~\ref{noetherian-flat-precover-lemma},
there exists a short exact sequence of $S$\+contramodule $R$\+modules
$0\rarrow K\rarrow F\rarrow C\rarrow0$, where the $R$\+mdoule $F$
is flat and the $R$\+module $K$ is cotorsion.
 What is important for us is that the $R$\+module $K$ is almost
cotorsion; since the $R$\+module $C$ is almost cotorsion by assumption,
it follows that the $R$\+module $F$ is almost cotorsion, too.

 Furthermore, any $R$\+flat $S$\+contramodule $R$\+module is
$S$\+separated by Lemma~\ref{noetherian-delta=lambda-lemma}.
 The $R$\+module $F$ being $S$\+separated, it follows that its
submodule $K$ is $S$\+separated, too.
 Thus $K\rarrow F$ is an injective morphism of $R$\+almost cotorsion
$S$\+separated $S$\+contramodules with the cokernel~$C$.
\end{proof}

\begin{cor} \label{bounded-torsion-cokernel-of-ac-separated}
 Let $R$ be a commutative ring and $S\subset R$ be a countable
multiplicative subset such that the $S$\+torsion in $R$ is bounded.
 Then any $R$\+almost cotorsion $S$\+contramodule $R$\+module
can be presented as the cokernel of an injective morphism of
$R$\+almost cotorsion $S$\+separated $S$\+complete $R$\+modules.
\end{cor}

\begin{proof}
 By Theorem~\ref{quite-flat-cotorsion-theory-complete}, the quite
flat cotorsion theory in $R\modl_{S\ctra}$ is complete.
 Hence for any $S$\+contramodule $R$\+module $C$ there exists
a short exact sequence $0\rarrow K\rarrow F\rarrow C\rarrow0$,
where $K$ is an almost cotorsion $S$\+contramodule $R$\+module and
$F$ is a quite flat $S$\+contramodule $R$\+module.

 By Theorem~\ref{almost-cotorsion-contramodules-thm}, $K$ is an almost
cotorsion $R$\+module.
 Assuming that $C$ is an almost cotorsion $R$\+module, we can conclude
that the $R$\+module $F$ is almost cotorsion, too.
 By Proposition~\ref{quite-flat-contramodules-are-flat}, $F$ is
a flat $\fR$\+contramodule.
 By~\cite[Corollary~D.1.7]{Pcosh} or~\cite[Corollary~6.15]{PR},
all flat $\fR$\+contramodules are $S$\+separated.
 So the $R$\+module $F$ is $S$\+separated, and it follows that its
submodule $K$ is $S$\+separated, too.

 Thus $K\rarrow F$ is an injective morphism of $R$\+almost cotorsion
$S$\+separated $S$\+contramodule $R$\+modules with the cokernel~$C$.
\end{proof}

\begin{rem}
 Just as in~\cite[Sections~5.5\+-5.6]{PS}, for any countable
multiplicative subset $S$ in a commutative ring $R$ one can define
the full subcategory of \emph{quotseparated} $S$\+contramodule
$R$\+modules $R\modl_{S\ctra}^\qs\subset R\modl$ and show that it is
equivalent to the abelian category $\fR\contra$.
 The one can proceed to construct, following essentially the exposition
in~\cite[Section~D.4]{Pcosh} with the words ``very flat'' replaced by
``quite flat'' and the word ``contraadjusted'' replaced by
``almost cotorsion'', and with the same simplifications as in~\cite{PS},
the quite flat cotorsion theory on the abelian category
$R\modl_{S\ctra}^\qs$.
 In the same way as in~\cite{PS}, it follows that any $R$\+almost
cotorsion quotseparated $S$\+contramodule $R$\+module is
the cokernel of an injective morphism of $R$\+almost cotorsion
$S$\+separated $S$\+complete $R$\+modules.
 We omit the details.
\end{rem}

\Section{Quite Flat and Almost Cotorsion Modules}
\label{quite-flat-almost-cotorsion-secn}

 Let us start from recalling the definitions given in
Section~\ref{quite-flat-introd}.
 Let $R$ be a commutative ring.
 We say that an $R$\+module $C$ is \emph{almost cotorsion} if
$\Ext_R^1(S^{-1}R,C)=0$ for all at most countable multiplicative
subsets $S\subset R$.
 An $R$\+module $F$ is said to be \emph{quite flat} if $\Ext_R^1(F,C)=0$
for all almost cotorsion $R$\+modules~$C$.

 An $R$\+module $F$ is quite flat if and only if it is a direct summand
of a transfinitely iterated extension, in the sense of the inductive
limit, of $R$\+modules isomorphic to $S^{-1}R$, where $S\subset R$ are
(at most) countable multiplicative subsets~\cite[Corollary~6.14]{GT}.
 The pair of full subcategories (quite flat $R$\+modules, almost
cotorsion $R$\+modules) is called the \emph{quite flat cotorsion
theory} in $R\modl$.

 The following four lemmas list the general properties of the classes
of almost cotorsion and quite flat modules over commutative rings.

\begin{lem} \label{almost-cotorsion-quite-flat-closedness-properties}
\textup{(a)} For any commutative ring $R$, the class of all almost
cotorsion $R$\+modules is closed under extensions, quotients
(by arbitrary submodules), infinite products, and transfinitely iterated
extensions in the sense of the projective limit in the category $R\modl$.
\par
\textup{(b)} For any commutative ring $R$, the class of all quite
flat $R$\+modules is closed under extensions, kernels of surjective
morphisms, direct summands, infinite direct sums, and transfinitely
iterated extensions in the sense of the inductive limit in
the category $R\modl$. \par
\textup{(c)} The projective dimension of any quite flat $R$\+module
does not exceed\/~$1$.
\end{lem}

\begin{proof}
 The properties of closedness with respect to quotients in part~(a) and
kernels of surjective morphisms in part~(b) follow from the fact that
$\pd_RS^{-1}R\le1$ for all countable multiplicative subsets $S\subset R$.
 So does the assertion of part~(c).
 All the other assertions are general properties of $\Ext^1$\+orthogonal
classes in $R\modl$ \cite[Lemma~1 and Proposition~18]{ET}.
\end{proof}

\begin{lem}
\textup{(a)} For any commutative ring $R$, the class of quite flat
$R$\+modules is closed with respect to the tensor products over~$R$.
\par
\textup{(b)} For any commutative ring $R$, any quite flat $R$\+module
$F$, and any almost cotorsion $R$\+module $C$, the $R$\+module\/
$\Hom_R(F,C)$ is almost cotorsion.
\end{lem}

\begin{proof}
 Similar to the proof of~\cite[Lemma~1.2.1]{Pcosh}, using
the isomorphism $\Ext_R^1(F\ot_RG,\>C)\simeq\Ext_R^1(G,\Hom_R(F,C))$,
holding for any $R$\+module $G$, any quite flat $R$\+module $F$,
and any almost cotorsion $R$\+module~$C$.
\end{proof}

\begin{lem} \label{almost-cotorsion-quite-flat-change-of-ring}
 Let $f\:R\rarrow R'$ be a homomorphism of commutative rings.
 Then \par
\textup{(a)} any almost cotorsion $R'$\+module is also an almost
cotorsion $R$\+module in the $R$\+module structure obtained by
the restriction of scalars via~$f$; \par
\textup{(b)} if $F$ is a quite flat $R$\+module, then the $R'$\+module
$R'\ot_RF$ obtained by the extension of scalars via~$f$ is also
quite flat.
\end{lem}

\begin{proof}
 Both the assertions hold, essentially, because for any (at most)
countable multiplicative subset $S\subset R$, the image $S'=f(S)
\subset R'$ of the multiplicative subset $S$ under the map~$f$
is an (at most) countable multiplicative subset $S'\subset R'$,
and the $R'$\+algebra/module $R'\ot_RS^{-1}R$ is isomorphic to
$S'{}^{-1}R'$ (cf.~\cite[Lemma~1.2.2(a\+b)]{Pcosh}).
\end{proof}

\begin{lem} \label{almost-cotorsion-reflected-lem}
 Let $f\:R\rarrow R'$ be a homomorphism of commutative rings such that
for every element $r'\in R'$ there exist an element $r\in R$ and
an invertible element $u\in R'$ for which $r'=uf(r)$.
 Then an $R'$\+module is almost cotorsion if and only if it is
almost cotorsion as an $R$\+module.
\end{lem}

\begin{proof}
 The ``only if'' assertion holds by
Lemma~\ref{almost-cotorsion-quite-flat-change-of-ring}.
 To prove the ``if'', suppose that we are given a countable
multiplicative subset $S'\subset R'$.
 For every element $s'\in S'$, choose an element $s\in R$ such that
there exists an invertible element $u\in R'$ for which $s'=uf(s)$.
 Let $S\subset R$ be the multiplicative subset generated by all
the elements $s\in R$ so obtained.
 Then $S$ is a countable multiplicative subset in $R$ and we have
$R'\ot_RS^{-1}R=S'{}^{-1}R'$.
 Now if $C$ is an $R'$\+module that is almost cotorsion as
an $R$\+module, then $\Ext^1_{R'}(S'{}^{-1}R',C)=\Ext^1_{R'}(R'\ot_R
S^{-1}R,\>C)=\Ext^1_R(S^{-1}R,C)=0$, hence $C$ is an almost cotorsion
$R'$\+module.
\end{proof}

 In view of
Lemma~\ref{almost-cotorsion-quite-flat-closedness-properties}(c),
it follows from Theorem~\ref{countable-spectrum-thm} that
the projective dimension of a flat module over a commutative
Noetherian ring with countable spectrum cannot exceed~$1$.
 Let us present a simpler alternative proof of this result before
proceeding to prove Theorem~\ref{countable-spectrum-thm}.

\begin{lem} \label{localization-determined-by-spectrum}
 Let $S$ and $T\subset R$ be two multiplicative subsets in
a commutative ring~$R$.
 Assume that, for any prime ideal\/ $\p\subset R$, one has\/
$\p\cap S\ne\varnothing$ if and only if\/ $\p\cap T\ne\varnothing$.
 Then the two $R$\+algebras $S^{-1}R$ and $T^{-1}R$ are naturally
isomorphic.
\end{lem}

\begin{proof}
 It suffices to check that an element $r\in R$ is invertible in
$S^{-1}R$ if and only if it is invertible in $T^{-1}R$.
 Indeed, if an element $r\in R$ is not invertible in $S^{-1}R$,
then there exists a maximal ideal $\m\subset S^{-1}R$ such that
the image of~$r$ in $S^{-1}R$ belongs to~$\m$.
 Denote by $\p\subset R$ the full preimage of the maximal ideal
$\m\subset S^{-1}R$ under the localization morphism $R\rarrow S^{-1}R$.
 Then $\p$~is a prime ideal in $R$ such that $r\in\p$ and
$\p\cap S=\varnothing$.
 Conversely, if $\p\subset R$ is a prime ideal such that $\p\cap S
=\varnothing$, then the extension of $\p$ in $S^{-1}R$ is not
the unit ideal in $S^{-1}R$; so if $r\in\p$ then $r$~is not
invertible in~$S^{-1}R$.
 Thus an element $r\in R$ becomes invertible in $S^{-1}R$ if and
only if one has $\p\cap S\ne\varnothing$ for all prime ideals
$\p\subset R$ containing~$r$.
 As the set of all elements $r\in R$ that become invertible in
$T^{-1}R$ can be described similarly, the assertion follows.
\end{proof}

 We refer to the dissertation~\cite[Section~3.2]{Co} for a general
discussion of Noetherian rings with small spectrum.

\begin{lem} \label{countable-spectrum-countable-multsubsets}
 Let $R$ be a commutative ring with (at most) countable spectrum,
and let $S\subset R$ be a multiplicative subset.
 Then there exists an (at most) countable multiplicative subset
$T\subset S$ such that $T^{-1}R=S^{-1}R$.
\end{lem}

\begin{proof}
 It suffices to choose, for every prime ideal $\p\subset R$ such that
$\p\cap S\ne\varnothing$, an element $t\in\p\cap S$.
 Set $T\subset R$ to be multiplicative subset generated by all
the elements $t\in R$ so obtained.
 Then the multiplicative subset $T$ is (at most) countable,
it is contained in $S$, and $T^{-1}R=S^{-1}R$
by Lemma~\ref{localization-determined-by-spectrum}.
\end{proof}

\begin{thm} \label{raynaud-gruson-multsubsets}
 For any Noetherian commutative ring $R$, the supremum of projective
dimensions of flat $R$\+modules is equal to the supremum of projective
dimensions of the $R$\+modules $S^{-1}R$ over all multiplicative
subsets $S\subset R$.
\end{thm}

\begin{proof}
 This is~\cite[Th\'eor\`eme~II.3.3.1]{RG} (see~\cite[\SS1]{Gru} for
a correction of a mistake in the exposition in~\cite{RG}).
\end{proof}

\begin{cor}
 Let $R$ be a Noetherian commutative ring with countable spectrum.
 Then the projective dimension of any flat $R$\+module does not
exceed\/~$1$.
\end{cor}

\begin{proof}
 According to Lemma~\ref{countable-spectrum-countable-multsubsets}
and~\cite[Lemma~1.9]{PMat}, the projective dimension of
the $R$\+module $S^{-1}R$ does not exceed~$1$ for any multiplicative
subset $S\subset R$.
 Thus it remains to apply Theorem~\ref{raynaud-gruson-multsubsets}.

 Alternatively, one can use
Theorem~\ref{countable-spectrum-thm} and
Lemma~\ref{almost-cotorsion-quite-flat-closedness-properties}(c).
\end{proof}

\begin{proof}[First proof of Theorem~\ref{countable-spectrum-thm}]
 Let us show that Theorem~\ref{countable-spectrum-thm} follows
from Main Lemma~\ref{noetherian-quite-flat-main-lemma}.
 Clearly, for any element $r\in R$, the quotient ring $R/rR$ is
a Noetherian ring with (at most) countable spectrum.
 Proceeding by Noetherian induction (cf.\ the proof of~\cite[Main
Theorem~1.3]{PS} in~\cite[Section~2]{PS}), we can assume that, for
any nonzero element $r\in R$, all flat $R/rR$\+modules are quite flat.

 Let $\q_1$,~\dots, $\q_k\subset R$ be the minimal prime ideals in~$R$.
 Set $S=R\setminus\bigcup_{i=1}^k\q_k\subset R$.
 Then $S$ is a multiplicative subset in $R$ and the ring $S^{-1}R$ is
Artinian, so all flat $S^{-1}R$\+modules are projective.
 By Lemma~\ref{countable-spectrum-countable-multsubsets}, there exists
a countable multiplicative subset $T\subset S$ such that
$T^{-1}R=S^{-1}R$.

 Let $F$ be a flat $R$\+module.
 Then the the $R/tR$\+module $F/tF$ is flat for all $t\in T$ and
the $T^{-1}R$\+module $T^{-1}F$ is flat.
 By the Noetherian induction assumption, it follows that
the $R/tR$\+module $F/tF$ is quite flat; and all flat
$T^{-1}R$\+modules are projective, hence also quite flat.
 Applying Main Lemma~\ref{noetherian-quite-flat-main-lemma} to
the Noetherian commutative ring $R$ with the countable multiplicative
subset $T\subset R$ and the flat $R$\+module $F$, we conclude that
the $R$\+module $F$ is quite flat.
\end{proof}

\begin{proof}[Second proof of Theorem~\ref{countable-spectrum-thm}]
Theorem~\ref{countable-spectrum-thm} follows immediately from
(is, in fact, an equivalent restatement of)
Corollary~\ref{countable-spectrum-almost-cotorsion-are-cotorsion}.
\end{proof}

\begin{proof}[Proof of
Corollaries~\ref{countable-spectrum-almost-cotorsion-are-cotorsion}
and~\ref{countable-spectrum-cotorsion-obtainable-cor}]
 It was essentially already explained in
Section~\ref{quite-flat-introd} that
Proposition~\ref{countable-spectrum-prop} implies
Corollaries~\ref{countable-spectrum-almost-cotorsion-are-cotorsion}
and~\ref{countable-spectrum-cotorsion-obtainable-cor}, but let us
repeat this explanation here adding a couple of references.
 The point is that all the $R$\+modules right $1$\+obtainable from
vector spaces over the residue fields of prime ideals in $R$ are
not only almost cotorsion but, in fact, cotorsion.
 Indeed, all vector spaces are obviously cotorsion as modules over
their respective fields, all restrictions of scalars take cotorsion
modules to cotorsion modules~\cite[Lemma~1.3.4(a)]{Pcosh}, and all
$R$\+modules right $1$\+obtainable from cotorsion $R$\+modules are
cotorsion (by Lemma~\ref{obtainable-orthogonal-lemma}).
\end{proof}

\begin{proof}[Proof of
Proposition~\ref{countable-spectrum-prop}]
 The ``if'' part has been already explained in the preceding proof.
 One deduces the ``only if'' from
Main Lemma~\ref{noetherian-almost-cotorsion-main-lemma} proceeding by
Noetherian induction on the ring~$R$.

 Assume that, for all nonzero elements $r\in R$, all almost cotorsion
$R/rR$\+modules are right $1$\+obtainable from vector spaces over
the residue fields of prime ideals of $R/rR$.
 As in the first proof of Theorem~\ref{countable-spectrum-thm} above,
one can find a countable multiplicative subset $T\subset R$ such that
$0\notin T$ and the ring of fractions $T^{-1}R$ is Artinian.
 Then all $T^{-1}R$\+modules are obtainable as finitely iterated
extensions of modules over the residue fields of the maximal ideals
of $T^{-1}R$.

 By Main Lemma~\ref{noetherian-almost-cotorsion-main-lemma}, all
almost cotorsion $R$\+modules are right $1$\+obtainable from
almost cotorsion $R/tR$\+modules, $t\in T$, and almost cotorsion
$T^{-1}R$\+modules.
 It remains to recall that all the residue fields of the rings
$R/tR$ and $T^{-1}R$ are, at the same time, residue fields of
the ring~$R$.
\end{proof}

\begin{proof}[Proof of
Main Lemmas~\ref{noetherian-quite-flat-main-lemma}
and~\ref{bounded-torsion-quite-flat-main-lemma}]
 The ``only if'' assertions hold for any multiplicative subset
$S$ in a commutative ring $R$ by
Lemma~\ref{almost-cotorsion-quite-flat-change-of-ring}(b).
 The ``if'' assertions of
Main Lemmas~\ref{noetherian-quite-flat-main-lemma}
and~\ref{bounded-torsion-quite-flat-main-lemma} are deduced
from the ``only if'' assertions of 
Main Lemmas~\ref{noetherian-almost-cotorsion-main-lemma}
and~\ref{bounded-torsion-almost-cotorsion-main-lemma}, respectively.

 Under any of the respective assumptions on $R$ and/or $S$, consider
a flat $R$\+module $F$ such that the $R/sR$\+module $F/sF$ is quite
flat for all $s\in S$ and the $S^{-1}R$\+module $S^{-1}F$ is quite flat.
 In order to show that the $R$\+module $F$ is quite flat, we will
check that $\Ext_R^1(F,C)=0$ for all almost cotorsion $R$\+modules~$C$.

 Denote by $\sE$ the class of all almost cotorsion $R/sR$\+modules,
$s\in S$, and all almost cotorsion $S^{-1}R$\+modules (viewed as
$R$\+modules via the restriction of scalars).
 By Lemmas~\ref{change-of-scalars-ext}
and~\ref{almost-cotorsion-quite-flat-closedness-properties}(c),
we have $\Ext^i_R(F,E)=0$ for all $R$\+modules $E\in\sE$ and all $i>0$.
 According to Main Lemma~\ref{noetherian-almost-cotorsion-main-lemma}
or~\ref{bounded-torsion-almost-cotorsion-main-lemma}, all almost
cotorsion $R$\+modules are right $1$\+obtainable from~$\sE$.
 By virtue of Lemma~\ref{obtainable-orthogonal-lemma}, it follows
it follows that $\Ext_R^i(F,C)=0$ for all almost cotorsion
$R$\+modules~$C$.
\end{proof}

 Main Lemmas~\ref{noetherian-almost-cotorsion-main-lemma}
and~\ref{bounded-torsion-almost-cotorsion-main-lemma} are deduced from
Corollaries~\ref{noetherian-cokernel-of-almost-cotorsion-separated}
and~\ref{bounded-torsion-cokernel-of-ac-separated} from
Section~\ref{separated-contramodules-subsecn}.
 Before proceeding with this proof, we will need one more lemma.

\begin{lem} \label{almost-cotorsion-separated-obtainable-as-projlim}
 Let $R$ be a commutative ring and $S\subset R$ be a countable
multiplicative subset.
 Then any $R$\+almost cotorsion $S$\+separated $S$\+complete
$R$\+module is an infinitely iterated extension, in the sense of
projective limit, of almost cotorsion $R/sR$\+modules, where
$s\in S$.
\end{lem}

\begin{proof}
 We will use the notation from
Section~\ref{countable-multsubsets-subsecn}.
 Let $C$ be an $R$\+almost cotorsion $S$\+separated $S$\+complete
$R$\+module.
 Then we have $C=\varprojlim_{s\in S}C/sC=\varprojlim_{n\ge1}C/t_nC$.
 Furthermore, we have exact sequences
$$
 C/s_nC\overset{t_{n-1}}\lrarrow C/t_nC\lrarrow C/t_{n-1}C\lrarrow0,
$$
showing that the kernels of the projection maps
$C/t_nC\rarrow C/t_{n-1}C$ are quotient modules of the $R$\+modules
$C/sC$, \ $s\in S$.
 Finally, it remains to recall that all quotient modules of an almost
cotorsion $R$\+module $C$ are almost cotorsion $R$\+modules by
Lemma~\ref{almost-cotorsion-quite-flat-closedness-properties}(a),
and all $R/sR$\+modules that are almost cotorsion $R$\+modules are
also almost cotorsion $R/sR$\+modules by
Lemma~\ref{almost-cotorsion-reflected-lem}.
\end{proof}

\begin{proof}[Proof of
Main Lemmas~\ref{noetherian-almost-cotorsion-main-lemma}
and~\ref{bounded-torsion-almost-cotorsion-main-lemma}]
 The ``if'' assertions hold for any multiplicative subset $S$ in
a commutative ring~$R$.
 Indeed, all almost cotorsion $R/sR$\+modules and almost cotorsion
$S^{-1}R$\+modules are almost cotorsion $R$\+modules by
Lemma~\ref{almost-cotorsion-quite-flat-change-of-ring}(a),
and all $R$\+modules right $1$\+obtainable from almost cotorsion
$R$\+modules are almost cotorsion by
Lemmas~\ref{obtainable-orthogonal-lemma}
and~\ref{almost-cotorsion-quite-flat-closedness-properties}(c).
 The nontrivial part is the ``only if''.

 Denote by $\sE\subset R\modl$ the class of all almost cotorsion
$R/sR$\+modules, $s\in S$, and all almost cotorsion $S^{-1}R$\+modules
(viewed as $R$\+modules via the restriction of scalars).
 Let $C$ be an almost cotorsion $R$\+module; we need to show that
$C$ is right $1$\+obtainable from~$\sE$.
 Our argument is based on the exact
sequence~\eqref{weakly-cotorsion-sequence} from
Section~\ref{one-multsubset-secn} (which is applicable because
an almost cotorsion $R$\+module $C$ is $S$\+weakly cotorsion for
a countable multiplicative subset $S\subset R$, by the definition).

 In order to prove that $C$ is right $1$\+obtainable from $\sE$, it
suffices to check that the $R$\+modules $\Hom_R(S^{-1}R,C)$ and
$\Delta_{R,S}(C)$ are right $1$\+obtainable from $\sE$, while
the $R$\+module $\Hom_R(S^{-1}R/R,\.C)$ is right $2$\+obtainable
from~$\sE$.
 Let us consider these three $R$\+modules one by one.

 The $R$\+module $\Hom_R(S^{-1}R,C)$ is almost cotorsion by
Lemma~\ref{hom-is-weakly-cotorsion}, hence it is an almost cotorsion
$S^{-1}R$\+module by Lemma~\ref{almost-cotorsion-reflected-lem}.
 So the $R$\+module $\Hom_R(S^{-1}R,C)$ already belongs to~$\sE$.

 The $R$\+module $\Delta_{R,S}(C)$ is almost cotorsion as a quotient
module of an almost cotorsion $R$\+module~$C$ (by
Lemma~\ref{almost-cotorsion-quite-flat-closedness-properties}(a)).
 Besides, it is an $S$\+contramodule $R$\+module
(by Lemma~\ref{S-contramodule-category-lem}(b)).
 Applying
Corollary~\ref{noetherian-cokernel-of-almost-cotorsion-separated}
or~\ref{bounded-torsion-cokernel-of-ac-separated},
we can present $\Delta_{R,S}(C)$ as the cokernel of an injective
morphism of $R$\+almost cotorsion $S$\+separated $S$\+complete
$R$\+modules $K\rarrow L$.
 By Lemma~\ref{almost-cotorsion-separated-obtainable-as-projlim},
both $K$ and $L$ are obtainable as infinitely iterated extensions
of almost cotorsion $R/sR$\+modules, $s\in S$.
 Thus the $R$\+module $\Delta_{R,S}(C)$ is simply right obtainable
from~$\sE$.

 The $R$\+module $\Hom_R(S^{-1}R/R,\.C)$ is an $S$\+contramodule by
Lemma~\ref{hom-from-torsion-is-contramodule} (and in fact even
an $S$\+separated $S$\+complete $R$\+module, as a projective limit
of $S$\+separated $S$\+complete $R$\+modules).
 So it is simply right obtainable from $R/sR$\+modules,
$s\in S$, by Lemma~\ref{countable-S-contramodules-obtainable}
(Lemma~\ref{projlim-obtainable-lem}(a) is sufficient).
 By Lemma~\ref{2-obtainable-lem}, all $R/sR$\+modules are right
$2$\+obtainable from almost cotorsion $R/sR$\+modules.
 Hence the $R$\+module $\Hom_R(S^{-1}R/R,\.C)$ is right
$2$\+obtainable from~$\sE$.
\end{proof}

\begin{rem}
 Notice that, for a Noetherian commutative ring of finite Krull
dimension with countable spectrum,
Theorem~\ref{finite-dimensional-ring-thm} provides a result stronger
than Theorem~\ref{countable-spectrum-thm}.
 Indeed, Theorem~\ref{countable-spectrum-thm} uses the localizations
with respect to all countable multiplicative subsets $S\subset R$ in
order to generate the flat cotorsion theory in $R\modl$, while
Theorem~\ref{finite-dimensional-ring-thm} offers a finite
collection of such multiplicative subsets.
 Nevertheless, these are two different theorems with completely
different proofs (one of them presented in
Section~\ref{finite-dimensional-secn} and the other one in this
Section~\ref{quite-flat-almost-cotorsion-secn}).
 Moreover, we do \emph{not} know how to extend our proof
of Theorem~\ref{finite-dimensional-ring-thm} to the case of Noetherian
rings of infinite Krull dimension, as it is based on
Theorem~\ref{several-multsubsets-thm}, which only works for finite
collections of multiplicative subsets.

 Still, how many multiplicative subsets are actually needed in
Theorem~\ref{countable-spectrum-thm}\,?
 The answer is: a countable number.
 For any Noetherian commutative ring $R$ with countable spectrum,
there exists a countable collection of countable multiplicative subsets
$S_1$, $S_2$, $S_3$,~\dots~$\subset R$ such that all $R$\+modules
$C$ for which $\Ext^1_R(S_j^{-1}R,C)=0$ for all $j=1$, $2$, $3$,~\dots\
are cotorsion.

 To prove this assertion, one only needs to follow the Noetherian
induction argument in the first proof of
Theorem~\ref{countable-spectrum-thm} or in the proof of
Proposition~\ref{countable-spectrum-prop} above.
 Starting from a Noetherian commutative ring $R$ with countable
spectrum, it produces a countable multiplicative subset $T\subset R$
such that the ring $T^{-1}R$ is Artinian.
 Then it passes from the ring $R$ to the countable collection
of rings $R/tR$, \,$t\in T$, and repeats the procedure with each
of them.
 What one obtains in this way is a rooted tree in which the degree of
any vertex is at most countable and (due to Noetherianity) there is
no infinite branch.
 By a countable version of the K\"onig Lemma, such a tree is countable.
 In each of its vertices, there sits a quotient ring of the original
ring $R$ by a certain ideal and a countable multiplicative subset in
this quotient ring.
 All one needs to do is to lift these multiplicative subsets in
the quotient rings to countable multiplicative subsets in $R$, using
the procedure from the proof of
Lemma~\ref{almost-cotorsion-reflected-lem}.
 This produces the desired collection of multiplicative subsets
$S_j\subset R$.

 One needs to follow the argument all the way down to the proof of
Main Lemma~\ref{noetherian-almost-cotorsion-main-lemma} in order to
convince oneself that all the $R$\+modules that are $S_j$\+weakly
cotorsion for all $j\ge1$ are right $1$\+obtainable from vector
spaces over the residue fields of the spectrum points of~$R$.
 Then it follows that all such $R$\+modules are cotorsion.
\end{rem}

\end{document}